\documentclass[11pt]{article}
\usepackage{amssymb,amsmath,color,amsfonts,float,amscd,amsthm,bm,mathrsfs} 

\allowdisplaybreaks[4]

\catcode`!=11
\let\!int\int \def\int{\displaystyle\!int}
\let\!lim\lim \def\lim{\displaystyle\!lim}
\let\!sum\sum \def\sum{\displaystyle\!sum}
\let\!sup\sup \def\sup{\displaystyle\!sup}
\let\!inf\inf \def\inf{\displaystyle\!inf}
\let\!cap\cap \def\cap{\displaystyle\!cap}
\let\!max\max \def\max{\displaystyle\!max}
\let\!min\min \def\min{\displaystyle\!min}

\let\oldsection\section
\renewcommand\section{\setcounter{equation}{0}\oldsection}

\newtheorem{theorem}{Theorem}[section]
\newtheorem{lemma}[theorem]{Lemma}
\newtheorem{corollary}[theorem]{Corollary}
\newtheorem{proposition}[theorem]{Proposition}

\theoremstyle{definition}

\theoremstyle{remark}
\newtheorem{remark}[theorem]{Remark}

\begin{document}

\title{Prandtl boundary layer expansion with strong boundary layers for inhomogeneous incompressible magnetohydrodynamics equations in Sobolev framework}

\author{Shengxin Li\footnote{Department of Applied Mathematics, The Hong Kong Polytechnic University, Hung Hom, Hong Kong; School of Mathematical Sciences, and CMA-Shanghai, Shanghai Jiao Tong University, Shanghai 200240, P. R. China. Email: lishengxin@sjtu.edu.cn}
\qquad\quad
Feng Xie
\footnote{School of Mathematical Sciences, and CMA-Shanghai, Shanghai Jiao Tong University, Shanghai 200240, P. R. China.
Email: tzxief@sjtu.edu.cn}
}
\date{}
\maketitle

\begin{abstract}
We consider the validity of Prandtl boundary layer expansion of solutions to the initial boundary value problem for inhomogeneous incompressible magnetohydrodynamics equations in the half plane when both viscosity and resistivity coefficients tend to zero, where the no-slip boundary condition is imposed on velocity while the perfectly conducting condition is given on magnetic field. Since there exist strong boundary layers, the essential difficulty in establishing the uniform $L^\infty$ estimates of the error functions comes from the unboundedness of vorticity of strong boundary layers. Under the assumptions that the viscosity and resistivity coefficients take the same order of a small parameter and the initial tangential component of magnetic field has a positive lower bound near the boundary, we prove the validity of Prandtl boundary layer ansatz in $L^\infty$ sense in Sobolev framework. Compared with the homogeneous incompressible case considered in \cite{LXY192}, there exists a strong boundary layer of density. Consequently, some suitable functionals should be designed and the elaborated co-normal energy estimates will be involved in analysis due to the variation of density and the interaction between the density and velocity.
\end{abstract}

%
%
\maketitle





\section{Introduction and main results}
In this paper we are concerned with the following inhomogeneous incompressible magnetohydrodynamics (MHD) equations in a domain of $\Omega=\{(x, y): x\in\mathbb{T}, y\in\mathbb{R}_+\}$.
\begin{align}\label{1.1}
\begin{cases}
\partial_t\rho^\varepsilon+{\bf{u}}^\varepsilon\cdot\nabla\rho^\varepsilon=0,\\
\rho^\varepsilon\left(\partial_t{\bf{u}}^\varepsilon+{\bf{u}}^\varepsilon\cdot\nabla{\bf{u}}^\varepsilon\right)+\nabla p^\varepsilon-{\bf{H}}^\varepsilon\cdot\nabla {\bf{H}}^\varepsilon-\mu\varepsilon\Delta {\bf{u}}^\varepsilon=0,\\
\partial_t{\bf{H}}^\varepsilon+{\bf{u}}^\varepsilon\cdot\nabla {\bf{H}}^\varepsilon-{\bf{H}}^\varepsilon\cdot\nabla {\bf{u}}^\varepsilon-\kappa\varepsilon\Delta{\bf{H}}^\varepsilon=0,\\
\nabla\cdot{\bf{u}}^\varepsilon=0,\qquad \nabla\cdot{\bf{H}}^\varepsilon=0.
\end{cases}
\end{align}
Here, ${\bf{u}}^\varepsilon=(u^\varepsilon, v^\varepsilon)$ and ${\bf{H}}^\varepsilon=(h^\varepsilon, g^\varepsilon)$ stand for velocity vector and magnetic field respectively, $\rho^\varepsilon$ denotes density while $p^\varepsilon$ is pressure. $\mu\varepsilon$ and $\kappa\varepsilon$ are the viscosity and resistivity coefficients with $\mu$ and $\kappa$ being positive constants. The initial data of \eqref{1.1} is given by
\begin{align}\label{1.2}
(\rho^\varepsilon, {\bf{u}}^\varepsilon, {\bf{H}}^\varepsilon)|_{t=0}=(\rho_0, {\bf{u}}_0, {\bf{H}}_0)(x, y)=(\rho_0, u_0, v_0, h_0, g_0)(x, y).
\end{align}
We assume that $\rho_0(x, y)$ is bounded and away from vacuum.
\begin{align*}
0<\underline{\rho}\le\rho_0(x, y)\le\overline{\rho}<+\infty,
\end{align*}
where $\underline{\rho}$ and $\overline{\rho}$ are two constants. And we impose the no-slip boundary condition on the velocity field while the perfectly conducting boundary condition is imposed on the magnetic field.
\begin{align}\label{1.3}
{\bf{u}}^\varepsilon|_{y=0}={\bf{0}},\qquad (\partial_yh^\varepsilon, g^\varepsilon)|_{y=0}=(0, 0).
\end{align}
The goal of this paper is to study the asymptotic behavior of solutions $(\rho^\varepsilon, {\bf{u}}^\varepsilon, {\bf{H}}^\varepsilon)$ to the initial boundary value problem \eqref{1.1}-\eqref{1.3} as $\varepsilon\to0$. Formally, when $\varepsilon=0$, \eqref{1.1} becomes the following inhomogeneous incompressible ideal MHD equations.
\begin{align}\label{1.4}
\begin{cases}
\partial_t\rho^0+{\bf{u}}^0\cdot\nabla\rho^0=0,\\
\rho^0\left(\partial_t{\bf{u}}^0+{\bf{u}}^0\cdot\nabla{\bf{u}}^0\right)+\nabla p^0-{\bf{H}}^0\cdot\nabla {\bf{H}}^0=0,\\
\partial_t{\bf{H}}^0+{\bf{u}}^0\cdot\nabla {\bf{H}}^0-{\bf{H}}^0\cdot\nabla {\bf{u}}^0=0,\\
\nabla\cdot{\bf{u}}^0=0,\qquad \nabla\cdot{\bf{H}}^0=0,
\end{cases}
\end{align}
where $\rho^0$ stands for density, ${\bf{u}}^0=(u^0, v^0)$ velocity field and ${\bf{H}}^0=(h^0, g^0)$ magnetic field. For this system, besides we take the same initial data \eqref{1.2}, and it is sufficient to impose the Dirichlet boundary conditions on normal components of velocity and magnetic field for well-posedness of solutions.
\begin{align}\label{1.5}
(v^0, g^0)|_{y=0}={\bf{0}}.
\end{align}

From the boundary conditions \eqref{1.3} and \eqref{1.5}, it is obvious that there is a mismatch of values between the tangential components $(u^\varepsilon, h^\varepsilon)$ and $(u^0, h^0)$ on the boundary, which is related to the so-called Prandtl boundary layer theory proposed by Prandtl in \cite{P04} in 1904. The Prandtl boundary layer ansatz implies that there is a thin layer of order $\sqrt{\varepsilon}$ near the boundary and there exist three boundary layer correctors of $(\rho_b^0, u_b^0,  h_b^0)\left(t, x, \frac{y}{\sqrt{\varepsilon}}\right)$, such that the solutions to \eqref{1.1} take the following expansion.
\begin{align}\label{1.6}
\begin{cases}
(\rho^\varepsilon, {\bf{u}}^\varepsilon, {\bf{H}}^\varepsilon)(t, x, y)=(\rho^0, {\bf{u}}^0, {\bf{H}}^0)(t, x, y)+(\rho_b^0, u_b^0, 0, h_b^0,
0)\left(t, x, \frac{y}{\sqrt{\varepsilon}}\right)+o(1),\\
p^\varepsilon(t, x, y)=p^0(t, x, y)+o(1),
\end{cases}
\end{align}
where $o(1)$ is supposed to tend to zero as $\varepsilon\to0$ in some topology.

Before proceeding, let us review some related works. The system of equations \eqref{1.1} has been widely studied and well understood for fixed $\varepsilon$, one can refer \cite{AP08, CTW11, DL98, GL97, G14, HW13, XQF22, ZY17} and references therein. However, there are very few results about the vanishing viscosity limit for \eqref{1.1}, which is closely related to the Prandtl boundary layer expansion.
As is known that Prandtl first derived the boundary layer equations \cite{P04} for the incompressible Navier-Stokes equations with no-slip boundary condition on velocity, which is now called as Prandtl equations.
Under the monotonicity condition on the tangential velocity in the vertical direction, Oleinik obtained the local existence of solutions to 2D Prandtl equations by using the Crocoo transformation in 1960s in \cite{O63}. One also refer to the classical book \cite{OS99} for this result and some other related progress in this field. Recently, two groups \cite{AWXY15} and \cite{MW15} respectively reproved this well-posedness result by energy method in Sobolev spaces. Xin and Zhang \cite{XZ04} obtained a global in time existence of weak solutions by imposing an additional favorable condition on pressure. The above results were extended to three dimensional case in \cite{LWY16} and \cite{LWY17}. When the monotonicity condition was violated, the boundary separation phenomenon can be observed and the ill-posedness of the Prandtl equations in Sobolev spaces was thus proved, one can refer to \cite{EE97, GD10, GN12, LWY162, LY17} and the reference therein for details.

Back to MHD equations, the MHD boundary layer equations were derived in  \cite{GP17,LXY19}. Under the assumption that the tangential component of magnetic field has a positive lower bound near the boundary, Liu, the second author and Yang established the well-posedness result of MHD boundary layer equations without the monotonicity condition on velocity by observing a key cancellation mechanism in \cite{LXY19}.
Under the same assumption, the well-posedness result was also obtained for the MHD boundary layer equations without magnetic diffusion
in \cite{LWXY20}. 
Li and Xu obtain the well-posedness result for the system without viscosity in \cite{LX21}. Moreover, the well-posedness of solutions was established for the inhomogeneous incompressible MHD boundary layer system \cite{GHY21} and for full compressible MHD boundary layer system in \cite{HLY19}.
All of these results depend on the structure assumption that the tangential component of magnetic field has a positive lower bound.
Except Sobolev framework, the well-posedness result also can be established in analytic spaces \cite{LiX21, LZ21, XY19} and Gevrey spaces \cite{LY21} without this structure assumption.

As mentioned above, we are devoted to verifying the inviscid limit of solutions to \eqref{1.1}-\eqref{1.4}, which is a fundamental but challenging problem in both mathematics and hydrodynamics. There are very few works of the rigorous verification of the Prandtl boundary layer expansion under no-slip boundary condition on velocity in Sobolev space. The convergence of Prandtl boundary layer expansion for the incompressible Navier-Stokes equations was obtained by Sammartino and Caflish  in analytic spaces in \cite{SC, SC2}. This result was reproved by the energy method in \cite{WWZ17}. Maekawa \cite{M14} proved that the inviscid limit provided that the support of initial vorticity of Euler outer flow is away from the boundary, and this result was improved in \cite{KVW20} when the initial data is analytic only close to the boundary.
There are also some works about the Prandtl expansion for the steady flows \cite{GZ, GM19, GI18, GN17, IM20, IM21} in Sobolev spaces and for incompressible Navier-Stokes system in Gevrey spaces \cite{GMM18}.

For plasma, the verification of the Prandtl ansatz for incompressible MHD system was obtained by Liu, the second author and Yang \cite{LXY192} in Sobolev spaces under the assumption that the tangential component of magnetic field does not degenerate near the boundary initially. One can refer to \cite{DLX, LYZ} for the case of steady MHD.
Very recently, the uniform energy estimates have been established for incompressible MHD equations \cite{LXY21} and compressible MHD equations \cite{CLX, CLX2} under the influence of the transverse magnetic field, and thus the inviscid limit was achieved by combining the uniform energy estimates and compact arguments.

Now it is position to state our main result.
\begin{theorem}\label{T1.1}
Let $m\ge 72$ be an integer. The initial data $({\bf{u}}_0, {\bf{H}}_0)$ satisfies the compatibility conditions for both \eqref{1.1}-\eqref{1.3} and \eqref{1.4}-\eqref{1.5} up to order $36$, and the divergence free conditions of $\nabla\cdot{\bf{u}}_0=0$ and $\nabla\cdot{\bf{H}}_0=0$ are satisfied by the initial data $({\bf{u}}_0, {\bf{H}}_0)$. Moreover, assume that the initial tangential magnetic field has a positive lower bound on the boundary. That is,
\begin{align*}
h_0(x, 0)\ge \sigma>0
\end{align*}
for some positive constant $\sigma.$ Then there exists $T_*>0$ which is independent of $\varepsilon$, such that for any smooth solutions $(\rho^\varepsilon, {\bf{u}}^\varepsilon, {\bf{H}}^\varepsilon)$ to \eqref{1.1}-\eqref{1.3} and solutions $(\rho^0, {\bf{u}}^0, {\bf{H}}^0)\in C([0, T_*], H^m)$ to \eqref{1.4}-\eqref{1.5}, and a boundary layer profile $(\rho_b^0, u_b^0, v_b^0, h_b^0, g_b^0)\in C([0, T_*], \mathbb{T}\times\mathbb{R}_+)$ which is a solution to \eqref{boun}, such that
\begin{align}\label{1.7}
\sup\limits_{0\le t\le T_*}\Big\|&({\bf{u}}^\varepsilon, {\bf{H}}^\varepsilon)(t, x, y)-({\bf{u}}^0, {\bf{H}}^0)(t, x, y)\notag\\
&-(u_b^0, \sqrt{\varepsilon}v_b^0, h_b^0, \sqrt{\varepsilon}g_b^0)\left(t, x, \frac y{\sqrt{\varepsilon}}\right)\Big\|_{L_{xy}^\infty}\le C\varepsilon^{\frac12}.
\end{align}
\end{theorem}
\begin{remark}
We do not state the convergence  result from density $\rho^{\varepsilon}$ to $\rho^0$ in Theorem \ref{T1.1},
because we do not derive the estimates of the second order derivatives for the error function of density in this paper, which are essential in obtaining the convergence of density in $L^\infty$ topology.
In fact, by the similar arguments as those in Section $\ref{Sec3.7}$, we can also establish such uniform estimates of the second order derivatives for the error function of $\tilde{\rho}$. However, \eqref{1.7} is enough for our purpose to justify the Prandtl boundary layer ansatz. Consequently, we omit the cumbersome calculations about the uniform estimates of density.
\end{remark}
In what follows, we briefly present the main difficulties and the key ingredients in the proof of Theorem \ref{T1.1}.

(a) To prove Theorem \ref{T1.1}, the crucial step is to establish the uniform energy estimates of the error functions. Similar as the homogeneous incompressible case \cite{LXY192}, the main difficulty comes from the stretching terms $v\partial_y\rho^a$, $v\partial_yu^a-g\partial_yh^a$ and $v\partial_yh^a-g\partial_yu^a$, which behave like $O(\varepsilon^{-\frac12})(v, g)$. As in \cite{LXY192}, we can cancel these singular terms
by introducing the new unknowns $(\tilde{\rho}, \tilde{\bf{u}}, \tilde{\bf{H}})$ which are the combinations of original unknowns $(\rho, \bf{u}, \bf{H})$ and the stream function $\psi$ of magnetic field $\bf{H}$.

(b) For inhomogeneous incompressible MHD case, the variation of density is involved, and there exists a strong boundary layer of density. Consequently, it is necessary to estimate $(\partial_x\tilde{\rho}, \partial_y\tilde{\rho})$ when performing the first order co-normal derivative estimate of ${\bf{U}}$ due to the interaction between density and  velocity. In general,
it is hard to control $\|\partial_xv^a\cdot\partial_y\tilde{\rho}\|_{L^2}$ since the $L^2$ norm of $\partial_y\tilde{\rho}$ is of order $\varepsilon^{-1}$, thus we turn to establish the estimate of $\|y\partial_y\tilde{\rho}\|_{L^2}$ thanks to the H\"ardy's trick and the divergence free condition.

(c) Since $\rho^\varepsilon$ satisfies a transport equation, we can use $\|\nabla{\bf{u}}^\varepsilon\|_{L_{xy}^\infty}$ to control $\|\nabla\rho^\varepsilon\|_{L_{xy}^\infty}$, see Lemma $\ref{L3.6}$. In this way, we can avoid to estimate the second order derivatives of $\tilde{\rho}$ whose proof is very sophisticated. However, as mentioned above, $\partial_yu^a$ is of order $O(\varepsilon^{-\frac12})$ and it is hard to control. Another key observation is that $\partial_y\rho^\varepsilon$ always appears together with the factor of $a^pg^\varepsilon$, which allows us to estimate the conormal derivative $y\partial_y$ to replace the estimate of normal derivative $\partial_y$ by H\"ardy's trick, and we point out that $y\partial_yu^a$ is indeed of order $O(1)$.

(d) To obtain the final $L_{txy}^\infty$ norm of the error functions, we need to derive the estimates of $\partial_{tx}\tilde{{\bf{u}}}$, $\partial_{xx}\tilde{{\bf{u}}}$ and $\partial_{tt}\tilde{{\bf{u}}}$. To this end, the most involved term is of the form $\partial_{tt}\rho^\varepsilon\partial_t\tilde{{\bf{u}}}$ because there is no diffusion term in the equation of $\tilde{\rho}$ and we also do not want to touch the second order derivative estimates of $\tilde{\rho}$.
Consequently, we have to transform $\partial_{tt}\rho^\varepsilon$ into $\partial_{xy}\rho^\varepsilon$, $\partial_{yy}\rho^\varepsilon$ and other lower order terms by the first equation in \eqref{1.1}. By integration by parts, we also can avoid to estimate $\partial_{xy}\rho^\varepsilon$ and $\partial_{yy}\rho^\varepsilon$. Thus, one need to bound the second order normal derivatives of $\tilde{{\bf{u}}}$. Unfortunately, due to the fast variable $\eta=\frac{y}{\sqrt{\varepsilon}}$ and the appearance of the boundary term because of the mixed boundary condition, it is still hard to obtain the estimates of $(\partial_{yy}\tilde{u}, \partial_{yy}\tilde{h})$ by energy method. The key point is that we get the estimates of $(\partial_{yy}\tilde{u}, \partial_{yy}\tilde{h})$ directly from the equation \eqref{3.9} on the cost of $\varepsilon^{-1}$, see Lemma $\ref{L3.4}$.
Moreover, it is also needed to control the $L_t^2(L_{xy}^2)-$norms of $\partial_{yyy}\tilde{h}$ due to the integration by parts which behaves like $O(\varepsilon^{-\frac52})$. To overcome this issue, we also need to establish the estimates for the first order conormal derivative of $y\partial_y{\bf{U}}$. In this way, half order of $\varepsilon$ can be reduced, then the whole energy estimate process is complete.

This paper is organized as follows.  In Section 2, we construct a suitable approximate solution and introduce some useful properties about it. The uniform estimates of error functions are established in Section 3, and the proof of the main theorem is presented in Section 4. Finally, we will provide some proofs and sophisticated expressions in Appendix.

Throughout this paper, we use the notation $A\lesssim B$ to stand for $A\leq C B$ for some generic constant $C>0$ independent of $\varepsilon$. And we denote the polynomial functions by $\mathcal{P}(\cdot)$, which may vary from line to line,  and the commutator is expressed by $[\cdot, \cdot]$. We also use $\partial_y^{-1}f$ to denote the integral of $-\int_y^\infty f(t, x, \tilde{y})\;d\tilde{y}$.

\section{Construction of the approximate solution}
In this section, we are devoted to constructing an approximate solution to \eqref{1.1} which has the following expansion.
\begin{equation}\label{2.1}
\begin{cases}
\rho^a(t, x, y)=\rho^0(t, x, y)+\rho_b^0\left(t, x, \frac{y}{\sqrt{\varepsilon}}\right)+\sqrt{\varepsilon}\left[\rho^1(t, x, y)+\rho_b^1\left(t, x, \frac{y}{\sqrt{\varepsilon}}\right)\right]\\
\qquad\qquad\qquad\qquad\quad+\varepsilon \left[\rho^2(t, x, y)+\rho_b^2\left(t, x, \frac{y}{\sqrt{\varepsilon}}\right)\right]\\
(u^a, h^a)(t, x, y)=(u^0, h^0)(t, x, y)+(u_b^0, h_b^0)\left(t, x, \frac{y}{\sqrt{\varepsilon}}\right)\\
\qquad\qquad\qquad\qquad\quad+\sqrt{\varepsilon}\left[(u^1, h^1)(t, x, y)+(u_b^1, h_b^1)\left(t, x, \frac{y}{\sqrt{\varepsilon}}\right)\right]\\
\qquad\qquad\qquad\qquad\quad+\varepsilon\left[(u^2, h^2)(t, x, y)+(u_b^2, h_b^2)\left(t, x, \frac{y}{\sqrt{\varepsilon}}\right)\right]\\
(v^a, g^a)(t, x, y)=(v^0, g^0)(t, x, y)+\sqrt{\varepsilon}\left[(v_b^0, g_b^0)\left(t, x, \frac{y}{\sqrt{\varepsilon}}\right)+(v^1, g^1)(t, x, y)\right]\\
\qquad\qquad\qquad\qquad\quad+\varepsilon\left[(v_b^1, g_b^1)\left(t, x, \frac{y}{\sqrt{\varepsilon}}\right)+(v^2, g^2)(t, x, y)\right]\\
\qquad\qquad\qquad\qquad\quad+\varepsilon^{\frac32}(v_b^2, g_b^2)\left(t, x, \frac{y}{\sqrt{\varepsilon}}\right)\\
p^a(t, x, y)=p^0(t, x, y)+\sqrt{\varepsilon}p^1(t, x, y)+\varepsilon\left[ p_b^1\left(t, x, \frac{y}{\sqrt{\varepsilon}}\right)+p^2(t, x, y)\right]+\varepsilon^{\frac32}  p_b^2\left(t, x, \frac{y}{\sqrt{\varepsilon}}\right).
\end{cases}
\end{equation}
Since all of boundary layer profiles decay to zero away from the boundary, it is assumed that
\begin{align*}
\lim\limits_{\eta\to +\infty}(\rho_b^i, u_b^i, v_b^i, h_b^i, g_b^i, p_b^i)(t, x, \eta)={\bf{0}},\quad i=0, 1, 2
\end{align*}
with $\eta=\frac{y}{\sqrt{\varepsilon}}$ being the fast variable. And the matching boundary conditions are listed as follows.
\begin{align}\label{2.2}
&u_b^i(t, x, 0)=-u^i(t, x, 0),\quad i=0, 1, 2,\notag\\
&\partial_\eta h_b^0|_{\eta=0}=0,\quad \partial_\eta h_b^1|_{\eta=0}=-\partial_yh^0(t, x, 0),\quad \partial_\eta h_b^2|_{\eta=0}=-\partial_yh^1(t, x, 0)\notag,\\
&(v^1, g^1)(t, x, 0)=-(v_b^0, g_b^0)(t, x, 0),\quad (v^2, g^2)(t, x, 0)=-(v_b^1, g_b^1)(t, x, 0).
\end{align}
Finally, we choose the zero Dirichlet boundary condition for $(v_b^2, g_b^2)$, i.e.
\begin{align*}
(v_b^2, g_b^2)|_{\eta=0}=(0, 0).
\end{align*}
In what follows, we use the notation $\bar{f}(t, x)$ to denote the trace of the function $f(t, x, y)$ on the boundary $\{y=0\}$. To simplify the representation, we introduce the following notations to denote the coefficients of the Taylor expansion of the Euler flows with respect to $y$ variable on $\{y=0\}$.
\begin{align*}
&\varrho^1=\eta\overline{\partial_y\rho^0}+\overline{\rho^1}, \quad \varrho^2=\frac12\eta^2\overline{\partial_y^2\rho^0}+\eta\overline{\partial_y\rho^1}+\overline{\rho^2},\\
&\mathcal{U}^1=\eta\overline{\partial_yu^0}+\overline{u^1},
\quad \mathcal{U}^2=\frac12\eta^2\overline{\partial_y^2u^0}+\eta\overline{\partial_yu^1}+\overline{u^2},\\
&\mathcal{H}^1=\eta\overline{\partial_yh^0}+\overline{h^1},\quad
\mathcal{H}^2=\frac12\eta^2\overline{\partial_y^2h^0}+\eta\overline{\partial_yh^1}+\overline{h^2},\\
&\mathcal{V}^1=\eta\overline{\partial_yv^0}-\overline{v_b^0},\quad
\mathcal{V}^2=\frac12\eta^2\overline{\partial_y^2v^0}+\eta\overline{\partial_yv^1}-\overline{v_b^1},\quad
\mathcal{V}^3=\frac16\eta^3\overline{\partial_y^3v^0}+\frac12\eta^2\overline{\partial_y^2v^1}+\eta\overline{\partial_yv^2},\\
&\mathcal{G}^1=\eta\overline{\partial_yg^0}-\overline{g_b^0},\quad
\mathcal{G}^2=\frac12\eta^2\overline{\partial_y^2g^0}+\eta\overline{\partial_yg^1}-\overline{g_b^1},\quad
\mathcal{G}^3=\frac16\eta^3\overline{\partial_y^3g^0}+\frac12\eta^2\overline{\partial_y^2g^1}+\eta\overline{\partial_yg^2}.
\end{align*}
Here, the matching boundary conditions \eqref{2.2} for $\mathcal{V}^1$, $\mathcal{G}^1$ and $\mathcal{V}^2$, $\mathcal{G}^2$ have been used.
Furthermore, we also use the following notations to denote the derivatives of the above symbols.
\begin{align*}
\partial_i\mathcal{F}^1=\eta\overline{\partial_{iy}f^0}+\overline{\partial_if^1},\quad
\partial_i\mathcal{F}^2=\frac12\eta^2\overline{\partial_{iyy}h^0}+\eta\overline{\partial_{iy}f^1}+\overline{\partial_if^2},
\end{align*}
where $i=t, x, y$.

\subsection{Zeroth order inner flow}
Putting ansatz \eqref{2.1} into \eqref{1.1}, setting the $\varepsilon^0$-th order terms equal to zero and letting $\eta\to +\infty$,
we deduce that the leading order inner flow $(\rho^0, {\bf{u}}^0, {\bf{H}}^0, p^0)$ satisfies
\begin{equation}\label{2.3}
\begin{cases}
\partial_t\rho^0+{\bf{u}}^0\cdot\nabla\rho^0=0,\\
\rho^0\left(\partial_t{\bf{u}}^0+{\bf{u}}^0\cdot\nabla{\bf{u}}^0\right)+\nabla p^0-{\bf{H}}^0\cdot\nabla {\bf{H}}^0=0,\\
\partial_t{\bf{H}}^0+{\bf{u}}^0\cdot\nabla {\bf{H}}^0-{\bf{H}}^0\cdot\nabla {\bf{u}}^0=0,\\
\nabla\cdot{\bf{u}}^0=0,\qquad \nabla\cdot{\bf{H}}^0=0,\\
(v^0, g^0)|_{y=0}=0,\quad (\rho^0, {\bf{u}}^0, {\bf{H}}^0)|_{t=0}=(\rho_0, {\bf{u}}_0, {\bf{H}}_0)(x, y).
\end{cases}
\end{equation}
The well-posedness of solutions to the initial boundary value problem \eqref{2.3} can be found in \cite{S93} and is stated as follows.
\begin{proposition}
\label{E1}
Let $m\ge72$ be an integer, the initial data satisfies that $(\rho_0-1, {\bf{u}}_0, {\bf{H}}_0)\in H^m(\mathbb{T}\times\mathbb{R}_+)$ and the divergence free conditions of $\nabla\cdot{\bf{u}}_0=\nabla\cdot{\bf{H}}_0=0$, then there exists a time $T_1>0$, such that \eqref{2.3} admits a solution $(\rho^0, {\bf{u}}^0, {\bf{H}}^0, p^0)$. Moreover,
\begin{align*}
(\rho^0-1, {\bf{u}}^0, {\bf{H}}^0, \nabla p^0)\in\cap_{j=0}^m C^j\left([0, T_1]; H^{m-j}(\mathbb{T}\times\mathbb{R}_+)\right).
\end{align*}
\end{proposition}
\begin{remark}
The constant $1$ is not essential which can be replaced by any positive constant. And we also require that the initial density is bounded and away from vacuum.  For the initial magnetic field, we require that the tangential component of magnetic field ${\bf{H}}_0$ has a positive lower bound on the boundary for the solvability of the MHD boundary layer equations in Sobolev spaces. That is, the condition $h_0(x, 0)\geq \sigma>0$ is stated in Theorem \ref{T1.1}.
\end{remark}

\subsection{Zeroth order boundary layer}
Putting ansatz \eqref{2.1} into \eqref{1.1} again, setting the $\varepsilon^0$-th order terms equal to zero and using \eqref{2.3},
we deduce that the leading order boundary layer profile $(\rho_b^0, u_b^0, v_b^0, h_b^0, g_b^0)(t, x, \eta)$ is governed by
\begin{align}\label{2.4}
\partial_t\rho_b^0+\left(\overline{u^0}+u_b^0\right)\partial_x\rho_b^0+\left(\mathcal{V}^1+v_b^0\right)\partial_\eta\rho_b^0+\overline{\partial_x\rho^0}u_b^0=0,
\end{align}
\begin{align}\label{2.5}
&\left(\overline{\rho^0}+\rho_b^0\right)\left(\partial_tu_b^0+(\overline{u^0}+u_b^0)\partial_xu_b^0+(\mathcal{V}^1+v_b^0)\partial_\eta u_b^0\right)-(\overline{h^0}+h_b^0)\partial_xh_b^0\notag\\
&\quad-(\mathcal{G}^1+g_b^0)\partial_\eta h_b^0
+\overline{\partial_tu^0}\rho_b^0+\overline{\partial_xu^0}\left(\overline{\rho^0}u_b^0+\rho_b^0\overline{u^0}+\rho_b^0 u_b^0\right)-\overline{\partial_xh^0}h_b^0-\mu\partial_\eta^2u_b^0=0
\end{align}
and
\begin{align}\label{2.6}
&\partial_th_b^0+(\overline{u^0}+u_b^0)\partial_xh_b^0+\left(\mathcal{V}^1+v_b^0\right)\partial_\eta h_b^0-(\overline{h^0}+h_b^0)\partial_xu_b^0
\notag\\
&\quad-(\mathcal{G}^1+g_b^0)\partial_\eta u_b^0+\overline{\partial_xh^0}u_b^0
-\overline{\partial_xu^0}h_b^0-\kappa\partial_\eta^2h_b^0=0.
\end{align}
The divergence free conditions are satisfied.
\begin{align*}
\partial_xu_b^0+\partial_\eta v_b^0=0,\qquad \partial_xh_b^0+\partial_\eta g_b^0=0.
\end{align*}
From \eqref{2.2}, we impose the following boundary conditions and far-field conditions.
\begin{align*}
(u_b^0, \partial_\eta h_b^0)|_{\eta=0}=-\left(\overline{u^0}(t, x), 0\right), \qquad \lim\limits_{\eta\to\infty}(u_b^0, h_b^0)={\bf{0}}.
\end{align*}
The initial data is chosen to be zero.
\begin{align*}
(u_b^0, h_b^0)|_{t=0}={\bf{0}}.
\end{align*}
\begin{remark}
Let
\begin{align*}
(\rho^p, u^p, v^p, h^p, g^p)(t, x, \eta)=(\overline{\rho^0}, \overline{u^0}, \mathcal{V}^1, \overline{h^0}, \mathcal{G}^1)(t, x)+(\rho_b^0, u_b^0, v_b^0, h_b^0, g_b^0)(t, x, \eta),
\end{align*}
then we find $(\rho^p, u^p, v^p, h^p, g^p)(t, x, \eta)$ solves the following initial boundary value problem.
\begin{align}\label{boun}
\begin{cases}
\partial_t\rho^p+(u^p\partial_x+v^p\partial_\eta)\rho^p=0,\\
\rho^p\partial_tu^p+\rho^p(u^p\partial_x+v^p\partial_\eta)u^p-(h^p\partial_x+g^p\partial_\eta)h^p=\mu\partial_\eta^2u^p-\overline{\partial_xp^0},\\
\partial_th^p+(u^p\partial_x+v^p\partial_\eta)h^p-(h^p\partial_x+g^p\partial_\eta)u^p=\kappa\partial_\eta^2h^p,\\
\partial_xu^p+\partial_\eta v^p=0,\quad \partial_xh^p+\partial_\eta g^p=0,\\
(u^p, v^p, \partial_\eta h^p, g^p)|_{\eta=0}={\bf{0}}, \quad \lim\limits_{\eta\to +\infty}(\rho^p, u^p, h^p)(t, x, \eta)=(\overline{\rho^0}, \overline{u^0}, \overline{h^0}).
\end{cases}
\end{align}
\end{remark}
Based on the main theorem in \cite{GHY21} and Proposition $\ref{E1}$, we have the local well-posedness of solutions to the initial boundary value problem for MHD boundary layer equations \eqref{boun}.
\begin{proposition}
\label{P2.4}
Let $(\rho^0, {\bf{u}}^0, {\bf{H}}^0, p^0)$ be a solution to \eqref{2.3} in Proposition \ref{E1}, then the system \eqref{boun} admits a unique solution $(\rho^p, u^p, v^p, h^p, g^p)(t, x, \eta)$ on $[0, T_2]$ where $T_2\in [0, T_1]$, and for any $t\in[0, T_2]$,
\begin{align*}
h^p(t, x, \eta)\ge\frac{\sigma}{2}
\end{align*}
with constant $\sigma$ being defined in Theorem $\ref{T1.1}$. Moreover, it holds that for any $l>0$,
\begin{align*}
\begin{array}{ll}
(\rho_b^0, u_b^0, h_b^0)=(\rho^p-\overline{\rho^0}, u^p-\overline{u^0}, h^p-\overline{h^0})\in\cap_{i=0}^{[m/2]-1}W^{i, \infty}\left([0, T_2]; H_l^{[m/2]-1-i}(\mathbb{T}\times\mathbb{R}_+)\right),\\
(v_b^0, g_b^0)\in\cap_{i=0}^{[m/2]-2}W^{i, \infty}\left([0, T_2]; H_l^{[m/2]-2-i}(\mathbb{T}\times\mathbb{R}_+)\right).
\end{array}
\end{align*}
\end{proposition}
Here and after, for any integer $m$ and any $l>0$, the weighted Sobolev spaces $H_l^m(\mathbb{T}\times\mathbb{R}_+)$ are defined as follows.
\begin{align*}
H_l^m(\mathbb{T}\times\mathbb{R}_+):=\left\{f(x, \eta): \|f\|_{H_l^m}:=\left(\sum\limits_{m_1+m_2\le m}\|\langle\eta\rangle^{l+m_2}\partial_x^{m_1}\partial_{\eta}^{m_2}f\|_{L^2(\mathbb{T}\times\mathbb{R}_+)}^2\right)^{\frac12}<+\infty\right\},
\end{align*}
where $\langle\eta\rangle=1+\eta$.

\subsection{First order inner flow}
Similarly, by putting ansatz \eqref{2.1} into \eqref{1.1}, setting the $\varepsilon^{\frac12}$-th order terms equal to zero and letting $\eta\to +\infty$, we deduce that the first order inner flow $(\rho^1, {\bf{u}}^1, {\bf{H}}^1, p^1)$ satisfies the following linearized ideal MHD equations around $(\rho^0, {\bf{u}}^0, {\bf{H}}^0, p^0)$.
\begin{align}\label{2.7}
\begin{cases}
\partial_t\rho^1+{\bf{u}}^0\cdot\nabla\rho^1+{\bf{u}}^1\cdot\nabla\rho^0=0,\\
\rho^0\left(\partial_t{\bf{u}}^1+{\bf{u}}^0\cdot\nabla{\bf{u}}^1\right)+\nabla p^1-{\bf{H}}^0\cdot\nabla {\bf{H}}^1+\rho^0{\bf{u}}^1\cdot\nabla{\bf{u}}^0+\rho^1\left(\partial_t{\bf{u}}^0+{\bf{u}}^0\cdot\nabla{\bf{u}}^0\right)-{\bf{H}}^1\cdot\nabla {\bf{H}}^0=0,\\
\partial_t{\bf{H}}^1+{\bf{u}}^0\cdot\nabla {\bf{H}}^1-{\bf{H}}^0\cdot\nabla {\bf{u}}^1+{\bf{u}}^1\cdot\nabla {\bf{H}}^0-{\bf{H}}^1\cdot\nabla {\bf{u}}^0=0,\\
\nabla\cdot{\bf{u}}^1=0,\qquad \nabla\cdot{\bf{H}}^1=0.
\end{cases}
\end{align}
The initial data is chosen to be zero
\begin{align}
\label{i1}
(\rho^1, {\bf{u}}^1, {\bf{H}}^1)|_{t=0}={\bf{0}},
\end{align}
and the boundary condition is imposed according to the matching conditions \eqref{2.2}.
\begin{align}
\label{b1}
(v^1, g^1)(t, x, 0)=-(v_b^0, g_b^0)(t, x, 0).
\end{align}
Similar as Proposition \ref{E1}, we also have the well-posedness result for \eqref{2.7}.
\begin{proposition}
\label{P2.5}
Let $(\rho^0, {\bf{u}}^0, {\bf{H}}^0, p^0)$ be a smooth solution to \eqref{2.3} in Proposition \ref{E1}, then there exists a unique solution $(\rho^1, {\bf{u}}^1, {\bf{H}}^1, p^1)$ to \eqref{2.7}-\eqref{b1}, and there exists a time $T_3\in [0, T_2]$, such that
\begin{align*}
(\rho^1, {\bf{u}}^1, {\bf{H}}^1, \nabla p^1)\in\cap_{j=0}^{[m/2]-2}C^j\left([0, T_3]; H^{[m-2]-2-j}(\mathbb{T}\times\mathbb{R}_+)\right).
\end{align*}
\end{proposition}

\subsection{First order boundary layer}
Putting ansatz \eqref{2.1} into \eqref{1.1}, setting the $\varepsilon^{\frac12}$-th order terms equal to zero and using \eqref{2.4}-\eqref{2.6},
we deduce that the first order boundary layer profile $(\rho_b^1, u_b^1, v_b^1, h_b^1, g_b^1)(t, x, \eta)$ is described by
\begin{equation}\label{2.8}
\begin{cases}
\partial_t\rho_b^1+\left(\overline{u^0}+u_b^0\right)\partial_x\rho_b^1+\left(\mathcal{V}^1+v_b^0\right)\partial_\eta\rho_b^1
+\left(\overline{\partial_x\rho^0}+\partial_x\rho_b^0\right)u_b^1+\partial_\eta\rho_b^0v_b^1=-f_1,\\

\left(\overline{\rho^0}+\rho_b^0\right)\left(\partial_tu_b^1+(\overline{u^0}+u_b^0)\partial_xu_b^1+(\mathcal{V}^1+v_b^0)\partial_\eta u_b^1\right)\\
\quad+\left(\overline{\partial_tu^0}+\partial_tu_b^0+(\overline{u^0}+u_b^0)(\overline{\partial_xu^0}+\partial_xu_b^0)+\left(\mathcal{V}^1+v_b^0\right)\partial_{\eta}u_b^0\right)\rho_b^1\\
\quad+\left(\overline{\rho^0}+\rho_b^0\right)\left(\overline{\partial_xu^0}+\partial_xu_b^0\right)u_b^1+\left(\overline{\rho^0}+\rho_b^0\right)\partial_{\eta}u_b^0v_b^1\\
\quad-(\overline{h^0}+h_b^0)\partial_xh_b^1-(\mathcal{G}^1+g_b^0)\partial_\eta h_b^1
-\left(\overline{\partial_xh^0}+\partial_xh_b^0\right)h_b^1-\partial_\eta h_b^0g_b^1-\mu\partial_\eta^2u_b^1=-f_2,\\

\partial_th_b^1+(\overline{u^0}+u_b^0)\partial_xh_b^1+(\mathcal{V}^1+v_b^0)\partial_\eta h_b^1+\left(\overline{\partial_xh^0}+\partial_xh_b^0\right)u_b^1+\partial_\eta h_b^0v_b^1\\
\quad-(\overline{h^0}+h_b^0)\partial_xu_b^1-(\mathcal{G}^1+g_b^0)\partial_\eta u_b^1
-\left(\overline{\partial_xu^0}+\partial_xu_b^0\right)h_b^1-\partial_\eta u_b^0g_b^1
-\kappa\partial_\eta^2h_b^1=-f_3,\\
\partial_xu_b^1+\partial_\eta v_b^1=0,\qquad \partial_xh_b^1+\partial_\eta g_b^1=0,\\
(u_b^1, \partial_\eta h_b^1)|_{\eta=0}=-\left(\overline{u^1}(t, x), \overline{\partial_yh^0}(t, x)\right), \qquad \lim\limits_{\eta\to\infty}(u_b^1, h_b^1)={\bf{0}},\\
(u_b^1, h_b^1)|_{t=0}={\bf{0}},
\end{cases}
\end{equation}
where the source terms $f_i\ (i=1,2,3)$ are defined as follows.
\begin{align*}
f_1=\mathcal{U}^1\partial_x\rho_b^0+\mathcal{V}^2\partial_\eta\rho_b^0+\partial_x\varrho^1u_b^0+\overline{\partial_y\rho^0}v_b^0,
\end{align*}
\begin{align*}
f_2=&\partial_t\mathcal{U}^1\rho_b^0+\varrho^1\partial_tu_b^0
+\left[\left(\overline{\rho^0}+\rho_b^0\right)\mathcal{U}^1+\left(\overline{u^0}+u_b^0\right)\varrho^1\right]\partial_xu_b^0\\
&+\left(\rho_b^0\mathcal{U}^1+\varrho^1u_b^0\right)\overline{\partial_xu^0}+\left(\overline{\rho^0}u_b^0+\rho_b^0\overline{u^0}+\rho_b^0u_b^0\right)\partial_x\mathcal{U}^1\\
&+\left[\left(\overline{\rho^0}+\rho_b^0\right)\mathcal{V}^2+\varrho^1(\mathcal{V}^1+v_b^0)\right]\partial_\eta u_b^0+\left(\overline{\rho^0}v_b^0+\rho_b^0\mathcal{V}^1+\rho_b^0v_b^0\right)\overline{\partial_yu^0}\\
&-h_b^0\partial_x\mathcal{H}^1
-g_b^0\overline{\partial_yh^0}-\mathcal{H}^1\partial_xh_b^0-\mathcal{G}^2\partial_\eta h_b^0
\end{align*}
and
\begin{align*}
f_3=&u_b^0\partial_x\mathcal{H}^1+v_b^0\overline{\partial_yh^0}+\mathcal{U}^1\partial_xh_b^0+\mathcal{V}^2\partial_\eta h_b^0\\
&-h_b^0\partial_x\mathcal{U}^1-g_b^0\overline{\partial_yu^0}-\mathcal{H}^1\partial_xu_b^0-\mathcal{G}^2\partial_\eta u_b^0.
\end{align*}
Moreover, we define the boundary layer of pressure in the following way, which makes the terms of order $\varepsilon^{\frac12}$ in the equation of $v^\varepsilon$ equal to zero.
\begin{align}\label{2.9}
p_b^1(t, x, \eta)=\int_{\eta}^\infty &\{\left(\overline{\rho^0}+\rho_b^0\right)\left(\partial_tv_b^0+(\overline{u^0}+u_b^0)\partial_xv_b^0+(\mathcal{V}^1+v_b^0)\partial_\eta v_b^0\right)\notag\\
+&\partial_t\mathcal{V}^1\rho_b^0+\left(\overline{\rho^0}u_b^0+\rho_b^0\overline{u^0}+\rho_b^0u_b^0\right)\partial_x\mathcal{V}^1\notag\\
+&\left(\rho_b^0(\mathcal{V}^1+v_b^0)+\overline{\rho_0}v_b^0\right)\overline{\partial_yv^0}-\left(\overline{h^0}+h_b^0\right)\partial_xg_b^0\notag\\
-&\left(\mathcal{G}^1+g_b^0\right)\partial_\eta g_b^0
-h_b^0\partial_x\mathcal{G}^1-g_b^0\overline{\partial_yg^0}-\mu\partial_{\eta\eta}v_b^0\}(t, x, \tilde{\eta})\;d\tilde{\eta}.
\end{align}
Similar as Proposition \ref{P2.4}, we have the local well-posedness for \eqref{2.8}.
\begin{proposition}
\label{P2.6}
Let $(\rho^0, {\bf{u}}^0, {\bf{H}}^0, p^0)$ and $(\rho^1, {\bf{u}}^1, {\bf{H}}^1, p^1)$ be the solutions to \eqref{2.3} and \eqref{2.7} respectively. Let $(\rho_b^0, u^0_b, v_b^0, h^0_b, g_b^0)(t, x, \eta)$ be a solution to \eqref{2.4}-\eqref{2.6}. Then the system \eqref{2.8} admits a unique solution $(\rho_b^1, u_b^1, h_b^1)$ on $[0, T_4]$ where $T_4\in [0, T_3]$, and for any $l>0$,
\begin{align*}
\begin{array}{ll}
(\rho_b^1, u_b^1, h_b^1)\in\cap_{i=0}^{[m/4]-3}W^{i, \infty}\left([0, T_4]; H_l^{[m/4]-3-i}(\mathbb{T}\times\mathbb{R}_+)\right),\\
(v_b^1, g_b^1)\in\cap_{i=0}^{[m/4]-4}W^{i, \infty}\left([0, T_4]; H_l^{[m/4]-4-i}(\mathbb{T}\times\mathbb{R}_+)\right).
\end{array}
\end{align*}
\end{proposition}

\subsection{Second order inner flow}
By the same procedure, we deduce from the $\varepsilon$-th order terms while $\eta\to +\infty$, that the second order inner flow $(\rho^2, {\bf{u}}^2, {\bf{H}}^2, p^2)$ satisfies the following linearized ideal MHD equations around $(\rho^0, {\bf{u}}^0, {\bf{H}}^0, p^0)$.
\begin{align}\label{2.10}
\begin{cases}
\partial_t\rho^2+{\bf{u}}^0\cdot\nabla\rho^2+{\bf{u}}^2\cdot\nabla\rho^0=-{\bf{u}}^1\cdot\nabla\rho^1,\\
\rho^0\left(\partial_t{\bf{u}}^2+{\bf{u}}^0\cdot\nabla{\bf{u}}^2\right)+\nabla p^2-{\bf{H}}^0\cdot\nabla {\bf{H}}^2+\rho^0{\bf{u}}^2\cdot\nabla{\bf{u}}^0\\
\quad+\rho^2\left(\partial_t{\bf{u}}^0+{\bf{u}}^0\cdot\nabla{\bf{u}}^0\right)-{\bf{H}}^2\cdot\nabla {\bf{H}}^0=-f_4,\\
\partial_t{\bf{H}}^2+{\bf{u}}^0\cdot\nabla {\bf{H}}^2-{\bf{H}}^0\cdot\nabla {\bf{u}}^2+{\bf{u}}^2\cdot\nabla {\bf{H}}^0-{\bf{H}}^2\cdot\nabla {\bf{u}}^0=-f_5,\\
\nabla\cdot{\bf{u}}^2=0,\qquad \nabla\cdot{\bf{H}}^2=0,
\end{cases}
\end{align}
where the source terms
\begin{align*}
f_4=\rho^1\left(\partial_t{\bf{u}}^1+{\bf{u}}^0\cdot\nabla{\bf{u}}^1+{\bf{u}}^1\cdot\nabla{\bf{u}}^0\right)
+\rho^0{\bf{u}}^1\cdot\nabla{\bf{u}}^1-{\bf{H}}^1\cdot\nabla{\bf{H}}^1-\mu\Delta{\bf{u}}^0
\end{align*}
and
\begin{align*}
f_5={\bf{u}}^1\cdot\nabla{\bf{H}}^1-{\bf{H}}^1\cdot\nabla{\bf{u}}^1-\kappa\Delta{\bf{H}}^0.
\end{align*}
The initial data is chosen to be zero
\begin{align}
\label{i2}
(\rho^2, {\bf{u}}^2, {\bf{H}}^2)|_{t=0}={\bf{0}}.
\end{align}
And the boundary condition is imposed as follows by \eqref{2.2}.
\begin{align}
\label{b2}
(v^2, g^2)(t, x, 0)=-(v_b^1, g_b^1)(t, x, 0).
\end{align}
Similar as Proposition \ref{E1}, we also have the well-posedness result for \eqref{2.10}.
\begin{proposition}
\label{P2.7}
Let $(\rho^0, {\bf{u}}^0, {\bf{H}}^0, p^0)$ and $(\rho^1, {\bf{u}}^1, {\bf{H}}^1, p^1)$ be the smooth solutions to \eqref{2.3} and \eqref{2.7} respectively. Then, there exists a unique solution $(\rho^2, {\bf{u}}^2, {\bf{H}}^2, p^2)$ to \eqref{2.10}-\eqref{b2} on $[0, T_5]$ where $T_5\in [0, T_4]$. Moreover,
\begin{align*}
(\rho^2, {\bf{u}}^2, {\bf{H}}^2, \nabla p^2)\in\cap_{j=0}^{[m/4]-4}C^j\left([0, T_5]; H^{[m/4]-4-j}(\mathbb{T}\times\mathbb{R}_+)\right).
\end{align*}
\end{proposition}

\subsection{Second order boundary layer}
Putting ansatz \eqref{2.1} into \eqref{1.1}, setting the $\varepsilon$-th order terms equal to zero and using \eqref{2.5},
we deduce that the second order boundary layer profile $(\rho_b^2, u_b^2, v_b^2, h_b^2, g_b^2)(t, x, \eta)$ is given by
\begin{align}\label{2.11}
\partial_t\rho_b^2+\left(\overline{u^0}+u_b^0\right)\partial_{x}\rho_b^2+\left(\mathcal{V}^1+v_b^0\right)\partial_\eta\rho_b^2
+\left(\overline{\partial_x\rho^0}+\partial_{x}\rho_b^0\right)u_b^2+\partial_\eta\rho_b^0v_b^2=-f_6
\end{align}
with the source term
\begin{align*}
f_6=&\mathcal{U}^2\partial_x\rho_b^0+\left(\mathcal{U}^1+u_b^1\right)\partial_x\rho_b^1
+u_b^1\partial_x\varrho^1+u_b^0\partial_x\varrho^2\\
&+\mathcal{V}^3\partial_\eta\rho_b^0+v_b^1\overline{\partial_y\rho^0}+\left(\mathcal{V}^2+v_b^1\right)\partial_\eta\rho_b^1
+v_b^0\partial_y\varrho^1,
\end{align*}
and
\begin{align}\label{2.12}
&\left(\overline{\rho^0}+\rho_b^0\right)\left(\partial_tu_b^2+(\overline{u^0}+u_b^0)\partial_xu_b^2+(\mathcal{V}^1+v_b^0)\partial_\eta u_b^2\right)\notag\\
&\quad+\left(\overline{\partial_tu^0}+\partial_tu_b^0+(\overline{u^0}+u_b^0)(\overline{\partial_xu^0}+\partial_xu_b^0)+\left(\mathcal{V}^1+v_b^0\right)\partial_{\eta}u_b^0\right)\rho_b^2\notag\\
&\quad+\left(\overline{\rho^0}+\rho_b^0\right)\left(\overline{\partial_xu^0}+\partial_xu_b^0\right)u_b^2+\left(\overline{\rho^0}+\rho_b^0\right)\partial_{\eta}u_b^0v_b^2\notag\\
&\quad-(\overline{h^0}+h_b^0)\partial_xh_b^2-(\mathcal{G}^1+g_b^0)\partial_\eta h_b^2
-\left(\overline{\partial_xh^0}+\partial_xh_b^0\right)h_b^2-\partial_\eta h_b^0g_b^2-\mu\partial_\eta^2u_b^2=-f_7
\end{align}
with
\begin{align*}
f_7=&\varrho^2\partial_tu_b^0+\left(\varrho^1+\rho_b^1\right)\partial_tu_b^1+\rho_b^1\partial_t\mathcal{U}^1+\rho_b^0\partial_t\mathcal{U}^2\\
&+\left[\varrho^2(\overline{u^0}+u_b^0)+(\overline{\rho^0}+\rho_b^0)\mathcal{U}^2+(\varrho^1+\rho_b^1)(\mathcal{U}^1+u_b^1)\right]\partial_xu_b^0\\
&+\left(\varrho^2u_b^0+\mathcal{U}^2\rho_b^0+\varrho^1u_b^1+\rho_b^1\mathcal{U}^1+\rho_b^1u_b^1\right)\overline{\partial_xu^0}\\
&+\left[(\varrho^1+\rho_b^1)(\overline{u^0}+u_b^0)+(\overline{\rho^0}+\rho_b^0)(\mathcal{U}^1+u_b^1)\right]\partial_xu_b^1\\
&+\left(\varrho^1u_b^0+\rho_b^1\overline{u^0}+\rho_b^1u_b^0+\overline{\rho^0}u_b^1+\rho_b^0\mathcal{U}^1+\rho_b^0u_b^1\right)\partial_x\mathcal{U}^1\\
&+\left(\overline{\rho^0}u_b^0+\rho_b^0\overline{u^0}+\rho_b^0u_b^0\right)\partial_x\mathcal{U}^2\\
&+\left[(\overline{\rho^0}+\rho_b^0)\mathcal{V}^3+(\varrho^1+\rho_b^1)(\mathcal{V}^2+v_b^1)+\varrho^2(\mathcal{V}^1+v_b^0)\right]\partial_\eta u_b^0\\
&+\left[(\overline{\rho^0}+\rho_b^0)(\mathcal{V}^2+v_b^1)+(\varrho^1+\rho_b^1)(\mathcal{V}^1+v_b^0)\right]\partial_\eta u_b^1\\
&+\left(\overline{\rho^0}v_b^1+\rho_b^0\mathcal{V}^2+\rho_b^0v_b^1+\varrho^1v_b^0+\rho_b^1\mathcal{V}^1+\rho_b^1v_b^0\right)\overline{\partial_yu^0}\\
&+\left(\overline{\rho^0}v_b^0+\rho_b^0\mathcal{V}^1+\rho_b^0v_b^0\right)\partial_y\mathcal{U}^1-\mathcal{U}^2\partial_xh_b^0
-\left(\mathcal{H}^1+h_b^1\right)\partial_xh_b^1
-h_b^1\partial_x\mathcal{H}^1\\
&-h_b^0\partial_x\mathcal{H}^2-\mathcal{G}^3\partial_\eta h_b^0-g_b^1\overline{\partial_yh^0}
-\left(\mathcal{G}^2+g_b^1\right)\partial_\eta h_b^1
-g_b^0\partial_y\mathcal{H}^1+\mu\partial_{xx}u_b^0,
\end{align*}
and
\begin{align}\label{2.13}
&\partial_th_b^2+(\overline{u^0}+u_b^0)\partial_xh_b^2+(\mathcal{V}^1+v_b^0)\partial_\eta h_b^2+\left(\overline{\partial_xh^0}+\partial_xh_b^0\right)u_b^2+\partial_\eta h_b^0v_b^2\notag\\
&\quad-(\overline{h^0}+h_b^0)\partial_xu_b^2-(\mathcal{G}^1+g_b^0)\partial_\eta u_b^2
-\left(\overline{\partial_xu^0}+\partial_xu_b^0\right)h_b^2-\partial_\eta u_b^0g_b^2
-\kappa\partial_\eta^2h_b^2=-f_8
\end{align}
with the source term
\begin{align*}
f_8=&\mathcal{U}^2\partial_xh_b^0+\left(\mathcal{U}^1+u_b^1\right)\partial_xh_b^1+u_b^1\partial_x\mathcal{H}^1+u_b^0\partial_x\mathcal{H}^2\\
&+\mathcal{V}^3\partial_\eta h_b^0+v_b^1\overline{\partial_yh^0}
+\left(\mathcal{V}^2+v_b^1\right)\partial_\eta h_b^1
+v_b^0\partial_y\mathcal{H}^1\\
&-\mathcal{H}^2\partial_xu_b^0
-\left(\mathcal{H}^1+h_b^1\right)\partial_xu_b^1
-h_b^1\partial_x\mathcal{U}^1
-h_b^0\partial_x\mathcal{U}^2\\
&-\mathcal{G}^3\partial_\eta u_b^0
-g_b^1\overline{\partial_yu^0}
-\left(\mathcal{G}^2+g_b^1\right)\partial_\eta u_b^1
-g_b^0\partial_y\mathcal{U}^1
+\kappa\partial_{xx}h_b^0.\\
\end{align*}
The divergence free conditions are satisfied.
\begin{align}
\label{div}
\partial_xu_b^2+\partial_\eta v_b^2=0,\qquad \partial_xh_b^2+\partial_\eta g_b^2=0.
\end{align}
The boundary condition and the initial data are stated as follows.
\begin{align}
\label{b_3}
(u_b^2, \partial_\eta h_b^2)|_{\eta=0}=-\left(\overline{u^2}(t, x), \overline{\partial_yh^1}(t, x)\right)
\end{align}
and
\begin{align}
\label{i3}
(u_b^2, h_b^2)|_{t=0}={\bf{0}}.
\end{align}
The representation of $p_b^2$ can be deduced from the $v$-equation similarly.
Then as Proposition \ref{P2.4}, we have the local well-posedness for \eqref{2.11}-\eqref{i3}.
\begin{proposition}
\label{P2.8}
Let $(\rho^0, {\bf{u}}^0, {\bf{H}}^0, p^0)$, $(\rho^1, {\bf{u}}^1, {\bf{H}}^1, p^1)$ and $(\rho^2, {\bf{u}}^2, {\bf{H}}^2, p^2)$ be the solutions to \eqref{2.3}, \eqref{2.7} and \eqref{2.10} respectively. Let $(\rho_b^0, u^0_b, v_b^0, h^0_b, g_b^0)$ and $(\rho_b^1, u^1_b, v_b^1, h^1_b, g_b^1)$ be the solutions to \eqref{2.4}-\eqref{2.6} and \eqref{2.8}-\eqref{2.9} respectively. Then, the system \eqref{2.11}-\eqref{i3} admits a unique solution $(\rho_b^2, u_b^2, h_b^2)$ on $[0, T_6]$ where $T_6\in [0, T_5]$. Moreover, for any $l>0$,
\begin{align*}
\begin{array}{ll}
(\rho_b^2, u_b^2, h_b^2)\in\cap_{i=0}^{[m/8]-5}W^{i, \infty}\left([0, T_6]; H_l^{[m/8]-5-i}(\mathbb{T}\times\mathbb{R}_+)\right),\\
(v_b^2, g_b^2)\in\cap_{i=0}^{[m/8]-6}W^{i, \infty}\left([0, T_6]; H_l^{[m/8]-6-i}(\mathbb{T}\times\mathbb{R}_+)\right).
\end{array}
\end{align*}
\end{proposition}

\subsection{The approximate solution}
Based on the construction above, we define the approximate solution $(\rho^a, {\bf{u}}^a, {\bf{H}}^a, p^a)$ as follows, which is a little different from \eqref{2.1}.
\begin{align}\label{2.14}
\begin{cases}
\rho^a(t, x, y)=\rho^0(t, x, y)+\rho_b^0\left(t, x, \frac{y}{\sqrt{\varepsilon}}\right)+\sqrt{\varepsilon}\left[\rho^1(t, x, y)+\rho_b^1\left(t, x, \frac{y}{\sqrt{\varepsilon}}\right)\right]\\
\qquad\qquad\qquad\qquad\qquad+\varepsilon \left[\rho^2(t, x, y)+\rho_b^2\left(t, x, \frac{y}{\sqrt{\varepsilon}}\right)\right],\\
(u^a, h^a)(t, x, y)=(u^0, h^0)(t, x, y)+(u_b^0, h_b^0)\left(t, x, \frac{y}{\sqrt{\varepsilon}}\right)\\
\qquad\qquad\qquad\quad+\sqrt{\varepsilon}\left[(u^1, h^1)(t, x, y)+(u_b^1, h_b^1)\left(t, x, \frac{y}{\sqrt{\varepsilon}}\right)\right]\\
\qquad\qquad\qquad\quad+\varepsilon\left[(u^2, h^2)(t, x, y)+\chi (y)(u_b^2, h_b^2)\left(t, x, \frac{y}{\sqrt{\varepsilon}}\right)\right]\\
\qquad\qquad\qquad\quad+\varepsilon^{\frac32}\left(\chi'(y)\int_0^{\frac{y}{\sqrt{\varepsilon}}} u_b^2(t, x, \tilde{\eta})d\tilde{\eta}, \chi'(y)\int_0^{\frac{y}{\sqrt{\varepsilon}}} h_b^2(t, x, \tilde{\eta})d\tilde{\eta}+\phi(t, x, \frac{y}{\sqrt{\varepsilon}})\right),\\
(v^a, g^a)(t, x, y)=(v^0, g^0)(t, x, y)+\sqrt{\varepsilon}\left[(v_b^0, g_b^0)\left(t, x, \frac{y}{\sqrt{\varepsilon}}\right)+(v^1, g^1)(t, x, y)\right]\\
\qquad\qquad\qquad\qquad\quad+\varepsilon\left[(v_b^1, g_b^1)\left(t, x, \frac{y}{\sqrt{\varepsilon}}\right)+(v^2, g^2)(t, x, y)\right]\\
\qquad\qquad\qquad\qquad\quad+\varepsilon^{\frac32}\left(\chi (y)v_b^2, \chi (y)g_b^2-\sqrt{\varepsilon}\int_0^{\frac{y}{\sqrt{\varepsilon}}} \partial_x\phi(t, x, \tilde{\eta})d\tilde{\eta}\right)\left(t, x, \frac{y}{\sqrt{\varepsilon}}\right),\\
p^a(t, x, y)=p^0(t, x, y)+\sqrt{\varepsilon}p^1(t, x, y)+\varepsilon\left[ p_b^1\left(t, x, \frac{y}{\sqrt{\varepsilon}}\right)+p^2(t, x, y)\right]+\varepsilon^{\frac32}  p_b^2\left(t, x, \frac{y}{\sqrt{\varepsilon}}\right).
\end{cases}
\end{align}
Here $\chi(y)$ is a smooth cut-off function which satisfies
\begin{align}\label{2.15}
\chi(y)=
\begin{cases}
1, \quad y\in [0, 1],\\
0, \quad y\in [2, +\infty),
\end{cases}
\end{align}
it ensures that $(v^a, g^a)$ decays fast as $y\to +\infty$.
The boundary corrector $\phi(t, x, \eta)$ is chosen to cancel the value of $\partial_yh^2$ on the boundary $y=0$, and it satisfies
\begin{align}\label{2.16}
\partial_\eta\phi(t, x, 0)=-\partial_yh^2(t, x, 0).
\end{align}
Then, it is direct to check that the approximate solution $(\rho^a, {\bf{u}}^a, {\bf{H}}^a, p^a)$ defined above solves
\begin{align}\label{2.17}
\begin{cases}
\partial_t\rho^a+{\bf{u}}^a\cdot\nabla\rho^a=R_0,\\
\rho^a\left(\partial_t{\bf{u}}^a+{\bf{u}}^a\cdot\nabla{\bf{u}}^a\right)+\nabla p^a-{\bf{H}}^a\cdot\nabla {\bf{H}}^a-\mu\varepsilon\Delta {\bf{u}}^a=R_{{\bf{u}}},\\
\partial_t{\bf{H}}^a+{\bf{u}}^a\cdot\nabla {\bf{H}}^a-{\bf{H}}^a\cdot\nabla {\bf{u}}^a-\kappa\varepsilon\Delta{\bf{H}}^a=R_{{\bf{H}}},\\
\nabla\cdot{\bf{u}}^a=0,\qquad \nabla\cdot{\bf{H}}^a=0,\\
(\rho^a, {\bf{u}}^a, {\bf{H}}^a)|_{t=0}=(\rho_0, {\bf{u}}_0, {\bf{H}}_0)(x, y),\\
(u^a, v^a, \partial_yh^a, g^a)|_{y=0}={\bf{0}},
\end{cases}
\end{align}
where the remainders $R_{{\bf{u}}}=(R_1, R_2)$, $R_{{\bf{H}}}=(R_3, R_4)$. For these remainders, we have the following propositions.
\begin{proposition}\label{remainder}
For any $\alpha\in\mathbb{N}^2$ and non-negative integer $k\le 1$, the remainders in \eqref{2.17} satisfy the following estimates.
\begin{align}\label{2.18}
\|\partial_{tx}^\alpha\partial_y^k R_i(t, \cdot)\|_{L^2}+\|(y\partial_y)\partial_y^k R_i(t, \cdot)\|_{L^2}\le C_i\varepsilon^{\frac{3-k}2},\quad i=0, 1, 2, 3, 4
\end{align}
hold for some positive constants $C_i\ (i=0, 1, 2, 3, 4)$ which are independent of $\varepsilon$.
\end{proposition}
\begin{remark}
The loss of $\varepsilon^{\frac12}$ happens when the normal derivative operator $\partial_y$ is applied on the boundary layer profile due to the fast variable $\eta=\frac{y}{\sqrt{\varepsilon}}$.
\end{remark}
The expressions of the remainders and the proof of Proposition \ref{remainder} will be given in details in Appendix B.

\section{Estimates of the error functions}
Let $(\rho^\varepsilon, {\bf{u}}^\varepsilon, {\bf{H}}^\varepsilon, p^\varepsilon)$ be a solution to \eqref{1.1}-\eqref{1.3}, and the approximate solution $(\rho^a, {\bf{u}}^a, {\bf{H}}^a, p^a)$ is constructed in the above section. Set
\begin{align}\label{3.1}
(\rho^\varepsilon, {\bf{u}}^\varepsilon, {\bf{H}}^\varepsilon, p^\varepsilon)(t, x, y)=(\rho^a, {\bf{u}}^a, {\bf{H}}^a, p^a)(t, x, y)+\varepsilon^{\frac32}(\rho, {\bf{u}}, {\bf{H}}, p)(t, x, y)
\end{align}
with ${\bf{u}}=(u, v)$ and ${\bf{H}}=(h, g)$. From \eqref{1.1} and \eqref{2.17}, we deduce that the error function $(\rho, {\bf{u}}, {\bf{H}}, p)$ satisfies the following initial boundary value problem.
\begin{align}\label{3.2}
\begin{cases}
\partial_t\rho+{\bf{u}}^\varepsilon\cdot\nabla\rho+u\partial_x\rho^a+v\partial_y\rho^a=-r_0,\\
\rho^\varepsilon\left(\partial_t{\bf{u}}+{\bf{u}}^\varepsilon\cdot\nabla{\bf{u}}\right)+\nabla p-{\bf{H}}^\varepsilon\cdot\nabla{\bf{H}}-\mu\varepsilon\Delta {\bf{u}}\\
\qquad+\rho\left(\partial_t{\bf{u}}^a+{\bf{u}}^a\cdot\nabla{\bf{u}}^a\right)
+\rho^\varepsilon{\bf{u}}\cdot\nabla{\bf{u}}^a-{\bf{H}}\cdot\nabla{\bf{H}}^a=-r_{{\bf{u}}},\\
\partial_t{\bf{H}}+{\bf{u}}^\varepsilon\cdot\nabla {\bf{H}}-{\bf{H}}^\varepsilon\cdot\nabla {\bf{u}}-\kappa\varepsilon\Delta{\bf{H}}
+{\bf{u}}\cdot\nabla{\bf{H}}^a-{\bf{H}}\cdot\nabla{\bf{u}}^a=-r_{{\bf{H}}},\\
\nabla\cdot{\bf{u}}=0,\qquad \nabla\cdot{\bf{H}}=0,\\
(\rho, {\bf{u}}, {\bf{H}})|_{t=0}={\bf{0}},\quad (u, v, \partial_yh, g)|_{y=0}={\bf{0}},
\end{cases}
\end{align}
where $r_{{\bf{u}}}=(r_1, r_2)$, $r_{{\bf{H}}}=(r_3, r_4)$ and $r_i=\varepsilon^{-\frac32}R_i\ (i=0, 1, 2, 3, 4)$. Furthermore, we derive from Proposition \ref{remainder} that for $k=0, 1$,
\begin{align}\label{3.3}
\|\partial_{tx}^{\alpha}\partial_y^kr_i(t, \cdot)\|_{L^2}+\|(y\partial_y)\partial_y^k r_i(t, \cdot)\|_{L^2}\le C_i\varepsilon^{-\frac k2},\quad |\alpha|\le2, \; i=0, 1, 2, 3, 4.
\end{align}

\subsection{A key transformation} The main difficulty in deriving uniform energy estimates of solutions to the initial boundary value problem \eqref{3.2} comes from the terms $v\partial_y\rho^a$, $v\partial_yu^a-g\partial_yh^a$ and $v\partial_yh^a-g\partial_yu^a$ which behave like $O(\varepsilon^{-\frac12})(v, g)$ due to the fast variable $\eta=\frac{y}{\sqrt{\varepsilon}}$. It is noted that there exists a strong boundary layer of density. To overcome this issue, we first learn from the divergence free condition of $\nabla\cdot{\bf{H}}=0$, there exists a stream function $\psi(t, x, y)$, such that
\begin{align}\label{3.4}
h=\partial_y\psi,\quad g=-\partial_x\psi,\quad \psi|_{y=0}=0,\quad \psi|_{t=0}=0,
\end{align}
and it satisfies
\begin{align}\label{3.5}
\partial_t\psi+(u^\varepsilon\partial_x+v^\varepsilon\partial_y)\psi-g^au+h^av-\kappa\varepsilon\Delta\psi=\partial_y^{-1}r_4=:r_5.
\end{align}
Then we introduce the following transformation
\begin{align}\label{3.6}
&\tilde{u}(t, x, y):=u(t, x, y)-\partial_y(a^p\cdot\psi)(t, x, y),\qquad \tilde{v}(t, x, y):=v(t, x, y)+\partial_x(a^p\cdot\psi)(t, x, y),\notag\\
&\tilde{h}(t, x, y):=h(t, x, y)-(b^p\cdot\psi)(t, x, y), \qquad \tilde{g}(t, x, y):=g(t, x, y),\notag\\
&\tilde{\rho}(t, x, y):=\rho(t, x, y)-(c^p\cdot\psi)(t, x, y),
\end{align}
where
\begin{align}\label{3.7}
a^p(t, x, y):=\zeta(y)\frac{u^p\left(t, x, \frac{y}{\sqrt{\varepsilon}}\right)}{h^p\left(t, x, \frac{y}{\sqrt{\varepsilon}}\right)},
\quad b^p(t, x, y):=\frac{\partial_yh^p\left(t, x, \frac{y}{\sqrt{\varepsilon}}\right)}{h^p\left(t, x, \frac{y}{\sqrt{\varepsilon}}\right)},
\quad c^p(t, x, y):=\frac{\partial_y\rho^p\left(t, x, \frac{y}{\sqrt{\varepsilon}}\right)}{h^p\left(t, x, \frac{y}{\sqrt{\varepsilon}}\right)},
\end{align}
here the cut-off function $\zeta(y)\in C^\infty(\mathbb{R}_+), 0\le\zeta(y)\le 1$, and
\begin{align*}
\zeta(y)=
\begin{cases}
1,\quad 0\le y\le1,\\
0,\quad y\ge2.
\end{cases}
\end{align*}
It is noticed that the initial data and the boundary conditions are unchanged under this transformation \eqref{3.6}.

\subsection{The equations of new error functions}
By introducing the new unknowns \eqref{3.6}, the singular terms in \eqref{3.2} will be cancelled. Consequently, we will consider the initial boundary value problem of the new unknowns \eqref{3.6} instead of \eqref{3.2}. For this purpose,
denote
\begin{align}\label{3.8}
{\bf{U}}(t, x, y):=(\tilde{u}, \tilde{v}, \tilde{h}, \tilde{g})^T(t, x, y),
\end{align}
then the system \eqref{3.2} can be rewritten as follows.
\begin{align}\label{3.9}
\begin{cases}
\partial_t\tilde{\rho}+{\bf{u}}^\varepsilon\cdot\nabla\tilde{\rho}+c^p\kappa\varepsilon(\partial_y\tilde{h}-\partial_x\tilde{g})+{\bf{C}}^0
={\bf{E}}^0,\\
{\bf{S}}^\varepsilon\partial_t{\bf{U}}+{\bf{A}}_1^\varepsilon\partial_x{\bf{U}}+{\bf{A}}_2^\varepsilon\partial_y{\bf{U}}+{\bf{C}}^\varepsilon+(p_x, p_y, 0, 0)^T-\varepsilon{\bf{B}}^\varepsilon\Delta {\bf{U}}={\bf{E}}^\varepsilon,\\
(\tilde{\rho}, {\bf{U}})|_{t=0}={\bf{0}},\qquad (\tilde{u}, \tilde{v}, \partial_y\tilde{h}, \tilde{g})|_{y=0}={\bf{0}},
\end{cases}
\end{align}
where ${\bf{S}}^\varepsilon, {\bf{A}}_1^\varepsilon, {\bf{A}}_2^\varepsilon$ and ${\bf{B}}^\varepsilon$ are matrices and ${\bf{C}}^0, {\bf{E}}^0, {\bf{C}}^\varepsilon, {\bf{E}}^\varepsilon$ are vectors. And ${\bf{S}}^\varepsilon$ is a diagonal matrix and has the form
\begin{align}\label{3.10}
{\bf{S}}^\varepsilon=diag(\rho^\varepsilon, \rho^\varepsilon, 1, 1),
\end{align}
and ${\bf{A}}_i^\varepsilon\ (i=1, 2)$ can be separated as follows:
\begin{align}\label{3.11}
{\bf{A}}_i^\varepsilon={\bf{A}}_i+\sqrt{\varepsilon}{\bf{A}}_i^p,\qquad i=1, 2,
\end{align}
where
\begin{align*}
{\bf{A}}_1=\left(
\begin{array}{cc}
\rho^\varepsilon(u^\varepsilon+a^ph^\varepsilon) {\bf{I}}_{2\times 2}& [\rho^\varepsilon(a^p)^2-1]h^\varepsilon{\bf{I}}_{2\times2}\\
-h^\varepsilon{\bf{I}}_{2\times2}& (u^\varepsilon-a^ph^\varepsilon){\bf{I}}_{2\times 2}
\end{array}
\right),
\end{align*}
\begin{align*}
{\bf{A}}_2=\left(
\begin{array}{cc}
\rho^\varepsilon(v^\varepsilon+a^pg^\varepsilon) {\bf{I}}_{2\times 2}& [\rho^\varepsilon(a^p)^2-1]g^\varepsilon{\bf{I}}_{2\times2}\\
-g^\varepsilon{\bf{I}}_{2\times2}& (v^\varepsilon-a^pg^\varepsilon){\bf{I}}_{2\times 2}
\end{array}
\right)
\end{align*}
and
\begin{align*}
{\bf{A}}_1^p=\sqrt{\varepsilon}\left(
\begin{array}{cc}
{\bf{0}}_{2\times 2}& \begin{array}{cc}
-2\mu\partial_xa^p& (\mu-\rho^\varepsilon\kappa)(\partial_ya^p+a^pb^p)\\
0 & (\rho^\varepsilon\kappa-3\mu)\partial_xa^p
\end{array}\\
{\bf{0}}_{2\times2}& {\bf{0}}_{2\times 2}
\end{array}
\right),
\end{align*}
\begin{align*}
{\bf{A}}_2^p=\sqrt{\varepsilon}\left(
\begin{array}{cc}
{\bf{0}}_{2\times 2}& \begin{array}{cc}
(\rho^\varepsilon\kappa-3\mu)\partial_ya^p+(\rho^\varepsilon\kappa-\mu)a^pb^p& 0\\
(\mu-\rho^\varepsilon\kappa)\partial_xa^p & -2\mu\partial_ya^p
\end{array}\\
{\bf{0}}_{2\times2}& {\bf{0}}_{2\times 2}
\end{array}
\right)
\end{align*}
with ${\bf{I}}_{2\times 2}$ being a $2\times 2$ identity matrix and ${\bf{0}}_{2\times 2}$ be a $2\times 2$ zero matrix. The matrix ${\bf{B}}^\varepsilon$ and vectors ${\bf{E}}^0$, ${\bf{E}}^\varepsilon$ are given by
\begin{align}\label{3.12}
{\bf{B}}^\varepsilon=\left(
\begin{array}{cc}
\mu {\bf{I}}_{2\times 2}& (\mu-\rho^\varepsilon\kappa)a^p{\bf{I}}_{2\times2}\\
{\bf{0}}_{2\times2}& \kappa{\bf{I}}_{2\times 2}
\end{array}
\right),
\end{align}
\begin{align}\label{3.13}
({\bf{E}}^0, {\bf{E}}^\varepsilon)=\left(-r_0-c^pr_5, -r_1-\rho^\varepsilon\partial_y(a^pr_5), -r_2+\rho^\varepsilon\partial_x(a^pr_5), -r_3-b^pr_5, -r_4\right)^T.
\end{align}
The expression of vectors ${\bf{C}}$ and ${\bf{C}}^\varepsilon$ will be given in Appendix C. With this transformation, the following lemma shows that the original error function $(\rho, {\bf{u}}, {\bf{H}})$ in fact can be controlled by the new unknown $(\tilde{\rho}, {\bf{U}})$ in the following sense.
\begin{lemma}
\label{L3.1}
There exists a constant $C$ which is independent of $\varepsilon$, such that, for any $|\alpha|\le2$,
\begin{align}
&\|\partial_{x}\rho(t, \cdot)\|_{L^2}+\|\partial_{tx}^{\alpha}({\bf{u}}, {\bf{H}})(t, \cdot)\|_{L^2}
\le C\left(\|\partial_{x}\tilde{\rho}(t, \cdot)\|_{L^2}+\sum\limits_{\beta\le\alpha}\|\partial_{tx}^{\beta}({\bf{U}})(t, \cdot)\|_{L^2}\right),\notag\\
&\varepsilon^{\frac12}\|\partial_y({\bf{u}}, {\bf{H}})\|_{L^2}\le C\left(\varepsilon^{\frac12}\|\partial_y{\bf{U}}\|_{L^2}+\|{\bf{U}}\|_{L^2}\right).
\end{align}
\end{lemma}
The proof of this Lemma can be found in \cite{LXY192}.

The properties of the vector ${\bf{E}}^\varepsilon$, the matrix ${\bf{A}}_i^p$ and the vector ${\bf{C}}^\varepsilon$ are summarized into the following two Propositions for later use.
\begin{proposition}
\label{P3.2}
For any multi-index $\alpha$ satisfying $|\alpha|\le2$ and $i=1, 2$, it holds
\begin{align}\label{3.14}
\|\partial_{tx}^{\alpha}{\bf{A}}_i^p\|_{L^\infty}+\varepsilon^{\frac k2}\|\partial_{tx}^\alpha\partial_y^k({\bf{E}}^0, {\bf{E}}^\varepsilon)\|_{L^2}+\|y\partial_y({\bf{E}}^0, {\bf{E}}^\varepsilon)\|_{L^2}\le C,\quad k=0, 1.
\end{align}
\end{proposition}
The proof is straightforward due to \eqref{3.3} and the properties of the approximate solution.
\begin{proposition}\label{vector}
There exists a positive constant $C$ which is independent of $\varepsilon$, such that
\begin{align}\label{3.15}
&\|{\bf{C}}^\varepsilon\|_{L^2}+\|{\bf{C}}^0\|_{L^2}\le C\|{\bf{U}}\|_{L^2},\quad \|\mathcal{Z}{\bf{C}}^0\|_{L^2}\le C\|\mathcal{Z}{\bf{U}}\|_{L^2}+\|{\bf{U}}|_{L^2},\notag\\
&\|\partial_\tau{\bf{C}}^\varepsilon\|_{L^2}\le C\|\partial_\tau{\bf{U}}\|_{L^2}+(1+Q(t))\|{\bf{U}}\|_{L^2},\notag\\
&\|y\partial_y{\bf{C}}^\varepsilon\|_{L^2}\le C\|y\partial_y{\bf{U}}\|_{L^2}+(1+Q(t))\|{\bf{U}}\|_{L^2}.
\end{align}
\end{proposition}
The proof will be given in Appendix C.

Next, we introduce a diagonal matrix
\begin{align}\label{3.16}
{\bf{D}}:=diag\left(1, 1, 1-\rho^\varepsilon(a^p)^2, 1-\rho^\varepsilon(a^p)^2\right)^T,
\end{align}
such that ${\bf{DA}}_i\ (i=1, 2)$ are symmetric and have the following form.
\begin{align}\label{3.17}
{\bf{D}}{\bf{A}}_1=\left(
\begin{array}{cc}
\rho^\varepsilon(u^\varepsilon+a^ph^\varepsilon) {\bf{I}}_{2\times 2}& [\rho^\varepsilon(a^p)^2-1]h^\varepsilon{\bf{I}}_{2\times2}\\
{[}\rho^\varepsilon(a^p)^2-1{]}h^\varepsilon{\bf{I}}_{2\times2}&  [1-\rho^\varepsilon(a^p)^2](u^\varepsilon-a^ph^\varepsilon){\bf{I}}_{2\times 2}
\end{array}
\right),
\end{align}
\begin{align}\label{3.18}
{\bf{D}}{\bf{A}}_2=\left(
\begin{array}{cc}
\rho^\varepsilon(v^\varepsilon+a^pg^\varepsilon) {\bf{I}}_{2\times 2}& [\rho^\varepsilon(a^p)^2-1]g^\varepsilon{\bf{I}}_{2\times2}\\
{[}\rho^\varepsilon(a^p)^2-1{]}g^\varepsilon{\bf{I}}_{2\times2}&  [1-\rho^\varepsilon(a^p)^2](v^\varepsilon-a^pg^\varepsilon){\bf{I}}_{2\times 2}
\end{array}
\right).
\end{align}
Furthermore, in order to guarantee both of the matrices ${\bf{D}}$ and ${\bf{DB}}^\varepsilon$ are positive definite, we need to impose some restrictions on $a^p$. Indeed, since there are no any boundary layers initially, that is, $a^p|_{t=0}=0$, then for any fixed $\delta>0$ small enough, there exists a $T_\delta\in [0, T_6]$, such that
\begin{align}\label{3.19}
\sup\limits_{t\in[0, T_\delta]}\|a^p\|_{L^\infty}^2\le\frac{4(\mu-\delta)(\kappa-\delta)}{(\mu+\rho^\varepsilon\kappa)^2-4\delta\rho^\varepsilon\kappa}\le\frac{4(\mu-\delta)(\kappa-\delta)}{(\mu+\underline{\rho}\kappa)^2-4\delta\underline{\rho}\kappa}.
\end{align}
Under this assumption, we immediately get that ${\bf{DB}}^\varepsilon$ is positive definite and it holds that for any vector ${\bf{X}}\in\mathbb{R}^4$
\begin{align}\label{3.20}
{\bf{DB}}^\varepsilon{\bf{X}}\cdot {\bf{X}}\ge \delta |{\bf{X}}|^2.
\end{align}
Moreover, we also have
\begin{align}\label{3.21}
1-\rho^\varepsilon(a^p)^2\ge\frac{(\mu-\rho^\varepsilon\kappa)^2+4\rho^\varepsilon\delta(\mu-\delta)}{(\mu+\rho^\varepsilon\kappa)^2-4\delta\rho^\varepsilon\kappa}\ge c_\delta>0,
\end{align}
which implies that ${\bf{D}}$ is also positive definite.

\subsection{The a priori estimate} The goal of this subsection is to define some energy functionals, give the a priori assumptions and establish the $L^\infty$ estimates of the density $\rho^\varepsilon$ together with its derivatives which will be used frequently in the next subsections. We start with the following energy functional.
\begin{align}\label{3.23}
N(t)&=\sum\limits_{|\alpha|\le 2}\sup\limits_{0\le s\le t}\|\varepsilon^{\frac\alpha 2}\partial_{\tau}^\alpha{\bf{U}}(s)\|_{L^2}^2
+\sup\limits_{0\le s\le t}\|\varepsilon^{\frac 12}y\partial_y{\bf{U}}(s)\|_{L^2}^2+\sum\limits_{|\beta|\le 1}\sup\limits_{0\le s\le t}\|\varepsilon^{\frac\beta 2}\mathcal{Z}^\beta\tilde{\rho}(s)\|_{L^2}^2\notag\\
&\quad+\sum\limits_{|\alpha|\le 2}\varepsilon\int_0^t \|\varepsilon^{\frac\alpha 2}\nabla\partial_{\tau}^\alpha{\bf{U}}(s)\|_{L^2}^2\;ds
+\varepsilon\int_0^t \|\varepsilon^{\frac 12}\nabla(y\partial_y{\bf{U}})(s)\|_{L^2}^2\;ds.
\end{align}
Here, $\partial_\tau=\partial_t, \partial_x$ denotes the tangential derivatives, and $\mathcal{Z}=\partial_t, \partial_x, y\partial_y, \sqrt{\varepsilon}\partial_y$ denotes the conormal derivatives. 

In what follows, we always assume there exist constants $C>0$ and $T>0$, such that
\begin{align}\label{3.24}
\|\varepsilon{\bf{U}}\|_{L_{txy}^\infty}\le C
\end{align}
holds for any $t\in [0, T]$. Then, according to \eqref{3.1} and \eqref{3.24}, it is direct to achieve that
\begin{align}\label{3.25}
\|({\bf{u}}^\varepsilon, {\bf{H}}^\varepsilon)\|_{L_{txy}^\infty}\le C.
\end{align}
Moreover, we also introduce the energy functional.
\begin{align}\label{3.26}
Q(t)&=\|\sqrt{\varepsilon} v\|_{L^\infty}+\int_0^t\|\varepsilon(\partial_x{\bf{u}}, \partial_x{\bf{H}})\|_{L^\infty}^2\;ds+\int_0^t\|\varepsilon^{\frac32}(\partial_t{\bf{u}}, \partial_t{\bf{H}})\|_{L^\infty}^2\;ds\notag\\
&\quad+\int_0^t \left(\varepsilon^{3}\|(\partial_{yy}u, \partial_{yy}h)\|_{L^2}^2+\varepsilon^{4}\|(\partial_{yy\tau}u, \partial_{yy\tau}h)\|_{L^2}^2\right)\;ds+\int_0^t\|\varepsilon^{\frac32}(\partial_yu, \partial_yh)\|_{L^\infty}^2\;ds.
\end{align}
In fact, $Q(t)$ can be controlled by $N(t)$ which is stated in the following Lemma.
\begin{lemma}
\label{L3.4}
For $Q(t)$ and $N(t)$ defined in \eqref{3.23} and \eqref{3.26} respectively, it holds that
\begin{align}\label{3.27}
Q(t)\lesssim 1+N(t)+N^2(t).
\end{align}
\end{lemma}
The proof relies on the standard Sobolev embedding inequalities, and we postpone it in Appendix A. From which and \eqref{3.1} we immediately deduce that
\begin{corollary}
\label{C3.5}
For $({\bf{u}}^\varepsilon, {\bf{H}}^\varepsilon)$ satisfies \eqref{3.1}, we have
\begin{align}\label{3.28}
\int_0^t\|(\partial_{\tau}{\bf{u}}^\varepsilon, \partial_{\tau}{\bf{H}}^\varepsilon)\|_{L^\infty}^2\;ds\lesssim Q(t)\lesssim 1+N(t)+N^2(t).
\end{align}
\end{corollary}
Since the proof is straightforward, we omit the details here. Moreover, we also have the following estimates for the density $\rho^\varepsilon$.
\begin{lemma}
\label{L3.6}
There exists a unique solution $\rho^\varepsilon$ of the system \eqref{1.1}, it follows that
\begin{align}\label{3.29}
\|\rho^\varepsilon\|_{L^\infty}\le C,\quad \|\mathcal{Z}\rho^\varepsilon\|_{L^\infty}\lesssim 1+Q(t)\lesssim 1+N(t)+N^2(t).
\end{align}
\end{lemma}
The proof will be presented in Appendix A as well. The main goal of this Section is to establish the following a priori energy estimates, whose proof is included in the subsequent subsections.
\begin{theorem}
\label{T3.7}
The initial boundary value problem \eqref{3.9} admits a unique solution $(\tilde{\rho}, {\bf{U}})(t, x, y)$,  and it holds that
\begin{align}\label{3.30}
N(t)+\varepsilon\int_0^t \|\nabla p^\varepsilon(s)\|_{L^2}^2\;ds\le C.
\end{align}
\end{theorem}

\subsection{Energy estimates of $\tilde{\rho}$ and ${\bf{U}}$}
This subsection is devoted to establishing the $L^2$ energy estimates of $\tilde{\rho}$ and ${\bf{U}}$.
\begin{proposition}
\label{P3.8}
For any fixed constant $\delta>0$ which satisfies \eqref{3.19}-\eqref{3.21}, there exists a unique solution $(\tilde{\rho}, {\bf{U}})(t, x, y)$ to the system \eqref{3.9}, such that, for any sufficiently small $\lambda>0$, it holds that
\begin{align}\label{3.31}
\|\tilde{\rho}(t, \cdot)\|_{L^2}+\|{\bf{U}}(t, \cdot)\|_{L^2}^2+\varepsilon\int_0^t \|\nabla {\bf{U}}(s, \cdot)\|_{L^2}^2\;ds\le
(\lambda+C_\lambda t)\mathcal{P}(N(t)),
\end{align}
where the positive constant $C_\lambda$ is independent of $\varepsilon$.
\end{proposition}

\begin{proof}
Multiplying the second equation in \eqref{3.9} by ${\bf{D}}$, taking the inner product with ${\bf{U}}$ and integrating the resulting equality over $[0, t]$, we have
\begin{align}\label{3.32}
&\frac12\langle {\bf{DS}}^\varepsilon{\bf{U}}, {\bf{U}}\rangle+\int_0^t\langle {\bf{DA}}_1^\varepsilon\partial_x{\bf{U}}, {\bf{U}}\rangle\;ds+\int_0^t\langle {\bf{DA}}_2^\varepsilon\partial_y{\bf{U}}, {\bf{U}}\rangle\;ds\notag\\
&\quad+\int_0^t\langle {\bf{DC}}^\varepsilon, {\bf{U}}\rangle\;ds
+\int_0^t \langle (p_x, p_y, 0, 0)^T, {\bf{U}}\rangle\;ds-\int_0^t \varepsilon\langle {\bf{DB}}^\varepsilon\Delta{\bf{U}}, {\bf{U}}\rangle\;ds\notag\\
&\quad=\frac12\|\sqrt{{\bf{DS}}^\varepsilon}{\bf{U}}(0, \cdot)\|_{L^2}^2+\frac12\int_0^t \langle\partial_t({\bf{DS}}^\varepsilon){\bf{U}}, {\bf{U}}\rangle\;ds+\int_0^t \langle{\bf{DE}}^\varepsilon, {\bf{U}}\rangle\;ds.
\end{align}
Next, we will handle term by term. By integration by parts and the divergence free condition, we obtain
\begin{align}\label{3.33}
\langle (p_x, p_y, 0, 0)^T, {\bf{U}}\rangle=0.
\end{align}
Since ${\bf{DS}}^\varepsilon$ is positive definite,
it follows that
\begin{align}\label{3.34}
\langle {\bf{DS}}^\varepsilon{\bf{U}}, {\bf{U}}\rangle\ge c_\delta\|{\bf{U}}(t, \cdot)\|_{L^2}^2.
\end{align}
To estimate the second term on the left hand side of \eqref{3.32}, it is helpful to separate it into two parts.
\begin{align*}
\int_0^t\langle {\bf{DA}}_1^\varepsilon\partial_x{\bf{U}}, {\bf{U}}\rangle\;ds=\int_0^t \langle({\bf{DA}}_1+{\bf{DA}}_1^p)\partial_x{\bf{U}}, {\bf{U}}\rangle\;ds=:I_1+I_2.
\end{align*}
By \eqref{3.14}, $I_2$ can be estimated directly.
\begin{align}\label{3.35}
|I_2|\le& \int_0^t \|{\bf{DA}}_1^p\|_{L^\infty} \sqrt{\varepsilon}\|\partial_x{\bf{U}}\|_{L^2}\|{\bf{U}}\|_{L^2}\;ds\notag\\
\le&\frac{\delta\varepsilon}{8}\int_0^t \|\nabla{\bf{U}}\|^2_{L^2}\;ds+C\int_0^t \|{\bf{U}}\|_{L^2}^2\;ds.
\end{align}
To estimate $I_1$, we use the formula of \eqref{3.18} and the integration by parts. And $\partial_x({\bf{DA}}_1)$ has the following expression,
\begin{align*}
\partial_x({\bf{D}}{\bf{A}}_1)&=\partial_x
\left(
\begin{array}{cc}
\rho^\varepsilon(u^\varepsilon+a^ph^\varepsilon) {\bf{I}}_{2\times 2}& [\rho^\varepsilon(a^p)^2-1]h^\varepsilon{\bf{I}}_{2\times2}\\
{[}\rho^\varepsilon(a^p)^2-1{]}h^\varepsilon{\bf{I}}_{2\times2}&  [1-\rho^\varepsilon(a^p)^2](u^\varepsilon-a^ph^\varepsilon){\bf{I}}_{2\times 2}
\end{array}
\right).
\end{align*}
Taking \eqref{3.25} and \eqref{3.29} into account, we deduce that
\begin{align*}
\|\partial_x({\bf{DA}}_1)\|_{L^\infty}\lesssim 1+N(t)+N^2(t)+\|\partial_x({\bf{u}}^\varepsilon, {\bf{H}}^\varepsilon)\|_{L^\infty}.
\end{align*}
From which we obtain for any $\lambda>0$ sufficiently small, that
\begin{align}\label{3.36}
|I_1|\le&\lambda\int_0^t \|\partial_x({\bf{u}}^\varepsilon, {\bf{H}}^\varepsilon)\|_{L^\infty}^2\;ds+C_\lambda\int_0^t \left(1+N(t)+N^2(t)+\|{\bf{U}}\|^2_{L^2}\right){\|\bf{U}}\|^2_{L^2}\;ds\notag\\
\le& \lambda Q(t)+C_\lambda\int_0^t \left(1+N(t)+N^2(t)+\|{\bf{U}}\|^2_{L^2}\right){\|\bf{U}}\|^2_{L^2}\;ds.
\end{align}
We continue to estimate the third one on the left hand side of \eqref{3.32}, which is also divided into two parts.
\begin{align*}
\int_0^t\langle {\bf{DA}}_2^\varepsilon\partial_y{\bf{U}}, {\bf{U}}\rangle\;ds=\int_0^t \langle({\bf{DA}}_2+{\bf{DA}}_2^p)\partial_y{\bf{U}}, {\bf{U}}\rangle\;ds=:I_3+I_4.
\end{align*}
By the same trick as \eqref{3.35}, it yields that
\begin{align}\label{3.37}
|I_4|\le& \int_0^t \|{\bf{DA}}_2^p\|_{L^\infty} \sqrt{\varepsilon}\|\partial_y{\bf{U}}\|_{L^2}\|{\bf{U}}\|_{L^2}\;ds\notag\\
\le&\frac{\delta\varepsilon}{8}\int_0^t \|\nabla{\bf{U}}\|^2_{L^2}\;ds+C\int_0^t \|{\bf{U}}\|_{L^2}^2\;ds.
\end{align}
Notice that
\begin{align*}
\partial_y({\bf{D}}{\bf{A}}_{2})&=
\partial_y\left(
\begin{array}{cc}
\rho^\varepsilon (v^\varepsilon +a^pg^\varepsilon) {\bf{I}}_{2\times 2}& [\rho^\varepsilon(a^p)^2-1]g^\varepsilon{\bf{I}}_{2\times2}\\
{[}\rho^\varepsilon(a^p)^2-1{]}g^\varepsilon{\bf{I}}_{2\times2}&  [1-\rho^\varepsilon(a^p)^2](v^\varepsilon-a^pg^\varepsilon){\bf{I}}_{2\times 2}
\end{array}
\right),
\end{align*}
we only show the estimate of $\partial_y\left(\rho^\varepsilon (v^\varepsilon+ a^pg^\varepsilon)\right)$ because other terms can be treated by the same arguments. Since
\begin{align*}
\partial_y\left(\rho^\varepsilon (v^\varepsilon+a^pg^\varepsilon)\right)=(y\partial_y\rho^\varepsilon )\left(a^p\frac{g^\varepsilon}{y}+\frac{v^\varepsilon}{y}\right)+\rho^\varepsilon(y\partial_ya^p)\frac{g^\varepsilon}{y}+\rho^\varepsilon (\partial_yv^\varepsilon+a^p \partial_yg^\varepsilon).
\end{align*}
It follows from \eqref{3.29} that
\begin{align*}
\|y\partial_y\rho^\varepsilon\|_{L^\infty}\lesssim 1+N(t)+N^2(t),\quad \|(\rho^\varepsilon, y\partial_y a^p,a^p)\|_{L^\infty}=O(1).
\end{align*}
Then by H\"ardy's trick and the divergence free conditions, we immediately get
\begin{align*}
\|\partial_y\left(\rho^\varepsilon(v^\varepsilon+ a^pg^\varepsilon)\right)\|_{L^\infty}\le C\left(1+N(t)+N^2(t)\right)\left(\|\partial_xu^\varepsilon\|_{L^\infty}+\|\partial_xh^\varepsilon\|_{L^\infty}\right).
\end{align*}
From which and integration by parts, we infer that
\begin{align}\label{3.38}
|I_3|\le \lambda Q(t)+C_\lambda\int_0^t \left(1+N^2(t)+N^4(t)\right)\|{\bf{U}}\|^4_{L^2}\;ds.
\end{align}
Also, using the fact that $\|\partial_t\rho^\varepsilon\|_{L^\infty}\lesssim 1+N(t)+N^2(t)$, it holds
\begin{align}\label{3.39}
\left|\int_0^t \langle\partial_t({\bf{DS}}^\varepsilon){\bf{U}}, {\bf{U}}\rangle\;ds\right|\le C\int_0^t \left(1+N(t)+N^2(t)\right)\|{\bf{U}}\|^2_{L^2}\;ds.
\end{align}
Next, thanks to Propositions $\ref{P3.2}$ and $\ref{vector}$, we deduce that
\begin{align}\label{3.40}
\left|\int_0^t\langle {\bf{DE}}^\varepsilon-{\bf{DC}}^\varepsilon, {\bf{U}}\rangle\;ds\right|\le C\int_0^t \|{\bf{U}}\|^2_{L^2}\;ds+Ct.
\end{align}
Finally, it remains to estimate the diffusion term. By integration by parts again, we have
\begin{align*}
-\varepsilon\langle {\bf{DB}}^\varepsilon\Delta{\bf{U}}, {\bf{U}}\rangle=\varepsilon\langle {\bf{DB}}^\varepsilon\nabla{\bf{U}}, \nabla{\bf{U}}\rangle+\varepsilon\langle {\nabla(\bf{DB}}^\varepsilon)\nabla{\bf{U}}, {\bf{U}}\rangle=:I_5+I_6.
\end{align*}
On one hand, recall \eqref{3.20}, it yields that
\begin{align}\label{3.41}
I_5\ge\delta\varepsilon\|\nabla{\bf{U}}\|_{L^2}^2.
\end{align}
On the other hand, it is direct to verify that
\begin{align*}
\sqrt{\varepsilon}\|\nabla({\bf{DB}}^\varepsilon)\|_{L^\infty}\lesssim 1+N(t)+N^2(t),
\end{align*}
which implies that
\begin{align}\label{3.42}
|I_6|\le \frac{\delta\varepsilon}{8}\|\nabla{\bf{U}}\|_{L^2}^2+C\left(1+N^2(t)+N^4(t)\right)\|{\bf{U}}\|^2_{L^2}.
\end{align}
Below, we start to establish the estimate of $\tilde{\rho}$. Taking inner product on the first equation in \eqref{3.9} with $\tilde{\rho}$, integrating it over $[0, t]$, we obtain by integration by parts and Proposition \ref{vector} that
\begin{align}\label{344}
\|\tilde{\rho}\|_{L^2}^2\le C\int_0^t \|\tilde{\rho}\|_{L^2}^2\;ds+\frac{\delta\varepsilon}{8}\int_0^t \|\nabla{\bf{U}}\|_{L^2}^2\;ds+C\int_0^t \|{\bf{U}}\|_{L^2}^2\;ds+Ct.
\end{align}
Plugging \eqref{3.33}-\eqref{344} into \eqref{3.32}, we arrive at
\begin{align}\label{3.43}
\|\tilde{\rho}\|_{L^2}^2&+\|{\bf{U}}(t, \cdot)\|_{L^2}^2+\varepsilon\int_0^t \|\nabla {\bf{U}}(s, \cdot)\|_{L^2}^2\;ds\notag\\
&\le C_\lambda\int_0^t (1+N(t)+N^2(t)+N^3(t)+N^4(t))(\|{\bf{U}}\|_{L^2}^2+\|{\bf{U}}\|_{L^2}^4)ds\notag\\
&\quad+2\lambda Q(t)+C\int_0^t \|\tilde{\rho}\|_{L^2}^2ds+Ct\notag\\
&\lesssim (C_\lambda t+\lambda)\mathcal{P}(N(t)),
\end{align}
where Lemma $\ref{L3.4}$ is used in the last line. Thus the proof of Proposition $\ref{P3.8}$ is done.
\end{proof}

\subsection{First order tangential derivative estimates of ${\bf{U}}$}
This subsection is devoted to establishing the energy estimates of $\partial_\tau{\bf{U}}$.
\begin{proposition}
\label{P3.9}
For any fixed constant $\delta>0$ which satisfies \eqref{3.19}-\eqref{3.21}, there exists a unique solution ${\bf{U}}(t, x, y)$ to the system \eqref{3.9}, it holds that for any sufficiently small $\lambda>0$,
\begin{align}\label{3.44}
\|\sqrt{\varepsilon}\partial_\tau{\bf{U}}(t, \cdot)\|_{L^2}^2+\varepsilon^2\int_0^t \|\nabla {\bf{U}}_\tau(s, \cdot)\|_{L^2}^2\;ds\le C+(C_\lambda t+\lambda)\mathcal{P}(N(t)),
\end{align}
where the positive constant $C$  is independent of $\varepsilon$.
\end{proposition}

\begin{proof}
Applying the first order tangential derivative operator $\partial_\tau(\tau=t, x)$ on the second equation in \eqref{3.9}, multiplying it by ${\bf{D}}$ from the left and taking the inner product on it with $\varepsilon\partial_\tau{\bf{U}}$, and integrating the resulting equality over $[0, t]$, we have
\begin{align}\label{3.45}
&\frac12\langle {\bf{DS}}^\varepsilon\partial_\tau{\bf{U}}, \varepsilon\partial_\tau{\bf{U}}\rangle
+\int_0^t\langle {\bf{DA}}_1^\varepsilon\partial_x\partial_\tau{\bf{U}}, \varepsilon\partial_\tau{\bf{U}}\rangle\;ds
+\int_0^t\langle {\bf{DA}}_2^\varepsilon\partial_y\partial_\tau{\bf{U}}, \varepsilon\partial_\tau{\bf{U}}\rangle\;ds\notag\\
&\quad+\int_0^t\langle {\bf{D}}\partial_\tau{\bf{C}}^\varepsilon, \varepsilon\partial_\tau{\bf{U}}\rangle\;ds
+\int_0^t \langle (\partial_\tau p_x, \partial_\tau p_y, 0, 0)^T, \varepsilon\partial_\tau{\bf{U}}\rangle\;ds
-\int_0^t \varepsilon\langle {\bf{DB}}^\varepsilon\Delta\partial_\tau{\bf{U}}, \varepsilon\partial_\tau{\bf{U}}\rangle\;ds\notag\\
&\quad+\int_0^t \left\langle{\bf{D}}\left(\partial_\tau{\bf{S}}^\varepsilon\partial_t{\bf{U}}+\partial_\tau{\bf{A}}_1^\varepsilon\partial_x{\bf{U}}+\partial_\tau{\bf{A}}_2^\varepsilon\partial_y{\bf{U}}\right), \varepsilon\partial_\tau{\bf{U}}\right\rangle\;ds-
\int_0^t \varepsilon\langle {\bf{D}}\partial_\tau{\bf{B}}^\varepsilon\Delta{\bf{U}}, \varepsilon\partial_\tau{\bf{U}}\rangle\;ds\notag\\
&\quad=\frac12\varepsilon\|\sqrt{{\bf{DS}}^\varepsilon}\partial_\tau{\bf{U}}(0, \cdot)\|_{L^2}^2+\frac12\int_0^t \langle\partial_t({\bf{DS}}^\varepsilon)\partial_\tau{\bf{U}}, \varepsilon\partial_\tau{\bf{U}}\rangle\;ds
+\int_0^t \langle{\bf{D}}\partial_\tau{\bf{E}}^\varepsilon, \varepsilon\partial_\tau{\bf{U}}\rangle\;ds.
\end{align}
By integration by parts, the divergence free condition and \eqref{3.21}, one has
\begin{align}\label{3.46}
\langle (\partial_\tau p_x, \partial_\tau p_y, 0, 0)^T, \varepsilon\partial_\tau{\bf{U}}\rangle=0
\end{align}
and
\begin{align}\label{3.47}
\langle {\bf{DS}}^\varepsilon\partial_\tau{\bf{U}}, \varepsilon\partial_\tau{\bf{U}}\rangle\ge c_\delta\|\sqrt{\varepsilon}\partial_\tau{\bf{U}}(t, \cdot)\|_{L^2}^2.
\end{align}
Then we separate the second term in \eqref{3.45} into two parts.
\begin{align*}
\int_0^t\langle {\bf{DA}}_1^\varepsilon\partial_x\partial_\tau{\bf{U}}, \varepsilon\partial_\tau{\bf{U}}\rangle\;ds=\int_0^t \langle({\bf{DA}}_1+{\bf{DA}}_1^p)\partial_x\partial_\tau{\bf{U}},\varepsilon \partial_\tau{\bf{U}}\rangle\;ds=:J_1+J_2.
\end{align*}
With the same procedure as \eqref{3.35}, it holds
\begin{align}\label{3.48}
|J_2|\le\frac{\delta\varepsilon^2}{8}\int_0^t \|\nabla{\bf{U}}_\tau\|^2_{L^2}\;ds+C\int_0^t \|\sqrt{\varepsilon}\partial_\tau{\bf{U}}\|_{L^2}^2\;ds.
\end{align}
Also, integration by parts for $J_1$ implies that
\begin{align}\label{3.49}
|J_1|\le \lambda Q(t)+C\int_0^t (1+N(t)+N^2(t)+\|\sqrt{\varepsilon}\partial_\tau{\bf{U}}\|^2_{L^2}){\|\sqrt{\varepsilon}\partial_\tau\bf{U}}\|_{L^2}^2\;ds.
\end{align}
Similar derivation yields that
\begin{align}\label{3.50}
\left|\int_0^t\langle {\bf{DA}}_2^\varepsilon\partial_y\partial_\tau{\bf{U}}, \varepsilon\partial_\tau{\bf{U}}\rangle\;ds\right|&\le \frac{\delta\varepsilon^2}{8}\int_0^t \|\nabla{\bf{U}}_\tau\|^2_{L^2}\;ds+\lambda Q(t)+C\int_0^t \|\sqrt{\varepsilon}\partial_\tau{\bf{U}}\|_{L^2}^2\;ds\notag\\
&\quad +C_\lambda\int_0^t \left(1+N^2(t)+N^4(t)\right){\|\sqrt{\varepsilon}\partial_\tau\bf{U}}\|_{L^2}^4\;ds.
\end{align}
Next, thanks to Propositions $\ref{P3.2}$, $\ref{vector}$ and \eqref{3.29}, we also have
\begin{align}\label{3.51}
\left|\int_0^t\langle {\bf{D}}\partial_\tau{\bf{E}}^\varepsilon, \varepsilon\partial_\tau{\bf{U}}\rangle\;ds\right|\le C\int_0^t\|\sqrt{\varepsilon}\partial_\tau{\bf{U}}\|_{L^2}^2\;ds+Ct,
\end{align}
\begin{align}\label{3.52}
\left|\int_0^t \langle\partial_t({\bf{DS}}^\varepsilon)\partial_\tau{\bf{U}}, \varepsilon\partial_\tau{\bf{U}}\rangle\;ds\right|\le C\int_0^t \left(1+N(t)+N^2(t)\right)\|\sqrt{\varepsilon}\partial_\tau{\bf{U}}\|_{L^2}^2\;ds
\end{align}
and
\begin{align}\label{3.53}
\left|\int_0^t\langle {\bf{D}}\partial_\tau{\bf{C}}^\varepsilon, \varepsilon\partial_\tau{\bf{U}}\rangle\;ds\right|\le C\int_0^t (1+N^2(t)+N^4(t)){\|\bf{U}}\|_{L^2}^2\;ds+C\int_0^t \|\sqrt{\varepsilon}\partial_\tau{\bf{U}}\|_{L^2}^2\;ds.
\end{align}
As for the diffusion term, it follows from integration by parts that
\begin{align*}
-\varepsilon\langle {\bf{DB}}^\varepsilon\Delta\partial_\tau{\bf{U}}, \varepsilon\partial_\tau{\bf{U}}\rangle=\varepsilon\langle {\bf{DB}}^\varepsilon\nabla{\bf{U}}_\tau, \varepsilon\nabla{\bf{U}}_\tau\rangle+\varepsilon\langle {\nabla(\bf{DB}}^\varepsilon)\nabla{\bf{U}}_\tau, \varepsilon\partial_\tau{\bf{U}}\rangle.
\end{align*}
According to the arguments in \eqref{3.41} and \eqref{3.42}, it implies
\begin{align}\label{3.54}
-\varepsilon\langle {\bf{DB}}^\varepsilon\Delta\partial_\tau{\bf{U}}, \varepsilon\partial_\tau{\bf{U}}\rangle\ge \frac{7\delta\varepsilon^2}{8}\|\nabla{\bf{U}}_\tau\|^2_{L^2}-C\left(1+N^2(t)+N^4(t)\right)\|\sqrt{\varepsilon}\partial_\tau{\bf{U}}\|_{L^2}^2.
\end{align}
It remains to estimate the third line in \eqref{3.45}. First
\begin{align}\label{3.55}
&\left|\int_0^t \left\langle{\bf{D}}\left(\partial_\tau{\bf{S}}^\varepsilon\partial_t{\bf{U}}+\partial_\tau{\bf{A}}_1^\varepsilon\partial_x{\bf{U}}\right), \varepsilon\partial_\tau{\bf{U}}\right\rangle\;ds\right|\notag\\
\le &\lambda Q(t)+C_\lambda\int_0^t \left(1+N(t)+N^2(t)+\|\sqrt{\varepsilon}\partial_\tau{\bf{U}}\|^2_{L^2}\right){\|\sqrt{\varepsilon}\partial_\tau\bf{U}}\|_{L^2}^2\;ds.
\end{align}
To estimate the second term, rewrite ${\bf{A}}_2^\varepsilon$ as follows.
\begin{align*}
{\bf{A}}_2^\varepsilon={\bf{A}}_2^a+\sqrt{\varepsilon}{\bf{A}}_2^p+\varepsilon^{\frac32}{\bf{A}}_2^0,
\end{align*}
where
\begin{align*}
{\bf{A}}_2^a=\left(
\begin{array}{cc}
\rho^\varepsilon(v^a+a^pg^a) {\bf{I}}_{2\times 2}& [\rho^\varepsilon(a^p)^2-1]g^a{\bf{I}}_{2\times2}\\
-g^a{\bf{I}}_{2\times2}& (v^a-a^pg^a){\bf{I}}_{2\times 2}
\end{array}
\right)
\end{align*}
and
\begin{align*}
{\bf{A}}_2^0=\left(
\begin{array}{cc}
\rho^\varepsilon(v+a^pg) {\bf{I}}_{2\times 2}& [\rho^\varepsilon(a^p)^2-1]g{\bf{I}}_{2\times2}\\
-g{\bf{I}}_{2\times2}& (v-a^pg){\bf{I}}_{2\times 2}
\end{array}
\right).
\end{align*}
Direct calculations lead to
\begin{align}\label{3.56}
\left|\int_0^t \left\langle{\bf{D}}\left(\partial_\tau{\bf{A}}_2^a+\sqrt{\varepsilon}\partial_\tau{\bf{A}}_2^p\right)\partial_y{\bf{U}}, \varepsilon\partial_\tau{\bf{U}}\right\rangle\;ds\right|\le C\varepsilon\int_0^t \|\nabla{\bf{U}}\|_{L^2}^2\;ds+ C\int_0^t \|\sqrt{\varepsilon}\partial_\tau{\bf{U}}\|_{L^2}^2\;ds.
\end{align}
And for any $\lambda>0$ sufficiently small, we deduce that
\begin{align}\label{3.57}
&\left|\int_0^t \left\langle{\bf{D}}\partial_\tau{\bf{A}}_2^0\left(\varepsilon^{\frac32}\partial_y{\bf{U}}\right), \varepsilon\partial_\tau{\bf{U}}\right\rangle\;ds\right|\notag\\
\le& \lambda Q(t)+C_\lambda\int_0^t \left(1+N^2(t)+N^4(t)\right)\left({\|\bf{U}}\|_{L^2}^4+{\|\sqrt{\varepsilon}\partial_\tau\bf{U}}\|_{L^2}^4\right)\;ds.
\end{align}
Recall the definition of ${\bf{B}}^\varepsilon$, we find
\begin{align*}
-\int_0^t \varepsilon\langle {\bf{D}}\partial_\tau{\bf{B}}^\varepsilon\Delta{\bf{U}}, \varepsilon\partial_\tau{\bf{U}}\rangle\;ds=\kappa\int_0^t \varepsilon\langle\partial_\tau(\rho^\varepsilon a^p)\Delta\tilde{h}, \varepsilon\partial_\tau\tilde{u}\rangle\;ds+\kappa\int_0^t \varepsilon\langle\partial_\tau(\rho^\varepsilon a^p)\Delta\tilde{g}, \varepsilon\partial_\tau\tilde{v}\rangle\;ds.
\end{align*}
It is straightforward to obtain that
\begin{align}\label{3.58}
&\left|\kappa\int_0^t \varepsilon\langle\partial_\tau(\rho^\varepsilon a^p)\partial_{xx}\tilde{h}, \varepsilon\partial_\tau\tilde{u}\rangle\;ds+\kappa\int_0^t \varepsilon\langle\partial_\tau(\rho^\varepsilon a^p)\Delta\tilde{g}, \varepsilon\partial_\tau\tilde{v}\rangle\;ds\right|\notag\\
\le&\frac{\delta\varepsilon^2}{8}\int_0^t \|\nabla{\bf{U}}_\tau\|^2_{L^2}\;ds+C\int_0^t \left(1+N^2(t)+N^4(t)\right)\|\sqrt{\varepsilon}\partial_\tau{\bf{U}}\|_{L^2}^2\;ds.
\end{align}
And it follows from H\"older's inequality that
\begin{align}\label{3.59}
&\left|\kappa\int_0^t \varepsilon\langle\partial_\tau(\rho^\varepsilon a^p)\partial_{yy}\tilde{h}, \varepsilon\partial_\tau\tilde{u}\rangle\;ds\right|\notag\\
\le& \lambda\int_0^t \|\varepsilon^{\frac32}\partial_{yy}\tilde{h}\|_{L^2}^2\;ds+C\int_0^t \left(1+N^2(t)+N^4(t)\right)\|\sqrt{\varepsilon}\partial_\tau{\bf{U}}\|_{L^2}^2\;ds\notag\\
\le& \lambda Q(t)+C_\lambda\int_0^t \left(1+N^2(t)\right)\|\sqrt{\varepsilon}\partial_\tau{\bf{U}}\|_{L^2}^2\;ds.
\end{align}
Combine \eqref{3.46}-\eqref{3.59} together and in virtue of \eqref{3.27}, \eqref{3.31}, we conclude that
\begin{align}\label{3.60}
\|\sqrt{\varepsilon}\partial_\tau{\bf{U}}(t, \cdot)\|_{L^2}^2&+\varepsilon^2\int_0^t \|\nabla {\bf{U}}_\tau(s, \cdot)\|_{L^2}^2\;ds\notag\\
&\le C_\lambda\int_0^t \left(1+N(t)+N^2(t)+N^4(t)\right)(\|\sqrt{\varepsilon\partial_\tau{\bf{U}}}\|_{L^2}^2+\|\sqrt{\varepsilon}\partial_\tau{\bf{U}}\|_{L^2}^4)\;ds\notag\\
&\quad+5\lambda Q(t)+C(t+\lambda)\mathcal{P}(N(t))+C\notag\\
&\le C+(C_\lambda t+\lambda)\mathcal{P}(N(t)).
\end{align}
The proof of Proposition $\ref{P3.9}$ is done.
\end{proof}

\subsection{First order conormal derivative estimates of ${\bf{U}}$}
This subsection is devoted to establishing the energy estimates of $y\partial_y{\bf{U}}$. Denote $\mathcal{Z}_0=y\partial_y$ for representation convenience.
\begin{proposition}
\label{P3.10}
For any fixed constant $\delta>0$ which satisfies \eqref{3.19}-\eqref{3.21}, there exists a unique solution ${\bf{U}}(t, x, y)$ to the system \eqref{3.9}. Moreover, for any $\lambda>0$ small enough, it holds that
\begin{align}\label{3.61}
&\|\sqrt{\varepsilon}\mathcal{Z}_0{\bf{U}}(t, \cdot)\|_{L^2}^2+\varepsilon^2\int_0^t \|\nabla \mathcal{Z}_0{\bf{U}}(s, \cdot)\|_{L^2}^2\;ds\nonumber\\
\le& C+(C_\lambda t+\lambda)\mathcal{P}(N(t))+C\varepsilon\int_0^t (\|\partial_yp\|_{L^2}^2+\|\mathcal{Z}_0 p\|_{L^2}^2)\;ds,
\end{align}
where the positive constant $C$ is independent of $\varepsilon$.
\end{proposition}

\begin{proof}
Applying the operator $\mathcal{Z}_0$ on the second equation in \eqref{3.9}, multiplying it by ${\bf{D}}$ from the left, taking inner product on the resulting equality with $\varepsilon \mathcal{Z}_0{\bf{U}}$, and integrating it over $[0, t]$, we have
\begin{align}\label{3.62}
&\frac12\langle {\bf{DS}}^\varepsilon \mathcal{Z}_0{\bf{U}}, \varepsilon \mathcal{Z}_0{\bf{U}}\rangle
+\int_0^t\langle {\bf{DA}}_1^\varepsilon\partial_x\mathcal{Z}_0{\bf{U}}, \varepsilon\mathcal{Z}_0{\bf{U}}\rangle\;ds
+\int_0^t\langle {\bf{DA}}_2^\varepsilon\partial_y\mathcal{Z}_0{\bf{U}}, \varepsilon \mathcal{Z}_0{\bf{U}}\rangle\;ds\notag\\
&\quad+\int_0^t\langle {\bf{D}}\mathcal{Z}_0{\bf{C}}^\varepsilon, \varepsilon\mathcal{Z}_0{\bf{U}}\rangle\;ds
+\int_0^t \langle (\partial_x\mathcal{Z}_0 p, \partial_y\mathcal{Z}_0 p, 0, 0)^T, \varepsilon\mathcal{Z}_0{\bf{U}}\rangle\;ds\notag\\
&\quad +\int_0^t\langle {\bf{DA}}_2^\varepsilon [\mathcal{Z}_0, \partial_y]{\bf{U}}, \varepsilon \mathcal{Z}_0{\bf{U}}\rangle\;ds
+\int_0^t \langle  [\mathcal{Z}_0, \partial_y] p, \varepsilon\mathcal{Z}_0\tilde{v}\rangle\;ds\notag\\
&\quad-\int_0^t \varepsilon\langle {\bf{DB}}^\varepsilon\Delta\mathcal{Z}_0{\bf{U}}, \varepsilon\mathcal{Z}_0{\bf{U}}\rangle\;ds
-\int_0^t \varepsilon\langle {\bf{DB}}^\varepsilon [\mathcal{Z}_0, \partial_y^2]{\bf{U}}, \varepsilon\mathcal{Z}_0{\bf{U}}\rangle\;ds\notag\\
&\quad+\int_0^t \left\langle{\bf{D}}\left(\mathcal{Z}_0{\bf{S}}^\varepsilon\partial_t{\bf{U}}+\mathcal{Z}_0{\bf{A}}_1^\varepsilon\partial_x{\bf{U}}+\mathcal{Z}_0{\bf{A}}_2^\varepsilon\partial_y{\bf{U}}\right), \varepsilon\mathcal{Z}_0{\bf{U}}\right\rangle\;ds\notag\\
&\quad-\int_0^t \varepsilon\langle {\bf{D}}\mathcal{Z}_0{\bf{B}}^\varepsilon\Delta{\bf{U}}, \varepsilon\mathcal{Z}_0{\bf{U}}\rangle\;ds
-\int_0^t \langle{\bf{D}}\mathcal{Z}_0{\bf{E}}^\varepsilon, \varepsilon\mathcal{Z}_0{\bf{U}}\rangle\;ds\notag\\
&\quad=\frac12\varepsilon\|\sqrt{{\bf{DS}}^\varepsilon}\mathcal{Z}_0{\bf{U}}(0, \cdot)\|_{L^2}^2+\frac12\int_0^t \langle\partial_t({\bf{DS}}^\varepsilon)\mathcal{Z}_0{\bf{U}}, \varepsilon\mathcal{Z}_0{\bf{U}}\rangle\;ds.
\end{align}
By integration by parts and the divergence free condition, we have
\begin{align*}
\langle (\partial_x \mathcal{Z}_0p, \partial_y\mathcal{Z}_0 p, 0, 0)^T, \varepsilon\mathcal{Z}_0{\bf{U}}\rangle=\langle \mathcal{Z}_0p, \varepsilon[\mathcal{Z}_0, \partial_y]\tilde{v}\rangle.
\end{align*}
Since
\begin{align}\label{3.63}
[\mathcal{Z}_0, \partial_y]=-\partial_y,
\end{align}
it follows that
\begin{align}\label{3.64}
\langle \mathcal{Z}_0p, \varepsilon[\mathcal{Z}_0, \partial_y]\tilde{v}\rangle\le C\varepsilon\int_0^t\|\partial_y\tilde{v}\|_{L^2}^2\;ds+C\varepsilon\int_0^t \|\mathcal{Z}_0 p\|_{L^2}^2\;ds.
\end{align}
Similarly, we can also deduce that
\begin{align}\label{3.65}
&\left|\int_0^t\langle {\bf{DA}}_2^\varepsilon [\mathcal{Z}_0, \partial_y]{\bf{U}}, \varepsilon \mathcal{Z}_0{\bf{U}}\rangle\;ds\right|
+\left|\int_0^t \langle  [\mathcal{Z}_0, \partial_y] p, \varepsilon\mathcal{Z}_0\tilde{v}\rangle\;ds\right|\notag\\
\le& C\varepsilon\int_0^t\|\partial_y{\bf{U}}\|_{L^2}^2\;ds+C\varepsilon\int_0^t \|\partial_yp\|_{L^2}^2\;ds+C\int_0^t\|\sqrt{\varepsilon}\mathcal{Z}_0{\bf{U}}\|_{L^2}^2\;ds.
\end{align}
On the other hand, the following fact holds
\begin{align*}
[\mathcal{Z}_0, \partial_y^2]=-2\partial_y^2,
\end{align*}
from which we infer that
\begin{align}\label{3.66}
&\left|\int_0^t \varepsilon\langle {\bf{DB}}^\varepsilon [\mathcal{Z}_0, \partial_y^2]{\bf{U}}, \varepsilon\mathcal{Z}_0{\bf{U}}\rangle\;ds\right|\notag\\
\le& \lambda\varepsilon^3\int_0^t \|\partial_{yy}(\tilde{u}, \tilde{h})\|_{L^2}^2\;ds+C\varepsilon^2\int_0^t \|\nabla(\tilde{v}_x, \tilde{g}_x)\|_{L^2}^2\;ds
+C\int_0^t \|\sqrt{\varepsilon}\mathcal{Z}_0{\bf{U}}\|_{L^2}^2\;ds.
\end{align}
Combine all the estimates together, we conclude
\begin{align}\label{3.67}
&\|\sqrt{\varepsilon}\mathcal{Z}_0{\bf{U}}(t, \cdot)\|_{L^2}^2+\varepsilon^2\int_0^t \|\nabla\mathcal{Z}_0 {\bf{U}}_\tau(s, \cdot)\|_{L^2}^2\;ds\nonumber\\
\lesssim& 1+(C_\lambda t+\lambda)\mathcal{P}(N(t))+C\varepsilon\int_0^t (\|\partial_yp\|_{L^2}^2+\|\mathcal{Z}_0 p\|_{L^2}^2)\;ds.
\end{align}
The proof of Proposition $\ref{P3.10}$ is complete.
\end{proof}

\subsection{First order derivative of $\tilde{\rho}$}\label{Sec3.7}
Now, we establish the estimates of $\mathcal{Z}\tilde{\rho}$ where $\mathcal{Z}$ is defined in \eqref{3.23}.
\begin{proposition}
\label{P3.11}
For any fixed constant $\delta>0$ which satisfies \eqref{3.19}-\eqref{3.21}, there exists a unique solution $\tilde{\rho}(t, x, y)$ to the system \eqref{3.9}. Moreover, for any $\lambda>0$ small enough, it holds that
\begin{align}\label{0}
\|\sqrt{\varepsilon}\mathcal{Z}\tilde{\rho}(t, \cdot)\|_{L^2}^2\le C+(C_\lambda t+\lambda)\mathcal{P}(N(t)),
\end{align}
where the positive constant $C$ is independent of $\varepsilon$.
\end{proposition}
\begin{proof}
Applying the derivative operator $\mathcal{Z}$ on the first equation in \eqref{3.9}, taking the inner product on it with $\varepsilon\mathcal{Z}\tilde{\rho}$ and integrating it over $[0, t]$, after integration by parts we have
\begin{align}\label{1}
\|\sqrt{\varepsilon}\mathcal{Z}\tilde{\rho}\|_{L^2}^2&= \|\sqrt{\varepsilon}\mathcal{Z}\tilde{\rho}_0\|_{L^2}^2+\int_0^t\langle\mathcal{Z}{\bf{u}}^\varepsilon\cdot\nabla\tilde{\rho}, \varepsilon\mathcal{Z}\tilde{\rho}\rangle\;ds+\int_0^t \langle v^\varepsilon [y\partial_y, \partial_y]\tilde{\rho}, \varepsilon\mathcal{Z}\tilde{\rho}\rangle\;ds\\
&\quad-\int_0^t \langle\mathcal{Z}{\bf{C}}^0, \varepsilon\mathcal{Z}\tilde{\rho}\rangle\;ds-\int_0^t \langle\mathcal{Z}[c^p\kappa\varepsilon(\partial_y\tilde{h}-\partial_x\tilde{g})], \varepsilon\mathcal{Z}\tilde{\rho}\rangle\;ds
+\int_0^t \langle\mathcal{Z}{\bf{E}}^0, \varepsilon\mathcal{Z}\tilde{\rho}\rangle\;ds\notag.
\end{align}
Notice that the third term on the right hand side of \eqref{1} appears only in the case that $\mathcal{Z}=y\partial_y$ and by \eqref{3.63} and H\"ardy's trick, we have
\begin{align}\label{8}
\left|\int_0^t \langle v^\varepsilon [y\partial_y, \partial_y]\tilde{\rho}, \varepsilon\mathcal{Z}\tilde{\rho}\rangle\;ds\right|\le \lambda Q(t)+C_\lambda\int_0^t (1+\|\sqrt{\varepsilon}\mathcal{Z}\tilde{\rho}\|_{L^2}^2)\|\sqrt{\varepsilon}\mathcal{Z}\tilde{\rho}\|_{L^2}^2\;ds.
\end{align}
Based on Propositions $\ref{P3.2}$ and $\ref{vector}$, it is direct to estimate
\begin{align}\label{2}
&\left|\int_0^t \langle\mathcal{Z}{\bf{C}}^0, \varepsilon\mathcal{Z}\tilde{\rho}\rangle\;ds\right|+\left|\int_0^t \langle\mathcal{Z}{\bf{E}}^0, \varepsilon\mathcal{Z}\tilde{\rho}\rangle\;ds\right|\\
\le&C\varepsilon\int_0^t\|\nabla{\bf{U}}\|_{L^2}^2\;ds+C\int_0^t\|{\bf{U}}\|_{L^2}^2\;ds+C\int_0^t\|\sqrt{\varepsilon}\partial_\tau{\bf{U}}\|_{L^2}^2\;ds
+C\int_0^t\|\sqrt{\varepsilon}\mathcal{Z}\tilde{\rho}\|_{L^2}^2\;ds+C.\notag\end{align}
Thanks to $\|\sqrt{\varepsilon}\mathcal{Z}c^p\|_{L^\infty}\le C$, we can obtain
\begin{align}\label{3}
\left|\int_0^t \langle(\mathcal{Z}c^p)\kappa\varepsilon(\partial_y\tilde{h}-\partial_x\tilde{g}), \varepsilon\mathcal{Z}\tilde{\rho}\rangle\;ds\right|\le
C\varepsilon\int_0^t\|\nabla{\bf{U}}\|_{L^2}^2\;ds+C\int_0^t\|\sqrt{\varepsilon}\mathcal{Z}\tilde{\rho}\|_{L^2}^2\;ds.
\end{align}
Similarly, when $\mathcal{Z}=\partial_t$ or $\partial_x$, one has
\begin{align}\label{4}
\left|\int_0^t \langle c^p\kappa\varepsilon\mathcal{Z}(\partial_y\tilde{h}-\partial_x\tilde{g}), \varepsilon\mathcal{Z}\tilde{\rho}\rangle\;ds\right|\le
C\varepsilon^2\int_0^t\|\nabla{\bf{U}}_{\tau}\|_{L^2}^2\;ds+C\int_0^t\|\sqrt{\varepsilon}\mathcal{Z}\tilde{\rho}\|_{L^2}^2\;ds.
\end{align}
For the case $\mathcal{Z}=y\partial_y$ or $\sqrt{\varepsilon}\partial_y$, since $\|yc^p\|_{L^\infty}, \|\sqrt{\varepsilon}c^p\|_{L^\infty}=O(1)$, then for any $\lambda>0$, we have
\begin{align}\label{5}
&\left|\int_0^t \langle c^p\kappa\varepsilon\mathcal{Z}(\partial_y\tilde{h}-\partial_x\tilde{g}), \varepsilon\mathcal{Z}\tilde{\rho}\rangle\;ds\right|\notag\\
\le&\lambda\varepsilon^3\int_0^t\|\partial_{yy}\tilde{h}\|_{L^2}^2\;ds+C\varepsilon^2\int_0^t\|\nabla{\bf{U}}_{\tau}\|_{L^2}^2\;ds+C\int_0^t\|\sqrt{\varepsilon}\mathcal{Z}\tilde{\rho}\|_{L^2}^2\;ds\notag\\
\le& \lambda Q(t)+C_\lambda\varepsilon^2\int_0^t\|\nabla{\bf{U}}_{\tau}\|_{L^2}^2\;ds+C\int_0^t\|\sqrt{\varepsilon}\mathcal{Z}\tilde{\rho}\|_{L^2}^2\;ds.
\end{align}
The following term is the most involved. We divide it into two parts.
\begin{align*}
\int_0^t\langle\mathcal{Z}{\bf{u}}^\varepsilon\cdot\nabla\tilde{\rho}, \varepsilon\mathcal{Z}\tilde{\rho}\rangle\;ds=\int_0^t\langle\mathcal{Z}u^\varepsilon\partial_x\tilde{\rho}, \varepsilon\mathcal{Z}\tilde{\rho}\rangle\;ds+\int_0^t\langle\mathcal{Z}v^\varepsilon\partial_y\tilde{\rho}, \varepsilon\mathcal{Z}\tilde{\rho}\rangle\;ds.
\end{align*}
For the first part, from \eqref{3.26}, we infer that
\begin{align*}
\int_0^t\|\mathcal{Z}u^\varepsilon\|_{L^\infty}^2\;ds=\int_0^t\|\mathcal{Z}u^a\|_{L^\infty}^2\;ds+\int_0^t\|\varepsilon^{\frac32}\mathcal{Z}u\|_{L^\infty}^2\;ds\lesssim 1+Q(t).
\end{align*}
Thus, we deduce that
\begin{align}\label{6}
\left|\int_0^t\langle\mathcal{Z}u^\varepsilon\partial_x\tilde{\rho}, \varepsilon\mathcal{Z}\tilde{\rho}\rangle\;ds\right|\le \lambda Q(t)+C_\lambda\int_0^t (1+\|\sqrt{\varepsilon}\mathcal{Z}\tilde{\rho}\|_{L^2}^2)\|\sqrt{\varepsilon}\mathcal{Z}\tilde{\rho}\|_{L^2}^2\;ds.
\end{align}
For the second part, we first consider the case that $\mathcal{Z}=\partial_t$ or $\partial_x$. Then by H\"ardy's inequality and divergence free condition, it implies that
\begin{align}\label{7}
\left|\int_0^t\langle\mathcal{Z}v^\varepsilon\partial_y\tilde{\rho}, \varepsilon\mathcal{Z}\tilde{\rho}\rangle\;ds\right|=&\left|\int_0^t\langle\frac{\partial_{\tau}v^a}{y}y\partial_y\tilde{\rho}, \varepsilon\mathcal{Z}\tilde{\rho}\rangle\;ds+\int_0^t\langle \varepsilon\partial_\tau v\cdot\sqrt{\varepsilon}\partial_y\tilde{\rho}, \varepsilon\mathcal{Z}\tilde{\rho}\rangle\;ds\right|\notag\\
\le&\lambda\int_0^t \|\varepsilon\partial_\tau v\|_{L^\infty}^2\;ds+C_\lambda\int_0^t(1+\|\sqrt{\varepsilon}\mathcal{Z}\tilde{\rho}\|_{L^2}^2)
\|\sqrt{\varepsilon}\mathcal{Z}\tilde{\rho}\|_{L^2}^2\;ds\notag\\
\le& \lambda Q(t)+C_\lambda\int_0^t(1+\|\sqrt{\varepsilon}\mathcal{Z}\tilde{\rho}\|_{L^2}^2)\|\sqrt{\varepsilon}\mathcal{Z}\tilde{\rho}\|_{L^2}^2\;ds.
\end{align}
If $\mathcal{Z}=y\partial_y$ or $\sqrt{\varepsilon}\partial_y$, by divergence free condition, it becomes
\begin{align*}
\int_0^t\langle\mathcal{Z}v^\varepsilon\partial_y\tilde{\rho}, \varepsilon\mathcal{Z}\tilde{\rho}\rangle\;ds=\int_0^t\langle\partial_yv^\varepsilon\mathcal{Z}\tilde{\rho}, \varepsilon\mathcal{Z}\tilde{\rho}\rangle\;ds=\int_0^t\langle\partial_xu^\varepsilon\mathcal{Z}\tilde{\rho}, \varepsilon\mathcal{Z}\tilde{\rho}\rangle\;ds.
\end{align*}
In this way, it can be estimated by the same arguments as those in \eqref{7}. Combining \eqref{1}-\eqref{7} and using Propositions \ref{P3.9} and \ref{P3.10}, we achieve that
\begin{align*}
\|\sqrt{\varepsilon}\mathcal{Z}\tilde{\rho}\|_{L^2}^2\le C+(C_\lambda t+\lambda)\mathcal{P}(N(t)).
\end{align*}
This finishes the proof of Proposition $\ref{P3.11}$.
\end{proof}

\subsection{Second order tangential derivative estimates of {\bf{U}}}
This subsection is devoted to establishing the energy estimates of $\partial_{\tau\tau}{\bf{U}}$.
\begin{proposition}
\label{P3.12}
For any fixed constant $\delta>0$ which satisfies \eqref{3.19}-\eqref{3.21}, there exists a unique solution ${\bf{U}}(t, x, y)$ to the system \eqref{3.9}, it holds that for any sufficiently small $\lambda>0$,
\begin{align}\label{3.68}
\|\partial_{\tau\tau}{\bf{U}}\|_{L^2}^2+\varepsilon\int_0^t \|\nabla {\bf{U}}_{\tau\tau}(s, \cdot)\|_{L^2}^2\;ds
\le C+(C_\lambda t+\lambda)\mathcal{P}(N(t)),
\end{align}
where positive constants $C$ and $C_\lambda$ are independent of $\varepsilon$.
\end{proposition}

\begin{proof}
Applying the second order tangential derivative $\partial_{\tau\tau}(\tau=t, x)$ on the second equation in \eqref{3.9}, then multiplying it by ${\bf{D}}$ from the left and taking the inner product with $\varepsilon^2\partial_{\tau\tau}{\bf{U}}$, finally integrating it over $[0, t]$, we have
\begin{align}\label{3.69}
&\frac12\langle {\bf{DS}}^\varepsilon\partial_{\tau\tau}{\bf{U}}, \varepsilon^2\partial_{\tau\tau}{\bf{U}}\rangle
+\int_0^t\langle {\bf{DA}}_1^\varepsilon\partial_x\partial_{\tau\tau}{\bf{U}}, \varepsilon^2\partial_{\tau\tau}{\bf{U}}\rangle\;ds
+\int_0^t\langle {\bf{DA}}_2^\varepsilon\partial_y\partial_{\tau\tau}{\bf{U}}, \varepsilon^2\partial_{\tau\tau}{\bf{U}}\rangle\;ds
\notag\\
&\quad+\int_0^t\langle {\bf{D}}\partial_{\tau\tau}{\bf{C}}^\varepsilon, \varepsilon^2\partial_{\tau\tau}{\bf{U}}\rangle\;ds
+\int_0^t \langle (\partial_{\tau\tau} p_x, \partial_{\tau\tau} p_y, 0, 0)^T, \varepsilon^2\partial_{\tau\tau}{\bf{U}}\rangle\;ds
\notag\\
&\quad-\int_0^t \varepsilon\langle {\bf{DB}}^\varepsilon\Delta\partial_{\tau\tau}{\bf{U}}, \varepsilon^2\partial_{\tau\tau}{\bf{U}}\rangle\;ds
-\int_0^t \langle{\bf{D}}\partial_{\tau\tau}{\bf{E}}^\varepsilon, \varepsilon^2\partial_{\tau\tau}{\bf{U}}\rangle\;ds
\notag\\
&\quad+2\int_0^t \left\langle{\bf{D}}\left(\partial_\tau{\bf{S}}^\varepsilon\partial_{t\tau}{\bf{U}}
+\partial_\tau{\bf{A}}_1^\varepsilon\partial_{x\tau}{\bf{U}}
+\partial_\tau{\bf{A}}_2^\varepsilon\partial_{y\tau}{\bf{U}}\right), \varepsilon^2\partial_{\tau\tau}{\bf{U}}\right\rangle\;ds
\notag\\
&\quad-2\int_0^t \varepsilon\langle {\bf{D}}\partial_\tau{\bf{B}}^\varepsilon\Delta\partial_\tau{\bf{U}}, \varepsilon^2\partial_{\tau\tau}{\bf{U}}\rangle\;ds
\notag\\
&\quad+\int_0^t \left\langle{\bf{D}}\left(\partial_{\tau\tau}{\bf{S}}^\varepsilon\partial_{t}{\bf{U}}
+\partial_{\tau\tau}{\bf{A}}_1^\varepsilon\partial_{x}{\bf{U}}
+\partial_{\tau\tau}{\bf{A}}_2^\varepsilon\partial_{y}{\bf{U}}\right), \varepsilon^2\partial_{\tau\tau}{\bf{U}}\right\rangle\;ds
\notag\\
&\quad -\int_0^t \varepsilon\langle {\bf{D}}\partial_{\tau\tau}{\bf{B}}^\varepsilon\Delta{\bf{U}}, \varepsilon^2\partial_{\tau\tau}{\bf{U}}\rangle\;ds
\notag\\
&\quad=\frac12\varepsilon^2\|\sqrt{{\bf{DS}}^\varepsilon}\partial_{\tau\tau}{\bf{U}}(0, \cdot)\|_{L^2}^2
+\frac12\int_0^t \langle\partial_t({\bf{DS}}^\varepsilon)\partial_{\tau\tau}{\bf{U}}, \varepsilon^2\partial_{\tau\tau}{\bf{U}}\rangle\;ds.
\end{align}
First, by integration by parts, the divergence free condition and \eqref{3.21}, we obtain
\begin{align}\label{3.70}
\langle (\partial_{\tau\tau} p_x, \partial_{\tau\tau} p_y, 0, 0)^T, \varepsilon^2\partial_{\tau\tau}{\bf{U}}\rangle=0
\end{align}
and
\begin{align}\label{3.71}
\langle {\bf{DS}}^\varepsilon\partial_{\tau\tau}{\bf{U}}, \varepsilon^2\partial_{\tau\tau}{\bf{U}}\rangle\ge c_\delta\|\varepsilon\partial_{\tau\tau}{\bf{U}}(t, \cdot)\|_{L^2}^2.
\end{align}
Similar arguments as those in \eqref{3.35}-\eqref{3.37} lead to
\begin{align}\label{3.72}
&\left|\int_0^t\langle {\bf{DA}}_1^\varepsilon\partial_x\partial_{\tau\tau}{\bf{U}}, \varepsilon^2\partial_{\tau\tau}{\bf{U}}\rangle\;ds\right|+
\left|\int_0^t\langle {\bf{DA}}_2^\varepsilon\partial_y\partial_{\tau\tau}{\bf{U}}, \varepsilon^2\partial_{\tau\tau}{\bf{U}}\rangle\;ds\right|\notag\\
&\le \frac{\delta\varepsilon^3}{4}\int_0^t \|\nabla{\bf{U}}_{\tau\tau}\|^2_{L^2}\;ds+2\lambda Q(t)+C\int_0^t (1+N(t)+N^2(t))\|\varepsilon\partial_{\tau\tau}{\bf{U}}\|_{L^2}^2\;ds\notag\\
&\quad+\int_0^t\left(1+N^2(t)+N^4(t)\right)\|\varepsilon\partial_{\tau\tau}{\bf{U}}\|_{L^2}^4\;ds.
\end{align}
Then, by Proposition $\ref{P3.2}$ and Lemma $\ref{L3.6}$, one has respectively
\begin{align}\label{3.73}
\left|\int_0^t\langle {\bf{D}}\partial_{\tau\tau}{\bf{E}}^\varepsilon, \varepsilon^2\partial_{\tau\tau}{\bf{U}}\rangle\;ds\right|\le C+C\int_0^t\|\varepsilon\partial_{\tau\tau}{\bf{U}}\|_{L^2}^2\;ds
\end{align}
and
\begin{align}\label{3.74}
\int_0^t \langle\partial_t({\bf{DS}}^\varepsilon)\partial_{\tau\tau}{\bf{U}}, \varepsilon^2\partial_{\tau\tau}{\bf{U}}\rangle\;ds\le C\int_0^t \left(1+N(t)+N^2(t)\right)\|\varepsilon\partial_{\tau\tau}{\bf{U}}\|_{L^2}^2\;ds.
\end{align}
By integration by parts again, similar derivation as \eqref{3.41} and \eqref{3.42} yields that
\begin{align}\label{3.76}
-\varepsilon\langle {\bf{DB}}^\varepsilon\Delta\partial_{\tau\tau}{\bf{U}}, \varepsilon^2\partial_{\tau\tau}{\bf{U}}\rangle\ge \frac{7\delta\varepsilon^3}{8}\|\nabla{\bf{U}}_{\tau\tau}\|^2_{L^2}-C\left(1+N^2(t)+N^4(t)\right)\|\varepsilon\partial_{\tau\tau}{\bf{U}}\|_{L^2}^2.
\end{align}
The first two terms of the forth line of \eqref{3.69} can be estimated as \eqref{3.55}.
\begin{align}\label{3.77}
&\left|\int_0^t \left\langle{\bf{D}}\left(\partial_\tau{\bf{S}}^\varepsilon\partial_{t\tau}{\bf{U}}+\partial_\tau{\bf{A}}_1^\varepsilon\partial_{x\tau}{\bf{U}}\right), \varepsilon^2\partial_{\tau\tau}{\bf{U}}\right\rangle\;ds\right|\notag\\
\le& C\int_0^t \left(1+N(t)+N^2(t)+\|\varepsilon\partial_{\tau\tau}{\bf{U}}\|^2_{L^2}\right){\|\varepsilon\partial_{\tau\tau}\bf{U}}\|_{L^2}^2\;ds.
\end{align}
Also, rewrite ${\bf{A}}_2^\varepsilon$ as in Section $3.5$
\begin{align*}
{\bf{A}}_2^\varepsilon={\bf{A}}_2^a+\sqrt{\varepsilon}{\bf{A}}_2^p+\varepsilon^{\frac32}{\bf{A}}_2^0.
\end{align*}
Similarly, we have
\begin{align}\label{3.78}
\left|\int_0^t \left\langle{\bf{D}}\left(\partial_\tau{\bf{A}}_2^a+\sqrt{\varepsilon}\partial_\tau{\bf{A}}_2^p\right)\partial_{y\tau}{\bf{U}}, \varepsilon^2\partial_{\tau\tau}{\bf{U}}\right\rangle\;ds\right|\le C\varepsilon\int_0^t \|\nabla{\bf{U}}_{\tau}\|_{L^2}^2\;ds+C\int_0^t {\|\varepsilon\partial_{\tau\tau}\bf{U}}\|_{L^2}^2\;ds.
\end{align}
By H\"older's inequality and Sobolev embedding inequality, one has
\begin{align}\label{3.79}
&\left|\int_0^t \left\langle{\bf{D}}\varepsilon^{\frac32}\partial_\tau{\bf{A}}_2^0\partial_{y\tau}{\bf{U}}, \varepsilon^2\partial_{\tau\tau}{\bf{U}}\right\rangle\;ds\right|\notag\\
\le&\int_0^t  \|{\bf{D}}\|_{L^\infty}\|\sqrt{\varepsilon}\partial_\tau{\bf{A}}_2^0\|_{L^2}\|\varepsilon^{\frac54}\partial_{y\tau}{\bf{U}}\|_{L_x^\infty L_y^2}\|\varepsilon^{\frac54}\partial_{\tau\tau}{\bf{U}}\|_{L_x^2 L_y^\infty}\;ds\notag\\
\le& C\int_0^t (1+Q(t))N(t)\|\varepsilon\partial_{y\tau}{\bf{U}}\|_{L^2}^{\frac12}\|\varepsilon^{\frac32}\partial_{xy\tau}{\bf{U}}\|_{L^2}^{\frac12}\|\varepsilon\partial_{\tau\tau}{\bf{U}}\|_{L^2}^{\frac12}\|\varepsilon^{\frac32}\partial_{y\tau\tau}{\bf{U}}\|_{L^2}^{\frac12}\;ds\notag\\
\le& \frac{\delta\varepsilon^3}{4}\int_0^t \|\nabla{\bf{U}}_{\tau\tau}\|^2_{L^2}\;ds+C\varepsilon^2\int_0^t \|\nabla{\bf{U}}_{\tau}\|^2_{L^2}\;ds+C\int_0^t (1+Q^4(t))N^4(t)|\varepsilon\partial_{\tau\tau}{\bf{U}}\|_{L^2}^2\;ds.
\end{align}
By the same procedure as \eqref{3.58}-\eqref{3.59}, we find
\begin{align}\label{3.80}
\left|\int_0^t \varepsilon\langle {\bf{D}}\partial_\tau{\bf{B}}^\varepsilon\Delta\partial_\tau{\bf{U}}, \varepsilon^2\partial_{\tau\tau}{\bf{U}}\rangle\;ds\right|
&\le\frac{\delta\varepsilon^3}{8}\int_0^t \|\nabla{\bf{U}}_{\tau\tau}\|^2_{L^2}\;ds+\lambda\varepsilon^4\int_0^t\|\partial_{\tau yy}\tilde{h}\|_{L^2}^2\;ds\notag\\
&\quad+C_\lambda\int_0^t \left(1+N^2(t)+N^4(t)\right)\|\varepsilon\partial_{\tau\tau}{\bf{U}}\|_{L^2}^2\;ds.
\end{align}
Finally, we establish the estimates of terms in the sixth and seventh lines in \eqref{3.69}, and they can be rewritten in the following manner:
\begin{align*}
\int_0^t \left\langle{\bf{D}}\left(\partial_{\tau\tau}{\bf{S}}^\varepsilon\partial_{t}{\bf{U}}
+\partial_{\tau\tau}{\bf{A}}_1^\varepsilon\partial_{x}{\bf{U}}
+\partial_{\tau\tau}{\bf{A}}_2^\varepsilon\partial_{y}{\bf{U}}\right), \varepsilon^2\partial_{\tau\tau}{\bf{U}}\right\rangle\;ds-\int_0^t \varepsilon\langle {\bf{D}}\partial_{\tau\tau}{\bf{B}}^\varepsilon\Delta{\bf{U}}, \varepsilon^2\partial_{\tau\tau}{\bf{U}}\rangle\;ds.
\end{align*}
To show the main idea of the proof and make the presentation more concise, we only consider the terms in the equation of $\tilde{u}$, and the other terms can be handled in a similar way.
\begin{align}\label{3.81}
&\langle\partial_{\tau\tau}\rho^\varepsilon\left(\partial_t\tilde{u}+(u^\varepsilon+a^ph^\varepsilon)\partial_x\tilde{u}+h^\varepsilon (a^p)^2\partial_x\tilde{h}-\kappa\varepsilon(\partial_ya^p+a^pb^p)\partial_x\tilde{g}\right.\notag\\
&\left.\quad+(v^\varepsilon+a^pg^\varepsilon)\partial_y\tilde{u}+g^\varepsilon (a^p)^2\partial_y\tilde{h}+\kappa\varepsilon(\partial_ya^p+a^pb^p)\partial_y\tilde{h}+\kappa\varepsilon a^p\Delta\tilde{h}\right), \varepsilon^2\partial_{\tau\tau}\tilde{u}\rangle\notag\\
&+\langle l.o.t., \varepsilon^2\partial_{\tau\tau}\tilde{u}\rangle.
\end{align}

Here, $l.o.t.$ consists of the terms at most one order tangential derivative of the density $\rho^\varepsilon$ which can be estimated directly as above. Hence, we only focus on the terms involving $\partial_{\tau\tau}\rho^\varepsilon$ for which we can not obtain any estimate directly. To overcome this issue, we consider $\partial_{\tau\tau}$ in the following two cases.\\
{\bf Case I:} $\partial_{\tau\tau}=\partial_{tx}$ or $\partial_{\tau\tau}=\partial_{xx}$.

For this case, we use the integration by parts for $x$ variable, and the first part of \eqref{3.81} becomes
\begin{align}\label{3.82}
&-\langle\partial_{\tau}\rho^\varepsilon\partial_x\left(\partial_{t}\tilde{u}+(u^\varepsilon+a^ph^\varepsilon)\partial_x\tilde{u}+h^\varepsilon (a^p)^2\partial_x\tilde{h}-\kappa\varepsilon(\partial_ya^p+a^pb^p)\partial_x\tilde{g}\right.\notag\\
&\left.\quad+(v^\varepsilon+a^pg^\varepsilon)\partial_y\tilde{u}+g^\varepsilon (a^p)^2\partial_y\tilde{h}+\kappa\varepsilon(\partial_ya^p+a^pb^p)\partial_y\tilde{h}+\kappa\varepsilon a^p\Delta\tilde{h}\right), \varepsilon^2\partial_{\tau\tau}\tilde{u}\rangle\notag\\
&-\langle\partial_{\tau}\rho^\varepsilon\left(\partial_{t}\tilde{u}+(u^\varepsilon+a^ph^\varepsilon)\partial_x\tilde{u}+h^\varepsilon (a^p)^2\partial_x\tilde{h}-\kappa\varepsilon(\partial_ya^p+a^pb^p)\partial_x\tilde{g}\right.\notag\\
&\left.\quad+(v^\varepsilon+a^pg^\varepsilon)\partial_y\tilde{u}+g^\varepsilon (a^p)^2\partial_y\tilde{h}+\kappa\varepsilon(\partial_ya^p+a^pb^p)\partial_y\tilde{h}+\kappa\varepsilon a^p\Delta\tilde{h}\right), \varepsilon^2\partial_{x\tau\tau}\tilde{u}\rangle\notag\\
&=:K_1+K_2.
\end{align}
Notice that $K_1$ is just like the $l.o.t.$ in \eqref{3.81} which can also be estimated similarly. Next, by using \eqref{3.29} and \eqref{3.25}, we deduce that
\begin{align}\label{3.83}
|K_2|\le& \frac{\delta\varepsilon^3}{4}\int_0^t \|\nabla{\bf{U}}_{\tau\tau}\|^2_{L^2}\;ds+C\left(1+N^2(t)+N^4(t)\right)\varepsilon\int_0^t \|\nabla{\bf{U}}\|^2_{L^2}\;ds\\
&+C\left(1+N^2(t)+N^4(t)\right)\varepsilon^3\int_0^t \|\partial_{yy}h\|_{L^2}^2\;ds+C\int_0^t \left(1+N^2(t)+N^4(t)\right){\|\varepsilon\partial_{\tau\tau}\bf{U}}\|_{L^2}^2\;ds.\notag
\end{align}
{\bf Case II:} $\partial_{\tau\tau}=\partial_{tt}$.

From the first equation in \eqref{1.1}, we find that
\begin{align*}
\partial_{tt}\rho^\varepsilon=&-\partial_{t}(u^\varepsilon\partial_x\rho^\varepsilon+v^\varepsilon\partial_y\rho^\varepsilon)\\
=&-(\partial_{t}u^\varepsilon\partial_x\rho^\varepsilon+\partial_tv^\varepsilon\partial_y\rho^\varepsilon)-u^\varepsilon\partial_{x}\partial_t\rho^\varepsilon-v^\varepsilon\partial_y\partial_{t}\rho^\varepsilon.
\end{align*}
Substitute it into \eqref{3.81} and we separate the first term of \eqref{3.81} into three parts.
\begin{align*}
&-\langle(\partial_{t}u^\varepsilon\partial_x\rho^\varepsilon+\partial_tv^\varepsilon\partial_y\rho^\varepsilon)\left(\partial_t\tilde{u}+(u^\varepsilon+a^ph^\varepsilon)\partial_x\tilde{u}+h^\varepsilon (a^p)^2\partial_x\tilde{h}-\kappa\varepsilon(\partial_ya^p+a^pb^p)\partial_x\tilde{g}\right.\\
&\left.\quad+(v^\varepsilon+a^pg^\varepsilon)\partial_y\tilde{u}+g^\varepsilon (a^p)^2\partial_y\tilde{h}+\kappa\varepsilon(\partial_ya^p+a^pb^p)\partial_y\tilde{h}+\kappa\varepsilon a^p\Delta\tilde{h}\right), \varepsilon^2\partial_{\tau\tau}\tilde{u}\rangle\\
&-\langle(u^\varepsilon\partial_{x}\partial_t\rho^\varepsilon)\left(\partial_t\tilde{u}+(u^\varepsilon+a^ph^\varepsilon)\partial_x\tilde{u}+h^\varepsilon (a^p)^2\partial_x\tilde{h}-\kappa\varepsilon(\partial_ya^p+a^pb^p)\partial_x\tilde{g}\right.\\
&\left.\quad+(v^\varepsilon+a^pg^\varepsilon)\partial_y\tilde{u}+g^\varepsilon (a^p)^2\partial_y\tilde{h}+\kappa\varepsilon(\partial_ya^p+a^pb^p)\partial_y\tilde{h}+\kappa\varepsilon a^p\Delta\tilde{h}\right), \varepsilon^2\partial_{\tau\tau}\tilde{u}\rangle\\
&-\langle(v^\varepsilon\partial_y\partial_{t}\rho^\varepsilon)\left(\partial_t\tilde{u}+(u^\varepsilon+a^ph^\varepsilon)\partial_x\tilde{u}+h^\varepsilon (a^p)^2\partial_x\tilde{h}-\kappa\varepsilon(\partial_ya^p+a^pb^p)\partial_x\tilde{g}\right.\\
&\left.\quad+(v^\varepsilon+a^pg^\varepsilon)\partial_y\tilde{u}+g^\varepsilon (a^p)^2\partial_y\tilde{h}+\kappa\varepsilon(\partial_ya^p+a^pb^p)\partial_y\tilde{h}+\kappa\varepsilon a^p\Delta\tilde{h}\right), \varepsilon^2\partial_{\tau\tau}\tilde{u}\rangle\\
&=:K_3+K_4+K_5.
\end{align*}
We observe that $K_3$ can be estimated in the same way as the $l.o.t.$ in \eqref{3.81}. After the integration by parts for $x$ variable, $K_4$ can be estimated as \eqref{3.82}. As for $K_5$, we use integration by parts for $y$ variable, and it becomes
\begin{align}\label{3.84}
&-\langle(\partial_yv^\varepsilon\partial_{t}\rho^\varepsilon)\left(\partial_t\tilde{u}+(u^\varepsilon+a^ph^\varepsilon)\partial_x\tilde{u}+h^\varepsilon (a^p)^2\partial_x\tilde{h}-\kappa\varepsilon(\partial_ya^p+a^pb^p)\partial_x\tilde{g}\right.\notag\\
&\left.\quad+(v^\varepsilon+a^pg^\varepsilon)\partial_y\tilde{u}+g^\varepsilon (a^p)^2\partial_y\tilde{h}+\kappa\varepsilon(\partial_ya^p+a^pb^p)\partial_y\tilde{h}+\kappa\varepsilon a^p\Delta\tilde{h}\right), \varepsilon^2\partial_{\tau\tau}\tilde{u}\rangle\notag\\
&-\langle(v^\varepsilon\partial_{t}\rho^\varepsilon)\left(\partial_t\tilde{u}+(u^\varepsilon+a^ph^\varepsilon)\partial_x\tilde{u}+h^\varepsilon (a^p)^2\partial_x\tilde{h}-\kappa\varepsilon(\partial_ya^p+a^pb^p)\partial_x\tilde{g}\right.\notag\\
&\left.\quad+(v^\varepsilon+a^pg^\varepsilon)\partial_y\tilde{u}+g^\varepsilon (a^p)^2\partial_y\tilde{h}+\kappa\varepsilon(\partial_ya^p+a^pb^p)\partial_y\tilde{h}+\kappa\varepsilon a^p\Delta\tilde{h}\right), \varepsilon^2\partial_{y\tau\tau}\tilde{u}\rangle\notag\\
&-\langle(v^\varepsilon\partial_{t}\rho^\varepsilon)\partial_y\left(\partial_t\tilde{u}+(u^\varepsilon+a^ph^\varepsilon)\partial_x\tilde{u}+h^\varepsilon (a^p)^2\partial_x\tilde{h}-\kappa\varepsilon(\partial_ya^p+a^pb^p)\partial_x\tilde{g}\right.\notag\\
&\left.\quad+(v^\varepsilon+a^pg^\varepsilon)\partial_y\tilde{u}+g^\varepsilon (a^p)^2\partial_y\tilde{h}+\kappa\varepsilon(\partial_ya^p+a^pb^p)\partial_y\tilde{h}+\kappa\varepsilon a^p\Delta\tilde{h}\right), \varepsilon^2\partial_{\tau\tau}\tilde{u}\rangle.
\end{align}
Here, we only concentrate on the terms in the last line since the others terms can be handled similarly as above. First, we learn from H\"ardy's inequality that
\begin{align}\label{3.85}
&\int_0^t\langle v^\varepsilon\partial_t\rho^\varepsilon (v^\varepsilon+a^pg^\varepsilon)\partial_{yy}\tilde{u}, \varepsilon^2\partial_{\tau\tau}\tilde{u}\rangle\;ds\notag\\
\lesssim &\int_0^t\langle v^a\partial_t\rho^\varepsilon (v^\varepsilon+a^pg^\varepsilon)\partial_{yy}\tilde{u}, \varepsilon^2\partial_{\tau\tau}\tilde{u}\rangle\;ds
+\int_0^t\langle \varepsilon^{\frac32}v\partial_t\rho^\varepsilon (v^\varepsilon+a^pg^\varepsilon)\partial_{yy}\tilde{u}, \varepsilon^2\partial_{\tau\tau}\tilde{u}\rangle\;ds\notag\\
&\lesssim \int_0^t\left\langle \frac{v^a}{y}\partial_t\rho^\varepsilon (v^\varepsilon+a^pg^\varepsilon)y\partial_{yy}\tilde{u}, \varepsilon^2\partial_{\tau\tau}\tilde{u}\right\rangle\;ds
+\int_0^t\langle \varepsilon v\partial_t\rho^\varepsilon (v^\varepsilon+a^pg^\varepsilon)\varepsilon^{\frac32}\partial_{yy}\tilde{u}, \varepsilon\partial_{\tau\tau}\tilde{u}\rangle\;ds\notag\\
&\le C\varepsilon^2\int_0^t \|\nabla\mathcal{Z}_0\tilde{u}\|_{L^2}^2\;ds+C\varepsilon\int_0^t \|\nabla\tilde{u}\|_{L^2}^2\;ds+C\varepsilon^3\int_0^t \|\partial_{yy}\tilde{u}\|_{L^2}^2\;ds\notag\\
&\quad+C\int_0^t \left(1+N^2(t)+N^4(t)\right)\|\varepsilon\partial_{\tau\tau}\tilde{u}\|_{L^2}^2\;ds,
\end{align}
where in the last line, we used divergence free condition, \eqref{3.24} and \eqref{3.29}. Next
\begin{align}\label{3.86}
&\int_0^t\langle v^\varepsilon\partial_t\rho^\varepsilon (\kappa\varepsilon a^p\partial_{yyy}\tilde{h}), \varepsilon^2\partial_{\tau\tau}\tilde{u}\rangle\;ds\notag\\
\le&\lambda\varepsilon^4\int_0^t \|\partial_{yyy}\tilde{h}\|_{L^2}^2\;ds+C_\lambda\int_0^t \left(1+N^2(t)+N^4(t)\right)\|\varepsilon\partial_{\tau\tau}\tilde{u}\|_{L^2}^2\;ds.
\end{align}
The remaining terms in \eqref{3.84} can be estimated similarly. Yet, we have not estimated the fourth term in \eqref{3.69}. Indeed, for those terms involving at most one order derivative of $\rho^\varepsilon$, we can get the similar estimates as Proposition \ref{vector} by the same procedure as Appendix C. For the terms of second order derivatives of $\rho^\varepsilon$, we can adopt the way above and omit the details for brevity. Therefore, we finish the estimates of all the terms in \eqref{3.69}, Finally, by conclusion, we arrive at
\begin{align}\label{3.87}
\|\varepsilon\partial_{\tau\tau}{\bf{U}}(t, \cdot)\|_{L^2}^2+\varepsilon^3\int_0^t \|\nabla {\bf{U}}_{\tau\tau}(s, \cdot)\|_{L^2}^2\;ds\le C+(C_\lambda t+\lambda)\mathcal{P}(N(t)).
\end{align}
Then the proof of Proposition $\ref{P3.12}$ is complete.
\end{proof}

\subsection{Pressure estimates}
To close the energy estimates achieved in the above section, it suffices to obtain the estimates of pressure through its elliptic property. By applying the divergence operator on the equation of ${\bf{u}}$ in \eqref{3.2}, we find $p$ satisfies the following elliptic equation with Neumann boundary condition:
\begin{align}\label{3.88}
\begin{cases}
\Delta p=\nabla\cdot{\bf{F}}\\
\partial_yp|_{y=0}=(-\mu\varepsilon\partial_{xy}u-r_2)|_{y=0},
\end{cases}
\end{align}
where
\begin{align}\label{3.89}
{\bf{F}}&=-\rho^\varepsilon\left(\partial_t{\bf{u}}+{\bf{u}}^\varepsilon\cdot\nabla{\bf{u}}\right)+{\bf{H}}^\varepsilon\cdot\nabla{\bf{H}}-\rho\left(\partial_t{\bf{u}}^a+{\bf{u}}^a\cdot\nabla{\bf{u}}^a\right)\notag\\
&\quad-\rho^\varepsilon{\bf{u}}\cdot\nabla{\bf{u}}^a+{\bf{H}}\cdot\nabla{\bf{H}}^a-r_{{\bf{u}}}.
\end{align}
For this system, we have the following elliptic estimates of pressure $p$.
\begin{proposition}
\label{P3.13}
For any solution $p$ of \eqref{3.88}, it satisfies that
\begin{align}\label{3.90}
\|\nabla p\|_{L^2}\lesssim \|{\bf{F}}\|_{L^2}+\varepsilon\|\partial_{yy}u\|_{L^2}^{\frac12}\|\partial_yu\|_{L^2}^{\frac12},\quad \|\nabla\partial_x p\|_{L^2}\lesssim\|\partial_x{\bf{F}}\|_{L^2}+\varepsilon\|\partial_{xyy}u\|_{L^2}^{\frac12}\|\partial_{xy}u\|_{L^2}^{\frac12}.
 \end{align}
\end{proposition}
\begin{proof} To prove this proposition, we first decompose
\begin{align*}
p=p_1+p_2,
\end{align*}
where $p_1$ solves
\begin{align}\label{3.99}
\begin{cases}
\Delta p_1=\nabla\cdot {\bf{F}},\\
\partial_yp_1|_{y=0}=-r_2|_{y=0}=F_2|_{y=0},
\end{cases}
\end{align}
and $p_2$ solves
\begin{align}\label{3.100}
\begin{cases}
\Delta p_2=0,\\
\partial_yp_2|_{y=0}=-\mu\varepsilon\partial_{xy}u|_{y=0}.
\end{cases}
\end{align}
Then it follows from \cite{MR12} that
\begin{align}\label{3.101}
\|\nabla p_1\|_{L^2}\lesssim \|{\bf{F}}\|_{L^2}, \quad \|\nabla\partial_x p_1\|_{L^2}\lesssim \|\partial_x{\bf{F}}\|_{L^2},
\end{align}
and
\begin{align}\label{3.102}
\|\nabla p_2\|_{L^2}\lesssim \varepsilon\|\partial_{yy}u\|_{L^2}^{\frac12}\|\partial_yu\|_{L^2}^{\frac12},\quad\|\nabla\partial_x p_2\|_{L^2}\lesssim \varepsilon\|\partial_{xyy}u\|_{L^2}^{\frac12}\|\partial_{xy}u\|_{L^2}^{\frac12}.
\end{align}
\end{proof}

From which and the definition of ${\bf{F}}$, we immediately obtain that
\begin{corollary}
\label{Cor3.14}
For any solution $p$ of \eqref{3.88}, it holds that for any $t\in [0,T]$,
\begin{align}\label{3.91}
\int_0^t \varepsilon\|\nabla p\|_{L^2}^2+\varepsilon^2\|\nabla\partial_x p\|_{L^2}^2\;ds \lesssim t\mathcal{P}(N(t))+\varepsilon Q(t).
 \end{align}
\end{corollary}

\section{Proof of the main theorem} Now, we are ready to prove the main result of this paper. By collecting Propositions $\ref{P3.9}-\ref{P3.13}$ and Corollary $\ref{Cor3.14}$ together, then choosing $\lambda$ and $t$ sufficiently small, we first finish the proof of Theorem $\ref{T3.7}$.

\begin{proof}[Proof of Theorem $\ref{T1.1}$] We are devoted to deriving the $L_{txy}^\infty$ estimates of ${\bf{U}}$. By Sobolev embedding, one has
\begin{align}\label{4.1}
\|{\bf{U}}\|_{L^\infty_{txy}}\lesssim\|{\bf{U}}\|_{L_{tx}^\infty L_y^2}^{\frac12}\cdot\|\partial_y{\bf{U}}\|_{L_{tx}^\infty L_y^2}^{\frac12}.
\end{align}
Then, by Sobolev embedding again, we obtain
\begin{align}\label{4.2}
\|{\bf{U}}\|_{L_{tx}^\infty L_y^2}^{\frac12}\lesssim \|{\bf{U}}\|_{L_t^\infty L_{xy}^2}^{\frac14}\|\partial_x{\bf{U}}\|_{L_t^\infty L_{xy}^2}^{\frac14}\lesssim\varepsilon^{-\frac18}.
\end{align}
By the same procedure, we also have
\begin{align}\label{4.3}
\|\partial_y{\bf{U}}\|_{L_{tx}^\infty L_y^2}^{\frac12}&\lesssim \|\partial_y{\bf{U}}\|_{L_x^\infty L_{ty}^2}^{\frac14}\|\partial_{ty}{\bf{U}}\|_{L_x^\infty L_{ty}^2}^{\frac14}\notag\\
&\lesssim \|\partial_y{\bf{U}}\|_{L_{txy}^2}^{\frac18}\|\partial_{xy}{\bf{U}}\|_{L_{txy}^2}^{\frac18}\cdot \|\partial_{ty}{\bf{U}}\|_{L_{txy}^2}^{\frac18}\|\partial_{txy}{\bf{U}}\|_{L_{txy}^2}^{\frac18}\notag\\
&\lesssim\varepsilon^{-\frac12}.
\end{align}
Plugging \eqref{4.2} and \eqref{4.3} into \eqref{4.1}, we arrive at
\begin{align}\label{4.4}
\|{\bf{U}}\|_{L^\infty_{txy}}\le C\varepsilon^{-\frac58}.
\end{align}
Recall \eqref{3.1},
\begin{align}\label{4.5}
({\bf{u}}^\varepsilon, {\bf{H}}^\varepsilon)(t, x, y)=({\bf{u}}^a, {\bf{H}}^a)(t, x, y)+\varepsilon^{\frac32}({\bf{u}}, {\bf{H}})(t, x, y).
\end{align}
From which we immediately get that
\begin{align}\label{4.6}
\varepsilon^{\frac32}\|({\bf{u}}, {\bf{H}})\|_{L_{txy}^\infty}\le C\varepsilon^{\frac78}.
\end{align}
On the other hand, recall the definition of $({\bf{u}}^a, {\bf{H}}^a)$, it satisfies that
\begin{align}\label{4.7}
({\bf{u}}^a, {\bf{H}}^a)(t, x, y)=({\bf{u}}^0, {\bf{H}}^0)(t, x, y)+(u_b^0, \sqrt{\varepsilon}v_b^0, h_b^0, \sqrt{\varepsilon}g_b^0)\left(t, x, \frac{y}{\sqrt{\varepsilon}}\right)+O(\sqrt{\varepsilon}).
\end{align}
Combine \eqref{4.5}-\eqref{4.7}, we deduce that
\begin{align}\label{4.8}
\sup\limits_{0\le t\le T_*}\Big\|&({\bf{u}}^\varepsilon, {\bf{H}}^\varepsilon)(t, x, y)-({\bf{u}}^0, {\bf{H}}^0)(t, x, y)\notag\\
&-(u_b^0, \sqrt{\varepsilon}v_b^0, h_b^0, \sqrt{\varepsilon}g_b^0)\left(t, x, \frac y{\sqrt{\varepsilon}}\right)\Big\|_{L_{xy}^\infty}\le C\varepsilon^{\frac12}.
\end{align}
Moreover, by \eqref{4.6} and the continuous argument, the assumption \eqref{3.24} also holds. This finishes the proof of Theorem \ref{T1.1}.
\end{proof}

\section*{Acknowledgement}
The research of S. Li's research was supported by the Center for Nonlinear Analysis, The Hong Kong Polytechnic University. The research of F. Xie's research was supported by National Natural Science Foundation of China No.12271359, 11831003, 12161141004, Shanghai Science and Technology Innovation Action Plan No. 20JC1413000 and Institute of Modern Analysis-A Frontier Research Center of Shanghai.

\section*{Appendix A}
In this section, we are devoted to proving Lemma \ref{L3.4} and Lemma $\ref{L3.6}$.
\begin{proof}[Proof of Lemma $\ref{L3.4}$]
First, we establish the second order normal derivatives of $\tilde{h}$. From the equation of $\tilde{h}$ in \eqref{3.9}, it is direct to check that
\begin{align}\label{A.1}
\varepsilon^3\int_0^t \|\partial_{yy}\tilde{h}\|_{L^2}^2\;ds&\le \varepsilon\int_0^t \|\partial_{t}\tilde{h}\|_{L^2}^2\;ds+\varepsilon\int_0^t \|({\bf{u}}^\varepsilon-a^p{\bf{H}}^\varepsilon)\cdot\nabla\tilde{h}\|_{L^2}^2\;ds+\varepsilon\int_0^t \|{\bf{H}}^\varepsilon\cdot\nabla\tilde{u}\|_{L^2}^2\;ds\notag\\
&\quad +\varepsilon\int_0^t \|{\bf{E}}_3^\varepsilon\|_{L^2}^2\;ds+\varepsilon\int_0^t \|{\bf{C}}_3^\varepsilon\|_{L^2}^2\;ds+\kappa^2\varepsilon^3\int_0^t \|\partial_{xx}\tilde{h}\|_{L^2}^2\;ds\notag\\
&\lesssim 1+N(t).\tag{A.1}
\end{align}
By the same procedure, we can also prove that
\begin{align}\label{A.2}
\varepsilon^4\int_0^t \|\partial_{yy\tau}\tilde{h}\|_{L^2}^2\;ds\lesssim 1+N(t)+N^2(t). \tag{A.2}
\end{align}
For the third order normal derivatives, by H\"ardy's trick, one has
\begin{align}\label{A.3}
\varepsilon^4\int_0^t \|\partial_{yyy}\tilde{h}\|_{L^2}^2\;ds&\le \varepsilon^2\int_0^t \|\partial_{ty}\tilde{h}\|_{L^2}^2\;ds+\varepsilon^2\int_0^t \|\partial_y({\bf{u}}^\varepsilon-a^p{\bf{H}}^\varepsilon)\cdot\nabla\tilde{h}\|_{L^2}^2\;ds+\varepsilon^2\int_0^t \|\partial_y{\bf{C}}_3^\varepsilon\|_{L^2}^2\;ds\notag\\
&\quad+\varepsilon^2\int_0^t \|\partial_y{\bf{H}}^\varepsilon\cdot\nabla\tilde{u}\|_{L^2}^2\;ds+\varepsilon^2\int_0^t \|\partial_y{\bf{E}}_3^\varepsilon\|_{L^2}^2\;ds+\kappa^2\varepsilon^4\int_0^t \|\nabla\tilde{h}_{xx}\|_{L^2}^2\;ds\notag\\
&\quad+\varepsilon^2\int_0^t \|(u^\varepsilon-a^ph^\varepsilon)\partial_{xy}\tilde{h}\|_{L^2}^2\;ds+\varepsilon^2\int_0^t \|h^\varepsilon\partial_{xy}\tilde{u}\|_{L^2}^2\;ds\notag\\
&\quad+\varepsilon^2\int_0^t \|(\frac{v^\varepsilon}{y}-a^p\frac{g^\varepsilon}{y})y\partial_{yy}\tilde{h}\|_{L^2}^2\;ds+\varepsilon^2\int_0^t \|\frac{g^\varepsilon}{y}y\partial_{yy}\tilde{u}\|_{L^2}^2\;ds\notag\\
&\lesssim 1+N(t)+N^2(t).\tag{A.3}
\end{align}
From the above estimates for $\tilde{h}$ and the equation of $\tilde{u}$ in \eqref{3.9}, we immediately get that
\begin{align}\label{A.4}
\varepsilon^3\int_0^t \|\partial_{yy}\tilde{u}\|_{L^2}^2\;ds+\varepsilon^4\int_0^t \|\partial_{yy\tau}\tilde{u}\|_{L^2}^2\;ds\lesssim 1+N(t)+N^2(t).\tag{A.4}
\end{align}
Based on the above estimates and Sobolev embedding inequality, we can also achieve the estimates in $L^\infty$ norms in \eqref{3.26}. For the convenience of readers, we prove the estimates for one of these terms, and other terms can be handled similarly.
\begin{align}\label{A.5}
\int_0^t\|\varepsilon\partial_x{\bf{u}}\|_{L^\infty}^2\;ds\lesssim&\int_0^t\|\sqrt{\varepsilon}\partial_x{\bf{u}}\|_{L^2}^{\frac14}\|\varepsilon\partial_{xx}{\bf{u}}\|_{L^2}^{\frac14}\|\varepsilon\partial_{xy}{\bf{u}}\|_{L^2}^{\frac14}\|\varepsilon^{\frac32}\partial_{xxy}{\bf{u}}\|_{L^2}^{\frac14}\;ds\notag\\
\lesssim&\int_0^t\|\sqrt{\varepsilon}\partial_x{\bf{u}}\|_{L^2}^2\;ds+\int_0^t\|\varepsilon\partial_{xx}{\bf{u}}\|_{L^2}^2\;ds+\int_0^t\|\varepsilon\partial_{xy}{\bf{u}}\|_{L^2}^2\;ds\notag\\
&+\int_0^t\|\varepsilon^{\frac32}\partial_{xxy}{\bf{u}}\|_{L^2}^2\;ds\notag\\
\lesssim& 1+N(t)+N^2(t).\tag{A.5}
\end{align}
Thus the proof of Lemma $\ref{L3.4}$ is complete.
\end{proof}

\begin{proof}[Proof of Lemma \ref{L3.6}]
From the first equation in \eqref{1.1}, we know that $\rho^\varepsilon$ satisfies the transport equation and thus
\begin{align*}
\|\rho^\varepsilon(t)\|_{L^\infty}=\|\rho_0\|_{L^\infty}\le C, \quad t\in [0, T].
\end{align*}
Then by applying the operator of $\partial_x$ on the first equation in \eqref{1.1}, we have
\begin{align}\label{A.6}
\partial_t\partial_x\rho^\varepsilon+{\bf{u}}^\varepsilon\cdot\nabla\partial_x\rho^\varepsilon=f_1, \tag{A.6}
\end{align}
where
\begin{align*}
f_1=-\partial_xu^\varepsilon\partial_x\rho^\varepsilon-\partial_xv^\varepsilon\partial_y\rho^\varepsilon.
\end{align*}
Besides, from \eqref{3.1}, H\"ardy's trick and divergence free condition, we obtain
\begin{align*}
\|\partial_xv^\varepsilon\partial_y\rho^\varepsilon\|_{L^\infty}&\le
\|\partial_xv^a\partial_y\rho^\varepsilon\|_{L^\infty}+\|\varepsilon^{\frac32}\partial_xv\partial_y\rho^\varepsilon\|_{L^\infty}\\
&\le \left\|\frac{\partial_xv^a}{y}y\partial_y\rho^\varepsilon\right\|_{L^\infty}+\delta\|\varepsilon\partial_xv\|_{L^\infty}^2+C\|\sqrt{\varepsilon}\partial_y\rho^\varepsilon\|_{L^\infty}^2.
\end{align*}
By H\"olders inequality, the above source term of $f_1$ can be estimated in the following way.
\begin{align*}
\int_0^t \|f_1\|_{L^\infty}\;ds\lesssim&\int_0^t \|\partial_xu^\varepsilon\|_{L^\infty}^2\;ds+\int_0^t \|\partial_x\rho^\varepsilon\|_{L^\infty}^2\;ds\\
&+\int_0^t\|y\partial_y\rho^\varepsilon\|_{L^\infty}^2\;ds+\int_0^t \|\partial_xv^\varepsilon\|_{L^\infty}^2\;ds+\int_0^t\|\sqrt{\varepsilon}\partial_y\rho^\varepsilon\|_{L^\infty}^2\;ds\\
\lesssim&\,1+N(t)+N^2(t)+\int_0^t\|(\partial_x\rho^\varepsilon, y\partial_y\rho^\varepsilon, \sqrt{\varepsilon}\partial_y\rho^\varepsilon)\|_{L^\infty}^2\;ds.
\end{align*}
Thus the maximum principle of the above transport equation implies that
\begin{align}\label{A.7}
\|\partial_x\rho^\varepsilon\|_{L^\infty} &\le \|\partial_x\rho_0\|_{L^\infty}+\int_0^t \|f_1\|_{L^\infty}\;ds\notag\\
&\lesssim 1+N(t)+N^2(t)+\int_0^t\|(\partial_x\rho^\varepsilon, y\partial_y\rho^\varepsilon, \sqrt{\varepsilon}\partial_y\rho^\varepsilon)\|_{L^\infty}^2\;ds.\tag{A.7}
\end{align}
Similarly, applying the operators of $y\partial_y$ and $\sqrt{\varepsilon}\partial_y$ on the first equation in \eqref{1.1} respectively, we have
\begin{align}\label{A.8}
\partial_t(y\partial_y\rho^\varepsilon)+{\bf{u}}^\varepsilon\cdot\nabla(y\partial_y\rho^\varepsilon)=f_2:=-y\partial_yu^\varepsilon\partial_x\rho^\varepsilon-\partial_yv^\varepsilon(y\partial_y)\rho^\varepsilon+v^\varepsilon\partial_y\rho^\varepsilon,\tag{A.8}
\end{align}
\begin{align}\label{A.9}
\partial_t(\sqrt{\varepsilon}\partial_y\rho^\varepsilon)+{\bf{u}}^\varepsilon\cdot\nabla(\sqrt{\varepsilon}\partial_y\rho^\varepsilon)=f_3:=-\sqrt{\varepsilon}\partial_yu^\varepsilon\partial_x\rho^\varepsilon-\partial_yv^\varepsilon(\sqrt{\varepsilon}\partial_y)\rho^\varepsilon.\tag{A.9}
\end{align}
Using the maximum principle of the above transport equations again, we infer that
\begin{align}\label{A.10}
\|y\partial_y\rho^\varepsilon\|_{L^\infty} &\le \|y\partial_y\rho_0\|_{L^\infty}+\int_0^t \|f_2\|_{L^\infty}\;ds\notag\\
&\lesssim 1+N(t)+N^2(t)+\int_0^t\|(\partial_x\rho^\varepsilon, y\partial_y\rho^\varepsilon)\|_{L^\infty}^2\;ds,
\tag{A.10}
\end{align}
\begin{align}\label{A.11}
\|\sqrt{\varepsilon}\partial_y\rho^\varepsilon\|_{L^\infty} &\le \|\sqrt{\varepsilon}\partial_y\rho_0\|_{L^\infty}+\int_0^t \|f_3\|_{L^\infty}\;ds\notag\\
&\lesssim 1+N(t)+N^2(t)+\int_0^t\|(\partial_x\rho^\varepsilon, \sqrt{\varepsilon}\partial_y\rho^\varepsilon)\|_{L^\infty}^2\;ds.\tag{A.11}
\end{align}
Combining \eqref{A.7}, \eqref{A.10}-\eqref{A.11} together, we conclude that for any $t\in[0, T_*]$
\begin{align}\label{A.12}
\|(\partial_x\rho^\varepsilon, y\partial_y\rho^\varepsilon, \sqrt{\varepsilon}\partial_y\rho^\varepsilon)\|_{L^\infty}&\lesssim 1+N(t)+N^2(t)+\int_0^t\|(\partial_x\rho^\varepsilon, \sqrt{\varepsilon}\partial_y\rho^\varepsilon, y\partial_y\rho^\varepsilon)\|_{L^\infty}^2\;ds.\tag{A.12}
\end{align}
Thus, for sufficiently small $t$, it holds
\begin{align}\label{A.13}
\|(\partial_x\rho^\varepsilon, y\partial_y\rho^\varepsilon, \sqrt{\varepsilon}\partial_y\rho^\varepsilon)\|_{L^\infty}\le C\left(1+N(t)+N^2(t)\right).\tag{A.13}
\end{align}
Finally, from the first equation in \eqref{1.1} again, we have
\begin{align}\label{A.14}
\|\partial_t\rho^\varepsilon\|_{L^\infty}\le&\|u^\varepsilon\partial_x\rho^\varepsilon\|_{L^\infty}+\|v^\varepsilon\partial_y\rho^\varepsilon\|_{L^\infty}\notag\\
\le&\|u^\varepsilon\partial_x\rho^\varepsilon\|_{L^\infty}+\|\frac{v^a}{y}y\partial_y\rho^\varepsilon\|_{L^\infty}+\|\varepsilon v\|_{L^\infty}\|\sqrt{\varepsilon}\partial_y\rho^\varepsilon\|_{L^\infty}\notag\\
\le&C\left(1+N(t)+N^2(t)\right), \tag{A.14}
\end{align}
where \eqref{3.24} is used in the list line. This finishes the proof of Lemma $\ref{L3.6}$.
\end{proof}

\section*{Appendix B}
In this section, we will give the expressions of the remainders $R_i\ (i=0, 1, 2, 3, 4)$. To this end, we first introduce some notations to simplify the representation.
\begin{align*}
\begin{cases}
\tau_u(t, x, y)=\sqrt{\varepsilon}\chi'(y)\int_0^{\frac{y}{\sqrt{\varepsilon}}} u_b^2(t, x, \tilde{\eta})d\tilde{\eta},\\
\tau_h(t, x, y)=\sqrt{\varepsilon}\chi'(y)\int_0^{\frac{y}{\sqrt{\varepsilon}}} h_b^2(t, x, \tilde{\eta})d\tilde{\eta}+\sqrt{\varepsilon}\phi(t, x, \eta),\\
\tau_g(t, x, y)=-\sqrt{\varepsilon}\chi'(y)\int_0^{\frac{y}{\sqrt{\varepsilon}}} \partial_x\phi(t, x, \tilde{\eta})d\tilde{\eta}.
\end{cases}
\end{align*}
With the above notations in hand, we denote the following new second order boundary layer profiles.
\begin{align*}
\widetilde{u_b^2}=\chi(y)u_b^2+\tau_u,\quad
\widetilde{v_b^2}=\chi(y)v_b^2,\quad
\widetilde{h_b^2}=\chi(y)h_b^2+\tau_h,\quad
\widetilde{g_b^2}=\chi(y)g_b^2+\tau_g.
\end{align*}
Then we can write the expressions of $R_i$ as follows.
\begin{align}\label{B.1}
R_0=&\left(\partial_x\rho^0-\overline{\partial_x\rho^0}+\sqrt\varepsilon(\partial_x\rho^1-\partial_x\varrho^1)+\varepsilon(\partial_x\rho^2-\partial_x\varrho^2)\right)u_b^0\notag\\
&+\sqrt{\varepsilon}\left(\partial_x\rho^0-\overline{\partial_x\rho^0}+\sqrt{\varepsilon}(\partial_x\rho^1-\partial_x\varrho^1)\right)u_b^1+\varepsilon\left(\partial_x\rho^0-\overline{\partial_x\rho^0}\right)\widetilde{u_b^2}\notag\\
&+\left(u^0-\overline{u^0}+\sqrt{\varepsilon}(u^1-\mathcal{U}^1)+\varepsilon(u^2-\mathcal{U}^2)\right)\partial_x\rho_b^0\notag\\
&+\sqrt{\varepsilon}\left(u^0-\overline{u^0}+\sqrt{\varepsilon}(u^1-\mathcal{U}^1)\right)\partial_x\rho_b^1+\varepsilon\left(u^0-\overline{u^0}\right)\partial_x\rho_b^2\notag\\
&+\sqrt{\varepsilon}\left(\partial_y\rho^0-\overline{\partial_y\rho^0}+\sqrt{\varepsilon}(\partial_y\rho^1-\partial_y\varrho^1)\right)v_b^0+\varepsilon\left(\partial_y\rho^0-\overline{\partial_y\rho^0}\right)v_b^1\notag\\
&+\left(v^0+\sqrt{\varepsilon}(v^1-\mathcal{V}^1)+\varepsilon(v^2-\mathcal{V}^2)-\varepsilon^{\frac32}\mathcal{V}^3\right)\partial_y\rho_b^0\notag\\
&+\sqrt{\varepsilon}\left(v^0+\sqrt{\varepsilon}(v^1-\mathcal{V}^1)+\varepsilon(v^2-\mathcal{V}^2)\right)\partial_y\rho_b^1+\varepsilon\left(v^0+\sqrt{\varepsilon}(v^1-\mathcal{V}^1)\right)\partial_y\rho_b^2\notag\\
&+\varepsilon\left(\widetilde{u_b^2}-u_b^2\right)\left(\overline{\partial_x\rho^0}+\partial_x\rho_b^0\right)+\varepsilon\left(\widetilde{v_b^2}-v_b^2\right)\partial_\eta\rho_b^0+\varepsilon^{\frac32}R^H_0,
\tag{B.1}
\end{align}
and
\begin{align}\label{B.2}
R_1=&\left(\partial_tu^0-\overline{\partial_tu^0}+\sqrt\varepsilon(\partial_tu^1-\partial_t\mathcal{U}^1)+\varepsilon(\partial_tu^2-\partial_t\mathcal{U}^2)\right)\rho_b^0\notag\\
&+\sqrt{\varepsilon}\left(\partial_tu^0-\overline{\partial_tu^0}+\sqrt{\varepsilon}(\partial_tu^1-\partial_t\mathcal{U}^1)\right)\rho_b^1+\varepsilon\left(\partial_tu^0-\overline{\partial_tu^0}\right)\rho_b^2\notag\\
&+\left(\rho^0-\overline{\rho^0}+\sqrt{\varepsilon}(\rho^1-\varrho^1)+\varepsilon(\rho^2-\varrho^2)\right)\partial_tu_b^0\notag\\
&+\sqrt{\varepsilon}\left(\rho^0-\overline{\rho^0}+\sqrt{\varepsilon}(\rho^1-\varrho^1)\right)\partial_tu_b^1+\varepsilon\left(\rho^0-\overline{\rho^0}\right)\partial_t\widetilde{u_b^2}+\varepsilon(\overline{\rho^0}+\rho_b^0)\left(\partial_t\widetilde{u_b^2}-\partial_tu_b^2\right)\notag\\
&+\left(\partial_xu^0-\overline{\partial_xu^0}+\sqrt{\varepsilon}(\partial_xu^1-\partial_x\mathcal{U}^1)+\varepsilon(\partial_xu^2-\partial_x\mathcal{U}^2)\right)(\overline{\rho^0}u_b^0+\rho_b^0\overline{u^0}+\rho_b^0u_b^0)\notag\\
&+\sqrt{\varepsilon}\left(\partial_xu^0-\overline{\partial_xu^0}+\sqrt{\varepsilon}(\partial_xu^1-\partial_x\mathcal{U}^1)\right)(\overline{\rho^0}u_b^1+\rho_b^0\mathcal{U}^1+\rho_b^0u_b^1+\rho_b^1\overline{u^0}+\rho_b^1u_b^0+\varrho^1u_b^0)\notag\\
&+\varepsilon(\partial_xu^0-\overline{\partial_xu^0})(\overline{\rho^0}\widetilde{u_b^2}+\rho_b^0\mathcal{U}^2+\rho_b^0\widetilde{u_b^2}+\varrho^1u_b^1+\rho_b^1\mathcal{U}^1+\rho_b^1u_b^1+\varrho^2u_b^0+\rho_b^2\overline{u^0}+\rho_b^2u_b^0)\notag\\
&+\left\{\left(\rho^0-\overline{\rho^0}+\sqrt{\varepsilon}(\rho^1-\varrho^1)+\varepsilon(\rho^2-\varrho^2)\right)\partial_xu^0\right.\notag\\
&\left.+\sqrt{\varepsilon}\left(\rho^0-\overline{\rho^0}+\sqrt{\varepsilon}(\rho^1-\varrho^1)\right)\partial_xu^1+\varepsilon(\rho^0-\overline{\rho^0})\partial_xu^2\right\}u_b^0\notag\\
&+\sqrt{\varepsilon}\left\{\left(\rho^0-\overline{\rho^0}+\sqrt{\varepsilon}(\rho^1-\varrho^1)\right)\partial_xu^0+\sqrt{\varepsilon}(\rho^0-\overline{\rho^0})\partial_xu^1\right\}u_b^1+\varepsilon(\rho^0-\overline{\rho^0})\partial_xu^0\widetilde{u_b^2}\notag\\
&+\left\{\left(u^0-\overline{u^0}+\sqrt{\varepsilon}(u^1-\mathcal{U}^1)+\varepsilon(u^2-\mathcal{U}^2)\right)\partial_xu^0\right.\notag\\
&\left.+\sqrt{\varepsilon}\left(u^0-\overline{u^0}+\sqrt{\varepsilon}(u^1-\mathcal{U}^1)\right)\partial_xu^1+\varepsilon(u^0-\overline{u^0})\partial_xu^2\right\}\rho_b^0\notag\\
&+\sqrt{\varepsilon}\left\{\left(u^0-\overline{u^0}+\sqrt{\varepsilon}(u^1-\mathcal{U}^1)\right)\partial_xu^0+\sqrt{\varepsilon}(u^0-\overline{u^0})\partial_xu^1\right\}\rho_b^1+\varepsilon(u^0-\overline{u^0})\partial_xu^0\rho_b^2\notag\\
&+\left\{\left(\rho^0-\overline{\rho^0}+\sqrt{\varepsilon}(\rho^1-\varrho^1)+\varepsilon(\rho^2-\varrho^2)\right)(u^0+u_b^0)\right.\notag\\
&\left.+\sqrt{\varepsilon}\left(\rho^0-\overline{\rho^0}+\sqrt{\varepsilon}(\rho^1-\varrho^1)\right)(u^1+u_b^1)+\varepsilon\left(\rho^0-\overline{\rho^0}\right)(u^2+\widetilde{u_b^2})\right.\notag\\
&\left.+(\overline{\rho^0}+\rho_b^0)\left(u^0-\overline{u^0}+\sqrt{\varepsilon}(u^1-\mathcal{U}^1)+\varepsilon(u^2-\mathcal{U}^2)\right)\right.\notag\\
&\left.+\sqrt{\varepsilon}(\varrho^1+\rho_b^1)\left(u^0-\overline{u^0}+\sqrt{\varepsilon}(u^1-\mathcal{U}^1)\right)+\varepsilon(\varrho^2+\rho_b^2)\left(u^0-\overline{u^0}\right)\right\}\partial_xu_b^0\notag\\
&+\sqrt{\varepsilon}\left\{\left(\rho^0-\overline{\rho^0}+\sqrt{\varepsilon}(\rho^1-\varrho^1)\right)(u^0+u_b^0)+\sqrt{\varepsilon}\left(\rho^0-\overline{\rho^0}\right)(u^1+u_b^1)\right.\notag\\
&\left.+(\overline{\rho^0}+\rho_b^0)\left(u^0-\overline{u^0}+\sqrt{\varepsilon}(u^1-\mathcal{U}^1)\right)+\sqrt{\varepsilon}(\varrho^1+\rho_b^1)\left(u^0-\overline{u^0}\right)\right\}\partial_xu_b^1\notag\\
&+\varepsilon\left\{(\overline{\rho^0}+\rho_b^0)(u^0-\overline{u^0})+(\rho^0-\overline{\rho^0})(\overline{u^0}+u_b^0)\right\}\partial_x\widetilde{u_b^2}\notag\\
&+\varepsilon(\overline{\rho^0}+\rho_b^0)(\overline{\partial_xu^0}+\partial_xu_b^0)\left(\widetilde{u_b^2}-u_b^2\right)
+\varepsilon\left[(\overline{\rho^0}+\rho_b^0)(\overline{u^0}+u_b^0)\right]\left(\partial_x\widetilde{u_b^2}-\partial_xu_b^2\right)\notag\\
&+\sqrt{\varepsilon}\left(\partial_yu^0-\overline{\partial_yu^0}+\sqrt{\varepsilon}(\partial_yu^1-\partial_y\mathcal{U}^1)\right)(\overline{\rho^0}v_b^0+\rho_b^0\mathcal{V}^1+\rho_b^0v_b^0)\notag\\
&+\varepsilon\left(\partial_yu^0-\overline{\partial_yu^0}\right)(\overline{\rho^0}v_b^1+\rho_b^0\mathcal{V}^2+\rho_b^0v_b^1+\varrho^1v_b^0+\rho_b^1\mathcal{V}^1+\rho_b^1v_b^0)\notag\\
&+\sqrt{\varepsilon}\left\{\left(\rho^0-\overline{\rho^0}+\sqrt{\varepsilon}(\rho^1-\varrho^1)\right)\partial_yu^0+\sqrt{\varepsilon}\left(\rho^0-\overline{\rho^0}\right)\partial_yu^1\right\}v_b^0\notag\\
&+\left\{\left(v^0+\sqrt{\varepsilon}(v^1-\mathcal{V}^1)+\varepsilon(v^2-\mathcal{V}^2)\right)\partial_yu^0+\sqrt{\varepsilon}\left(v^0+\sqrt{\varepsilon}(v^1-\mathcal{V}^1)\right)\partial_yu^1\right\}\rho_b^0\notag\\
&+\varepsilon\left(\rho^0-\overline{\rho^0}\right)\partial_yu^0v_b^1+\sqrt{\varepsilon}\left(v^0+\sqrt{\varepsilon}(v^1-\mathcal{V}^1)\right)\partial_yu^0\rho_b^1\notag\\
&+\left\{\left(\rho^0-\overline{\rho^0}+\sqrt{\varepsilon}(\rho^1-\varrho^1)+\varepsilon(\rho^2-\varrho^2)\right)(v^0+\sqrt{\varepsilon}v_b^0+\sqrt{\varepsilon}v^1)\right.\notag\\
&\left.+\varepsilon\left(\rho^0-\overline{\rho^0}+\sqrt{\varepsilon}(\rho^1-\varrho^1)\right)(v_b^1+v^2)+\varepsilon^{\frac32}\left(\rho^0-\overline{\rho^0}\right)\widetilde{v_b^2}\right.\notag\\
&\left.+(\overline{\rho^0}+\rho_b^0)\left(v^0+\sqrt{\varepsilon}(v^1-\mathcal{V}^1)+\varepsilon(v^2-\mathcal{V}^2)-\varepsilon^{\frac32}\mathcal{V}^3\right)\right.\notag\\
&\left.+\sqrt{\varepsilon}(\varrho^1+\rho_b^1)\left(v^0+\sqrt{\varepsilon}(v^1-\mathcal{V}^1)+\varepsilon(v^2-\mathcal{V}^2)\right)+\varepsilon(\varrho^2+\rho_b^2)\left(v^0+\sqrt{\varepsilon}(v^1-\mathcal{V}^1)\right)\right\}\partial_y u_b^0\notag\\
&+\sqrt{\varepsilon}\left\{\left(\rho^0-\overline{\rho^0}+\sqrt{\varepsilon}(\rho^1-\varrho^1)\right)(v^0+\sqrt{\varepsilon}v_b^0+\sqrt{\varepsilon}v^1)+\sqrt{\varepsilon}\left(\rho^0-\overline{\rho^0}\right)(v_b^1+v^2)\right.\notag\\
&\left.+(\overline{\rho^0}+\rho_b^0)\left(v^0+\sqrt{\varepsilon}(v^1-\mathcal{V}^1)+\sqrt{\varepsilon}(v^2-\mathcal{V}^2)\right)+\sqrt{\varepsilon}(\varrho^1+\rho_b^1)\left(v^0+\sqrt{\varepsilon}(v^1-\mathcal{V}^1)\right)\right\}\partial_y u_b^1\notag\\
&+\varepsilon\left\{(\overline{\rho^0}+\rho_b^0)\left(v^0+\sqrt{\varepsilon}(v^1-\mathcal{V}^1)\right)+\sqrt{\varepsilon}(\rho^0-\overline{\rho^0})(v^0+\sqrt{\varepsilon}v_b^0+\sqrt{\varepsilon}v^1)\right\}\partial_y\widetilde{u_b^2}\notag\\
&+\varepsilon(\overline{\rho^0}+\rho_b^0)\left(\widetilde{v_b^2}-v_b^2\right)\partial_\eta u_b^0+\varepsilon\left[(\overline{\rho^0}+\rho_b^0)(\mathcal{V}^1+v_b^0)\right]\left(\partial_\eta\widetilde{u_b^2}-\partial_\eta u_b^2\right)\notag\\
&-\left(\partial_xh^0-\overline{\partial_xh^0}+\sqrt\varepsilon(\partial_xh^1-\partial_x\mathcal{H}^1)+\varepsilon(\partial_xh^2-\partial_x\mathcal{H}^2)\right)h_b^0\notag\\
&-\sqrt{\varepsilon}\left(\partial_xh^0-\overline{\partial_xh^0}+\sqrt{\varepsilon}(\partial_xh^1-\partial_x\mathcal{H}^1)\right)h_b^1-\varepsilon\left(\partial_xh^0-\overline{\partial_xh^0}\right)\widetilde{h_b^2}\notag\\
&-\left(h^0-\overline{h^0}+\sqrt{\varepsilon}(h^1-\mathcal{H}^1)+\varepsilon(h^2-\mathcal{H}^2)\right)\partial_xh_b^0\notag\\
&-\sqrt{\varepsilon}\left(h^0-\overline{h^0}+\sqrt{\varepsilon}(h^1-\mathcal{H}^1)\right)\partial_xh_b^1-\varepsilon\left(h^0-\overline{h^0}\right)\partial_x\widetilde{h_b^2}\notag\\
&-\sqrt{\varepsilon}\left(\partial_yh^0-\overline{\partial_yh^0}+\sqrt{\varepsilon}(\partial_yh^1-\partial_y\mathcal{H}^1)\right)g_b^0-\varepsilon\left(\partial_yh^0-\overline{\partial_yh^0}\right)g_b^1\notag\\
&-\left(g^0+\sqrt{\varepsilon}(g^1-\mathcal{G}^1)+\varepsilon(g^2-\mathcal{G}^2)-\varepsilon^{\frac32}\mathcal{G}^3\right)\partial_yh_b^0\notag\\
&-\sqrt{\varepsilon}\left(g^0+\sqrt{\varepsilon}(g^1-\mathcal{G}^1)+\varepsilon(g^2-\mathcal{G}^2)\right)\partial_yh_b^1-\varepsilon\left(g^0+\sqrt{\varepsilon}(g^1-\mathcal{G}^1)\right)\partial_y\widetilde{h_b^2}\notag\\
&-\varepsilon\left(\overline{h^0}+h_b^0\right)\left(\partial_x\widetilde{h_b^2}-\partial_xh_b^2\right)-\varepsilon\left(\widetilde{h_b^2}-h_b^2\right)\left(\overline{\partial_xh^0}+\partial_xh_b^0\right)\notag\\
&-\varepsilon(\mathcal{G}^1+g_b^0)\left(\partial_\eta\widetilde{h_b^2}-\partial_\eta h_b^2\right)
-\varepsilon\left(\widetilde{g_b^2}-g_b^2\right)\partial_\eta h_b^0-\mu\varepsilon(\partial_{\eta\eta}\widetilde{u_b^2}-\partial_{\eta\eta}u_b^2)+\varepsilon^{\frac32}R_1^H,
\tag{B.2}
\end{align}
and
\begin{align}\label{B.3}
R_2=&\left(\partial_tv^0+\sqrt\varepsilon(\partial_tv^1-\partial_t\mathcal{V}^1)+\varepsilon(\partial_tv^2-\partial_t\mathcal{V}^2)\right)\rho_b^0+\sqrt{\varepsilon}\left(\partial_tv^0+\sqrt{\varepsilon}(\partial_tv^1-\partial_t\mathcal{V}^1)\right)\rho_b^1\notag\\
&+\sqrt{\varepsilon}\left(\rho^0-\overline{\rho^0}+\sqrt{\varepsilon}(\rho^1-\varrho^1)\right)\partial_tv_b^0+\varepsilon\left(\rho^0-\overline{\rho^0}\right)\partial_tv_b^1\notag\\
&+\left[\partial_xv^0+\sqrt{\varepsilon}(\partial_xv^1-\partial_x\mathcal{V}^1)+\varepsilon(\partial_xv^2-\partial_x\mathcal{V}^2)\right](\overline{\rho^0}u_b^0+\rho_b^0\overline{u^0}+\rho_b^0u_b^0)\notag\\
&+\sqrt{\varepsilon}\left[\partial_xv^0+\sqrt{\varepsilon}(\partial_xv^1-\partial_x\mathcal{V}^1)\right](\overline{\rho^0}u_b^1+\rho_b^0\mathcal{U}^1+\rho_b^0u_b^1+\rho_b^1\overline{u^0}+\rho_b^1u_b^0+\varrho^1u_b^0)\notag\\
&+\left\{\left[\rho^0-\overline{\rho^0}+\sqrt{\varepsilon}(\rho^1-\varrho^1)\right]\left(\partial_xv^0+\sqrt{\varepsilon}\partial_xv^1\right)+\varepsilon(\rho^0-\overline{\rho^0})\partial_xv^2\right\}u_b^0\notag\\
&+\left\{\left[u^0-\overline{u^0}+\sqrt{\varepsilon}(u^1-\mathcal{U}^1)\right]\left(\partial_xv^0+\sqrt{\varepsilon}\partial_xv^1\right)+\varepsilon(u^0-\overline{u^0})\partial_xv^2\right\}\rho_b^0\notag\\
&+\sqrt{\varepsilon}\left[\left(\rho^0-\overline{\rho^0}\right)\left(\partial_xv^0+\sqrt{\varepsilon}\partial_xv^1\right)\right]u_b^1+\sqrt{\varepsilon}\left[\left(u^0-\overline{u^0}\right)\left(\partial_xv^0+\sqrt{\varepsilon}\partial_xv^1\right)\right]\rho_b^1\notag\\
&+\sqrt{\varepsilon}\left\{\left[\rho^0-\overline{\rho^0}+\sqrt{\varepsilon}(\rho^1-\varrho^1)\right](u^0+u_b^0)+\sqrt{\varepsilon}\left(\rho^0-\overline{\rho^0}\right)(u^1+u_b^1)\right.\notag\\
&\left.+(\overline{\rho^0}+\rho_b^0)\left[u^0-\overline{u^0}+\sqrt{\varepsilon}(u^1-\mathcal{U}^1)\right]+\sqrt{\varepsilon}(\varrho^1+\rho_b^1)\left(u^0-\overline{u^0}\right)\right\}\partial_xv_b^0\notag\\
&+\varepsilon\left[(\overline{\rho^0}+\rho_b^0)\left(u^0-\overline{u^0}\right)+\left(\rho^0-\overline{\rho^0}\right)(\overline{u^0}+u_b^0)\right]\partial_xv_b^1\notag\\
&+\sqrt{\varepsilon}\left(\partial_yv^0-\overline{\partial_yv^0}+\sqrt{\varepsilon}(\partial_yv^1-\partial_y\mathcal{V}^1)\right)(\overline{\rho^0}v_b^0+\rho_b^0\mathcal{V}^1+\rho_b^0v_b^0)\notag\\
&+\varepsilon\left(\partial_yv^0-\overline{\partial_yv^0}\right)(\overline{\rho^0}v_b^1+\rho_b^0\mathcal{V}^2+\rho_b^0v_b^1+\varrho^1v_b^0+\rho_b^1\mathcal{V}^1+\rho_b^1v_b^0)\notag\\
&+\sqrt{\varepsilon}\left\{\left[\rho^0-\overline{\rho^0}+\sqrt{\varepsilon}(\rho^1-\varrho^1)\right]\partial_yv^0+\sqrt{\varepsilon}\left(\rho^0-\overline{\rho^0}\right)\partial_yv^1\right\}v_b^0\notag\\
&+\left\{\left[v^0+\sqrt{\varepsilon}(v^1-\mathcal{V}^1)+\varepsilon(v^2-\mathcal{V}^2)\right]\partial_yv^0+\sqrt{\varepsilon}\left[v^0+\sqrt{\varepsilon}(v^1-\mathcal{V}^1)\right]\partial_yv^1\right\}\rho_b^0\notag\\
&+\varepsilon\left(\rho^0-\overline{\rho^0}\right)\partial_yv^0v_b^1+\sqrt{\varepsilon}\left[v^0+\sqrt{\varepsilon}(v^1-\mathcal{V}^1)\right]\partial_yv^0\rho_b^1\notag\\
&+\left\{\sqrt{\varepsilon}\left[\rho^0-\overline{\rho^0}+\sqrt{\varepsilon}(\rho^1-\varrho^1)\right](v_b^0+v^1)+\varepsilon\left(\rho^0-\overline{\rho^0}\right)(v_b^1+v^2)\right.\notag\\
&\left.+(\overline{\rho^0}+\rho_b^0)\left[v^0+\sqrt{\varepsilon}(v^1-\mathcal{V}^1)+\varepsilon(v^2-\mathcal{V}^2)\right]+\sqrt{\varepsilon}(\varrho^1+\rho_b^1)\left(v^0+\sqrt{\varepsilon}(v^1-\mathcal{V}^1)\right)\right\}\partial_\eta v_b^0\notag\\
&+\sqrt{\varepsilon}\left\{\sqrt{\varepsilon}\left(\rho^0-\overline{\rho^0}\right)(v_b^0+v^1)+(\overline{\rho^0}+\rho_b^0)\left[v^0+\sqrt{\varepsilon}(v^1-\mathcal{V}^1)\right]\right\}\partial_\eta v_b^1\notag\\
&-\left(\partial_xg^0+\sqrt{\varepsilon}(\partial_xg^1-\partial_x\mathcal{G}^1)+\varepsilon(\partial_xg^2-\partial_x\mathcal{G}^2)\right)h_b^0-\sqrt{\varepsilon}\left(\partial_xg^0+\sqrt{\varepsilon}(\partial_xg^1-\partial_x\mathcal{G}^1)\right)h_b^1\notag\\
&-\sqrt{\varepsilon}\left(h^0-\overline{h^0}+\sqrt{\varepsilon}(h^1-\mathcal{H}^1)\right)\partial_xg_b^0-\varepsilon\left(h^0-\overline{h^0}\right)\partial_xg_b^1\notag\\
&-\sqrt{\varepsilon}\left(\partial_yg^0-\overline{\partial_yg^0}+\sqrt{\varepsilon}(\partial_yg^1-\partial_y\mathcal{G}^1)\right)g_b^0-\varepsilon\left(\partial_yg^0-\overline{\partial_yg^0}\right)g_b^1\notag\\
&-\left(g^0+\sqrt{\varepsilon}(g^1-\mathcal{G}^1)+\varepsilon(g^2-\mathcal{G}^2)\right)\partial_\eta g_b^0-\sqrt{\varepsilon}\left(g^0+\sqrt{\varepsilon}(g^1-\mathcal{G}^1)\right)\partial_\eta g_b^1+\varepsilon^{\frac32}R_2^H,
\tag{B.3}
\end{align}
and
\begin{align}\label{B.4}
R_3=&\left(\partial_xh^0-\overline{\partial_xh^0}+\sqrt\varepsilon(\partial_xh^1-\partial_x\mathcal{H}^1)+\varepsilon(\partial_xh^2-\partial_x\mathcal{H}^2)\right)u_b^0\notag\\
&+\sqrt{\varepsilon}\left(\partial_xh^0-\overline{\partial_xh^0}+\sqrt{\varepsilon}(\partial_xh^1-\partial_x\mathcal{H}^1)\right)u_b^1+\varepsilon\left(\partial_xh^0-\overline{\partial_xh^0}\right)\widetilde{u_b^2}\notag\\
&+\left(u^0-\overline{u^0}+\sqrt{\varepsilon}(u^1-\mathcal{U}^1)+\varepsilon(u^2-\mathcal{U}^2)\right)\partial_xh_b^0\notag\\
&+\sqrt{\varepsilon}\left(u^0-\overline{u^0}+\sqrt{\varepsilon}(u^1-\mathcal{U}^1)\right)\partial_xh_b^1+\varepsilon\left(u^0-\overline{u^0}\right)\partial_x\widetilde{h_b^2}\notag\\
&+\sqrt{\varepsilon}\left(\partial_yh^0-\overline{\partial_yh^0}+\sqrt{\varepsilon}(\partial_yh^1-\partial_y\mathcal{H}^1)\right)v_b^0+\varepsilon\left(\partial_yh^0-\overline{\partial_yh^0}\right)v_b^1\notag\\
&+\left(v^0+\sqrt{\varepsilon}(v^1-\mathcal{V}^1)+\varepsilon(v^2-\mathcal{V}^2)-\varepsilon^{\frac32}\mathcal{V}^3\right)\partial_yh_b^0\notag\\
&+\sqrt{\varepsilon}\left(v^0+\sqrt{\varepsilon}(v^1-\mathcal{V}^1)+\varepsilon(v^2-\mathcal{V}^2)\right)\partial_yh_b^1+\varepsilon\left(v^0+\sqrt{\varepsilon}(v^1-\mathcal{V}^1)\right)\partial_y\widetilde{h_b^2}\notag\\
&+\varepsilon\left(\partial_t\widetilde{h_b^2}-\partial_th_b^2\right)+\varepsilon\left(\overline{u^0}+u_b^0\right)\left(\partial_x\widetilde{h_b^2}-\partial_xh_b^2\right)+\varepsilon\left(\widetilde{u_b^2}-u_b^2\right)\left(\overline{\partial_xh^0}+\partial_xh_b^0\right)\notag\\
&+\sqrt{\varepsilon}\left(v^0+\sqrt{\varepsilon}(\mathcal{V}^1+v_b^0)\right)\left(\partial_\eta\widetilde{h_b^2}-\partial_\eta h_b^2\right)+\varepsilon\left(\widetilde{v_b^2}-v_b^2\right)\partial_\eta h_b^0\notag\\
&-\left(\partial_xu^0-\overline{\partial_xu^0}+\sqrt\varepsilon(\partial_xu^1-\partial_x\mathcal{U}^1)+\varepsilon(\partial_xu^2-\partial_x\mathcal{U}^2)\right)h_b^0\notag\\
&-\sqrt{\varepsilon}\left(\partial_xu^0-\overline{\partial_xu^0}+\sqrt{\varepsilon}(\partial_xu^1-\partial_x\mathcal{U}^1)\right)h_b^1-\varepsilon\left(\partial_xu^0-\overline{\partial_xu^0}\right)\widetilde{h_b^2}\notag\\
&-\left(h^0-\overline{h^0}+\sqrt{\varepsilon}(h^1-\mathcal{H}^1)+\varepsilon(h^2-\mathcal{H}^2)\right)\partial_xu_b^0\notag\\
&-\sqrt{\varepsilon}\left(h^0-\overline{h^0}+\sqrt{\varepsilon}(h^1-\mathcal{H}^1))\right)\partial_xu_b^1-\varepsilon\left(h^0-\overline{h^0}\right)\partial_x\widetilde{u_b^2}\notag\\
&-\sqrt{\varepsilon}\left(\partial_yu^0-\overline{\partial_yu^0}+\sqrt{\varepsilon}(\partial_yu^1-\partial_y\mathcal{U}^1)\right)g_b^0-\varepsilon\left(\partial_yu^0-\overline{\partial_yu^0}\right)g_b^1\notag\\
&-\left(g^0+\sqrt{\varepsilon}(g^1-\mathcal{G}^1)+\varepsilon(g^2-\mathcal{G}^2)-\varepsilon^{\frac32}\mathcal{G}^3\right)\partial_yu_b^0\notag\\
&-\sqrt{\varepsilon}\left(g^0+\sqrt{\varepsilon}(g^1-\mathcal{G}^1)+\varepsilon(g^2-\mathcal{G}^2)\right)\partial_yu_b^1-\varepsilon\left(g^0+\sqrt{\varepsilon}(g^1-\mathcal{G}^1)\right)\partial_y\widetilde{u_b^2}\notag\\
&-\varepsilon\left(\overline{h^0}+h_b^0\right)\left(\partial_x\widetilde{u_b^2}-\partial_xu_b^2\right)-\varepsilon\left(\widetilde{h_b^2}-h_b^2\right)\left(\overline{\partial_xu^0}+\partial_xu_b^0\right)\notag\\
&-\sqrt{\varepsilon}\left(g^0+\sqrt{\varepsilon}(\mathcal{G}^1+g_b^0)\right)\left(\partial_\eta\widetilde{u_b^2}-\partial_\eta u_b^2\right)
-\varepsilon\left(\widetilde{g_b^2}-g_b^2\right)\partial_\eta u_b^0\notag\\
&-\kappa\varepsilon(\partial_{\eta\eta}\widetilde{h_b^2}-\partial_{\eta\eta}h_b^2)+R^H_3,
\tag{B.4}
\end{align}
and
\begin{align}\label{B.5}
R_4=&\left(\partial_xg^0+\sqrt\varepsilon(\partial_xg^1-\partial_x\mathcal{G}^1)+\varepsilon(\partial_xg^2-\partial_x\mathcal{G}^2)\right)u_b^0+\sqrt{\varepsilon}\left(\partial_xg^0+\sqrt{\varepsilon}(\partial_xg^1-\partial_x\mathcal{G}^1)\right)u_b^1\notag\\
&+\sqrt{\varepsilon}\left(u^0-\overline{u^0}+\sqrt{\varepsilon}(u^1-\mathcal{U}^1)\right)\partial_xg_b^0+\varepsilon\left(u^0-\overline{u^0}\right)\partial_xg_b^1\notag\\
&+\sqrt{\varepsilon}\left(\partial_yg^0-\overline{\partial_yg^0}+\sqrt{\varepsilon}(\partial_yg^1-\partial_y\mathcal{G}^1)\right)v_b^0+\varepsilon\left(\partial_yg^0-\overline{\partial_yg^0}\right)v_b^1\notag\\
&+\left(v^0+\sqrt{\varepsilon}(v^1-\mathcal{V}^1)+\varepsilon(v^2-\mathcal{V}^2)\right)\partial_\eta g_b^0+\sqrt{\varepsilon}\left(v^0+\sqrt{\varepsilon}(v^1-\mathcal{V}^1)\right)\partial_\eta g_b^1\notag\\
&-\left(\partial_xv^0+\sqrt\varepsilon(\partial_xv^1-\partial_x\mathcal{V}^1)+\varepsilon(\partial_xv^2-\partial_x\mathcal{V}^2)\right)h_b^0-\sqrt{\varepsilon}\left(\partial_xv^0+\sqrt{\varepsilon}(\partial_xv^1-\partial_x\mathcal{V}^1)\right)h_b^1\notag\\
&-\sqrt{\varepsilon}\left(h^0-\overline{h^0}+\sqrt{\varepsilon}(h^1-\mathcal{H}^1)\right)\partial_xv_b^0-\varepsilon\left(h^0-\overline{h^0}\right)\partial_xv_b^1\notag\\
&-\sqrt{\varepsilon}\left(\partial_yv^0-\overline{\partial_yv^0}+\sqrt{\varepsilon}(\partial_yv^1-\partial_y\mathcal{V}^1)\right)g_b^0-\varepsilon\left(\partial_yg^0-\overline{\partial_yg^0}\right)g_b^1\notag\\
&-\left(g^0+\sqrt{\varepsilon}(g^1-\mathcal{G}^1)+\varepsilon(g^2-\mathcal{G}^2)\right)\partial_\eta v_b^0-\sqrt{\varepsilon}\left(g^0+\sqrt{\varepsilon}(g^1-\mathcal{G}^1)\right)\partial_\eta v_b^1+\varepsilon^{\frac32}R_4^H.
\tag{B.5}
\end{align}
Here $\varepsilon^{\frac32}R_i^H\ (i=0,1,2,3,4)$ are the terms of order $\varepsilon^{\frac32}$ and higher order.
\begin{align}\label{B.6}
R^H_0=&(u^2+\widetilde{u_b^2})\partial_x(\rho^1+\rho_b^1)+\left[(u^1+u_b^1)+\sqrt{\varepsilon}(u^2+\widetilde{u_b^2})\right]\partial_x(\rho^2+\rho_b^2)\notag\\
&+\widetilde{v_b^2}(\partial_y\rho^0+\partial_\eta\rho_b^1)+\left[(v_b^1+v^2)+\sqrt{\varepsilon}\widetilde{v_b^2}\right](\partial_y\rho^1+\partial_\eta\rho_b^2)\notag\\
&+\left[(v_b^0+v^1)+\sqrt{\varepsilon}(v_b^1+v^2)+\varepsilon\widetilde{v_b^2}\right]\partial_y\rho^2,
\tag{B.6}
\end{align}
and
\begin{align}\label{B.7}
R_1^H=&(\rho^2+\rho_b^2)\partial_t(u^1+u_b^1)+\left[(\rho^1+\rho_b^1)+\sqrt{\varepsilon}(\rho^2+\rho_b^2)\right]\partial_t(u^2+\widetilde{u_b^2})\notag\\
&\left[(\rho^1+\rho_b^1)(u^2+\widetilde{u_b^2})+(\rho^2+\rho_b^2)(u_1+u_b^1)+\sqrt{\varepsilon}(\rho^2+\rho_b^2)(u^2+\widetilde{u_b^2})\right](\partial_xu^0+\partial_xu_b^0)\notag\\
&+\left\{(\rho^0+\rho_b^0)(u^2+\widetilde{u_b^2})+(\rho^1+\rho_b^1)(u^1+u_b^1)+(\rho^2+\rho_b^2)(u^0+u_b^0)\right.\notag\\
&\left.+\sqrt{\varepsilon}\left[(\rho^1+\rho_b^1)(u^2+\widetilde{u_b^2})+(\rho^2+\rho_b^2)(u^1+u_b^1)\right]+\varepsilon(\rho^2+\rho_b^2)(u^2+\widetilde{u_b^2})\right\}(\partial_xu^1+\partial_xu_b^1)\notag\\
&+\left\{(\rho^0+\rho_b^0)(u^1+u_b^1)+(\rho^1+\rho_b^1)(u^0+u_b^0)\right.\notag\\
&\left.+\sqrt{\varepsilon}\left[(\rho^0+\rho_b^0)(u^2+\widetilde{u_b^2})+(\rho^1+\rho_b^1)(u^1+u_b^1)+(\rho^2+\rho_b^2)(u^0+u_b^0)\right]\right.\notag\\
&\left.+\varepsilon\left[(\rho^1+\rho_b^1)(u^2+\widetilde{u_b^2})+(\rho^2+\rho_b^2)(u^1+u_b^1)\right]+\varepsilon^{\frac32}(\rho^2+\rho_b^2)(u^2+\widetilde{u_b^2})\right\}(\partial_xu^2+\partial_x\widetilde{u_b^2})\notag\\
&+\left\{(\rho^1+\rho_b^1)\widetilde{v_b^2}+(\rho^2+\rho_b^2)\left((v_b^1+v^2)+\sqrt{\varepsilon}\widetilde{v_b^2}\right)\right\}\partial_\eta u_b^0\notag\\
&+\left\{(\rho^0+\rho_b^0)\widetilde{v_b^2}+(\rho^1+\rho_b^1)(v_b^1+v^2)+(\rho^2+\rho_b^2)(v_b^0+v^1)+\rho_b^2v^0\right.\notag\\
&\left.+\sqrt{\varepsilon}\left[(\rho^1+\rho_b^1)\widetilde{v_b^2}+(\rho^2+\rho_b^2)(v_b^1+v^2)\right]+\varepsilon(\rho^2+\rho_b^2)\widetilde{v_b^2}\right\}(\partial_yu^0+\partial_\eta u_b^1)\notag\\
&+\rho^2v^0\partial_\eta u_b^1+\left\{(\rho^0+\rho_b^0)(v_b^1+v^2)+(\rho^1+\rho_b^1)(v_b^0+v^1)+\rho_b^1v^0\right.\notag\\
&\left.+\sqrt{\varepsilon}\left[(\rho^0+\rho_b^0)\widetilde{v_b^2}+(\rho^1+\rho_b^1)(v_b^1+v^2)+(\rho^2+\rho_b^2)(v_b^0+v^1)+\rho_b^2v^0\right]\right.\notag\\
&\left.+\varepsilon\left[(\rho^1+\rho_b^1)\widetilde{v_b^2}+(\rho^2+\rho_b^2)(v_b^1+v^2)\right]+\varepsilon^{\frac32}(\rho^2+\rho_b^2)\widetilde{v_b^2}\right\}(\partial_yu^1+\partial_\eta\widetilde{u_b^2})\notag\\
&+\rho^2v^0(\partial_yu^1+\partial_\eta u_b^1)+\sqrt{\varepsilon}\rho^2v^0\partial_\eta\widetilde{u_b^2}+\varepsilon^{-\frac12}(\rho^av^a-\rho^0v^0)\partial_yu^2\notag\\
&-(h^2+\widetilde{h_b^2})\partial_x(h^1+h_b^1)-\left[(h^1+h_b^1)+\sqrt{\varepsilon}(h^2+\widetilde{h_b^2})\right]\partial_x(h^2+\widetilde{h_b^2})\notag\\
&-\widetilde{g_b^2}(\partial_yh^0+\partial_\eta h_b^1)-\left[(g_b^1+g^2)+\sqrt{\varepsilon}\widetilde{g_b^2}\right](\partial_yh^1+\partial_\eta\widetilde{h_b^2})\notag\\
&+\left[(g_b^0+g^1)+\sqrt{\varepsilon}(g_b^1+g^2)+\varepsilon\widetilde{g_b^2}\right]\partial_yh^2-\mu(\partial_{xx}u_b^1+\sqrt{\varepsilon}\partial_{xx}\widetilde{u_b^2})-\mu(\Delta u^1+\sqrt{\varepsilon}\Delta u^2),
\tag{B.7}
\end{align}
and
\begin{align}\label{B.8}
R^H_2=&\rho^a\partial_t\widetilde{v_b^2}+\rho^au^a\partial_x\widetilde{v_b^2}+\rho^av^a\partial_y\widetilde{v_b^2}-h^a\partial_x\widetilde{g_b^2}-g^a\partial_y\widetilde{g_b^2}\notag\\
&+\varepsilon^{-\frac12}\rho_b^2\partial_tv^0+(\rho^2+\rho_b^2)\partial_t(v_b^0+v^1)+\left(\rho^1+\rho_b^1+\sqrt{\varepsilon}(\rho^2+\rho_b^2)\right)\partial_t(v_b^1+v^2)\notag\\
&+\varepsilon^{-\frac12}\left[(\rho_b^2u^a+\rho^2(u^a-u^0)+\rho_b^1(u^1+u_b^1)+\rho^1u_b^1+\rho_b^0(u^2+\widetilde{u_b^2})+\rho^0\widetilde{u_b^2}\right]\partial_xv^0\notag\\
&+\left\{(\rho^1+\rho_b^1)(u^2+\widetilde{u_b^2})+(\rho^2+\rho_b^2)\left[(u^1+u_b^1)+\sqrt{\varepsilon}(u^2+\widetilde{u_b^2})\right]\right\}\partial_xv^0\notag\\
&+\left\{(\rho^0+\rho_b^0)(u^2+\widetilde{u_b^2})+(\rho^1+\rho_b^1)\left[(u^1+u_b^1)+\sqrt{\varepsilon}(u^2+\widetilde{u_b^2})\right]\right.\notag\\
&\left.+(\rho^2+\rho_b^2)u^a\right\}(\partial_xv_b^0+\partial_xv^1)+\varepsilon^{-\frac12}\left(\rho^au^a-(\rho^0+\rho_b^0)(u^0+u_b^0)\right)(\partial_xv_b^1+\partial_xv^2)\notag\\
&+\left\{(\rho^0+\rho_b^0)\widetilde{v_b^2}+(\rho^1+\rho_b^1)(v_b^1+v^2+\sqrt{\varepsilon}\widetilde{v_b^2})\right\}(\partial_yv^0+\partial_\eta v_b^0)\notag\\
&+\varepsilon^{-\frac12}\left[\rho^2(v^a-v^0)\partial_yv^0+\rho^2v^a\partial_\eta v_b^0+\rho_b^2v^a(\partial_yv^0+\partial_\eta v_b^0)\right]\notag\\
&+\left\{(\rho^0+\rho_b^0)\left(v_b^1+v^2+\sqrt{\varepsilon}\widetilde{v_b^2}\right)+(\rho^2+\rho_b^2)v^a\right\}(\partial_yv^1+\partial_\eta v_b^1)\notag\\
&+\varepsilon^{-\frac12}\left[\rho^1(v^a-v^0)\partial_yu^1+\rho^1v^a\partial_\eta v_b^1+\rho_b^1v^a(\partial_yv^1+\partial_\eta v_b^1\right]+\varepsilon^{-\frac12}(\rho^av^a-\rho^0v^0)\partial_yv^2\notag\\
&-\varepsilon^{-\frac12}\widetilde{h_b^2}\partial_xg^0-(h^2+\widetilde{h_b^2})\partial_x(g_b^0+g^1)-\left[(h^1+h_b^1)+\sqrt{\varepsilon}(h^2+\widetilde{h_b^2})\right]\partial_x(g_b^1+g^2)\notag\\
&-\widetilde{g_b^2}(\partial_yg^0+\partial_\eta g_b^0)-\left[(g_b^1+g^2)+\sqrt{\varepsilon}\widetilde{g_b^2}\right](\partial_yg^1+\partial_\eta g_b^1)-\varepsilon^{-\frac12}(g^a-g^0)\partial_yg^2\notag\\
&-\mu(\partial_{xx}v_b^0+\sqrt{\varepsilon}\partial_{xx}\widetilde{v_b^1})-\mu(\Delta v^1+\sqrt{\varepsilon}\Delta v^2+\varepsilon\Delta\widetilde{v_b^2}),
\tag{B.8}
\end{align}
and
\begin{align}\label{B.9}
R^H_3=&(u^2+\widetilde{u_b^2})\partial_x(h^1+h_b^1)+\left[(u^1+u_b^1)+\sqrt{\varepsilon}(u^2+\widetilde{u_b^2})\right]\partial_x(h^2+\widetilde{h_b^2})\notag\\
&+\widetilde{v_b^2}(\partial_yh^0+\partial_\eta h_b^1)+\left[(v_b^1+v^2)+\sqrt{\varepsilon}\widetilde{v_b^2}\right](\partial_yh^1+\partial_\eta\widetilde{h_b^2})+\varepsilon^{-\frac12}(v^a-v^0)\partial_yh^2\notag\\
&-(h^2+\widetilde{h_b^2})\partial_x(u^1+u_b^1)-\left[(h^1+h_b^1)+\sqrt{\varepsilon}(h^2+\widetilde{h_b^2})\right]\partial_x(u^2+\widetilde{u_b^2})\notag\\
&-\widetilde{g_b^2}(\partial_yu^0+\partial_\eta u_b^1)-\left[(g_b^1+g^2)+\sqrt{\varepsilon}\widetilde{g_b^2}\right](\partial_yu^1+\partial_\eta\widetilde{u_b^2})-\varepsilon^{-\frac12}(g^a-g^0)\partial_yu^2\notag\\
&-\kappa(\partial_{xx}h_b^1+\sqrt{\varepsilon}\partial_{xx}\widetilde{h_b^2})-\kappa(\Delta h^1+\sqrt{\varepsilon}\Delta h^2),
\tag{B.9}
\end{align}
and
\begin{align}\label{B.10}
R^H_4=&\partial_t\widetilde{g_b^2}+u^a\partial_x\widetilde{g_b^2}+v^a\partial_y\widetilde{g_b^2}-h^a\partial_x\widetilde{v_b^2}-g^a\partial_y\widetilde{v_b^2}\notag\\
&+\varepsilon^{-\frac12}\widetilde{u_b^2}\partial_xg^0+(u^2+\widetilde{u_b^2})\partial_x(g_b^0+g^1)+\left[(u^1+u_b^1)+\sqrt{\varepsilon}(u^2+\widetilde{u_b^2})\right]\partial_x(g_b^1+g^2)\notag\\
&+\widetilde{v_b^2}(\partial_yg^0+\partial_\eta g_b^0)+\left[(v_b^1+v^2)+\sqrt{\varepsilon}\widetilde{v_b^2}\right](\partial_yg^1+\partial_\eta g_b^1)+\varepsilon^{-\frac12}(v^a-v^0)\partial_yg^2\notag\\
&-\varepsilon^{-\frac12}\widetilde{h_b^2}\partial_xv^0-(h^2+\widetilde{h_b^2})\partial_x(v_b^0+v^1)-\left[(h^1+h_b^1)+\sqrt{\varepsilon}(h^2+\widetilde{h_b^2})\right]\partial_x(v_b^1+v^2)\notag\\
&-\widetilde{h_b^2}(\partial_yv^0+\partial_\eta v_b^0)-\left[(g_b^1+g^2)+\sqrt{\varepsilon}\widetilde{g_b^2}\right](\partial_yv^1+\partial_\eta v_b^1)-\varepsilon^{-\frac12}(g^a-g^0)\partial_yv^2\notag\\
&-\kappa(\partial_{xx}g_b^0+\sqrt{\varepsilon}\partial_{xx}\widetilde{g_b^1})-\kappa(\Delta g^1+\sqrt{\varepsilon}\Delta g^2+\varepsilon\Delta\widetilde{g_b^2}).
\tag{B.10}
\end{align}
With above expressions of the remainders, we are ready to prove Proposition \ref{remainder}.
\begin{proof}[Proof of Proposition \ref{remainder}]
It is sufficient to prove the case that $|\alpha|=0$ and $k=0$. For simplicity of presentation, we only consider the estimate of $R_1$, and the other terms can be estimated in a same way. We learn from Taylor's expansion that
\begin{align*}
u^0=\overline{u^0}+\sqrt{\varepsilon}\eta\overline{\partial_yu^0}+\frac12\varepsilon^2\eta^2\overline{\partial_y^2u^0}+O(\varepsilon^{\frac32}),\quad u^1=\overline{u^1}+\sqrt{\varepsilon}\eta\overline{\partial_yu^1}+O(\varepsilon).
\end{align*}
Recall the definition of $\mathcal{U}^1$ and $\mathcal{U}^2$, we find that
\begin{align*}
u^0-\overline{u^0}+\sqrt{\varepsilon}(u^1-\mathcal{U}^1)+\varepsilon(u^2-\mathcal{U}^2)=O(\varepsilon^{\frac32}),
\end{align*}
which implies
\begin{align}\label{B.11}
\|(u^0-\overline{u^0}+\sqrt{\varepsilon}(u^1-\mathcal{U}^1)+\varepsilon(u^2-\mathcal{U}^2))\partial_x\rho_b^0\|_{L^2}\le C\varepsilon^{\frac32}.\tag{B.11}
\end{align}
Similarly, one also deduce that
\begin{align}\label{B.12}
u^0-\overline{u^0}+\sqrt{\varepsilon}(u^1-\mathcal{U}^1)=O(\varepsilon),\quad u^0-\overline{u^0}=O(\sqrt{\varepsilon}).\tag{B.12}
\end{align}
Moreover, since $\overline{v^0}=0$, we obtain that
\begin{align*}
v^0+\sqrt{\varepsilon}(v^1-\mathcal{V}^1)+\varepsilon(v^2-\mathcal{V}^2)-\varepsilon^{\frac32}\mathcal{V}^3=O(\varepsilon^{2}).
\end{align*}
It remains to estimate the last line of $R_1$ and we only show the estimate of $\varepsilon(\widetilde{u_b^2}-u_b^2)$ here. Recall the definition of $\widetilde{u_b^2}$, one has
\begin{align*}
\varepsilon(\widetilde{u_b^2}-u_b^2)=\varepsilon(\chi (y)-1)u_b^2+\varepsilon^{\frac32}\tau_u.
\end{align*}
For the first term of the above equality, we have for any $l>0$
\begin{align*}
\|\varepsilon(\chi(y)-1)u_b^2\|_{L^2}^2&=\varepsilon^2\int_{\mathbb{T}}\int_{1/\sqrt{\varepsilon}}^\infty \left[\left(1-\chi(\sqrt{\varepsilon}\eta)\right)u_b^2(t, x, \eta)\right]^2\;d\eta dx\\
&\le 2\varepsilon^2\int_{\mathbb{T}}\int_{1/\sqrt{\varepsilon}}^\infty (\sqrt{\varepsilon}\eta)^{2l}\cdot(u_b^2(t, x, \eta))^2\;d\eta dx\\
&\le 4\varepsilon^{2+l}\|\eta^{l}u_b^2\|_{L^2}^2\\
&\le C\varepsilon^{2+l}.
\end{align*}
Thus, for $l\ge1$, we obtain that
\begin{align*}
\|\varepsilon(\widetilde{u_b^2}-u_b^2)\|_{L^2}\le C\varepsilon^{\frac32}.
\end{align*}
Then we achieve immediately that
\begin{align}\label{B.13}
\|\varepsilon(\widetilde{u_b^2}-u_b^2)\left(\overline{\partial_x\rho^0}+\partial_x\rho_b^0\right)\|_{L^2}\le C\varepsilon^{\frac32}.\tag{B.13}
\end{align}
This ends the proof of Proposition \ref{remainder}.
\end{proof}

\section*{Appendix C}
In this section, we will present the expression of the vector ${\bf{C}}^0$, ${\bf{C}}^\varepsilon$ and show the proof of Proposition \ref{vector}. First, the vector ${\bf{C}}^0$ is given by
\begin{align*}
{\bf{C}}^0=(\partial_x\rho^a+c^pg^a)u+(\partial_y\rho^a-c^ph^a)v+\kappa\varepsilon(c^p)^2h+(\partial_t+{\bf{u}}^\varepsilon\cdot\nabla+\kappa\varepsilon c^p\partial_y)c^p\cdot\psi,
\end{align*}
and ${\bf{C}}^\varepsilon$ has the following form
\begin{align}\label{C.1}
{\bf{C}}^\varepsilon={\bf{M}}\cdot(u, v, h, g)^T+\psi{\bf{V}},
\tag{C.1}
\end{align}
where ${\bf{M}}$ is a matrix and ${\bf{V}}$ is a vector. Moreover, the matrix ${\bf{M}}$ is stated as follows.
\begin{align*}
{\bf{M}}=\left(
\begin{array}{cccc}
-\rho^\varepsilon\partial_y(v^a-a^pg^a) & \rho^\varepsilon\partial_y(u^a-a^ph^a) & {\bf{M}}_{13} & {\bf{M}}_{14} \\
\rho^\varepsilon\partial_x(v^a-a^pg^a) & -\rho^\varepsilon\partial_x(u^a-a^ph^a) & {\bf{M}}_{23} & {\bf{M}}_{24} \\
\partial_xh^a+b^pg^a & \partial_yh^a-b^ph^a & {\bf{M}}_{33} & {\bf{M}}_{34} \\
\partial_xg^a & \partial_yg^a & {\bf{M}}_{43} & {\bf{M}}_{44} \\
\end{array}
\right),
\end{align*}
with
\begin{align*}
{\bf{M}}_{13}=&-\partial_xh^a+\rho^\varepsilon a^p\partial_xu^a+b^p\left\{ [\rho^\varepsilon(a^p)^2-1]g^\varepsilon+(\rho^\varepsilon \kappa-3\mu)\varepsilon\partial_ya^p+(\rho^\varepsilon\kappa-\mu)\varepsilon a^pb^p\right\}\\
&+\left\{\rho^\varepsilon[\partial_t+(u^\varepsilon+a^ph^\varepsilon)\partial_x+(v^\varepsilon+2a^pg^\varepsilon)\partial_y]
-\mu\varepsilon\partial_x^2-3\mu\varepsilon\partial_y^2\right\}a^p-2(\mu-\rho^\varepsilon\kappa)\varepsilon a^p\partial_yb^p,\\
{\bf{M}}_{14}=&-\partial_yh^a+\rho^\varepsilon a^p\partial_yu^a-\rho^\varepsilon a^p\partial_ya^ph^\varepsilon-b^p\left\{[\rho^\varepsilon(a^p)^2-1]h^\varepsilon-2\mu\varepsilon\partial_xa^p\right\}\\
&+2\mu\varepsilon\partial_{xy}a^p+2(\mu-\rho^\varepsilon\kappa)\varepsilon a^p\partial_xb^p,\\
{\bf{M}}_{23}=&-\partial_xg^a+\rho^\varepsilon a^p\partial_xv^a-\rho^\varepsilon a^p\partial_xa^pg^\varepsilon+2
\mu\varepsilon\partial_{xy}a^p+(\mu-\rho^\varepsilon\kappa)\varepsilon b^p\partial_xa^p,\\
{\bf{M}}_{24}=&-\partial_yg^a+\rho^\varepsilon a^p\partial_yv^a+\left\{\rho^\varepsilon[\partial_t+(u^\varepsilon+2a^ph^\varepsilon)\partial_x+(v^\varepsilon+a^pg^\varepsilon)\partial_y]-3\mu\varepsilon\partial_x^2-\mu\varepsilon\partial_y^2\right\}a^p,\\
{\bf{M}}_{33}=&-\partial_xu^a-(h^\varepsilon\partial_x+2g^\varepsilon\partial_y)a^p-a^pb^pg^a-2\kappa\varepsilon\partial_yb^p,\\
{\bf{M}}_{34}=&-\partial_yu^a+(\partial_ya^p+a^pb^p)h^\varepsilon+2\kappa\varepsilon\partial_xb^p,\\
{\bf{M}}_{43}=&-\partial_xv^a+\partial_xa^pg^\varepsilon,\\
{\bf{M}}_{44}=&-\partial_yv^a-(2h^\varepsilon\partial_x+g^\varepsilon\partial_y)a^p.
\end{align*}
The vector ${\bf{V}}=(V_i)\ (i=1, 2, 3, 4)$ is given by
\begin{align*}
V_1=&\partial_xb^p\left\{[\rho^\varepsilon(a^p)^2-1]h^\varepsilon-2\mu\varepsilon\partial_xa^p\right\}\\
&+\partial_yb^p\left\{[\rho^\varepsilon(a^p)^2-1]g^\varepsilon+(3\rho^\varepsilon\kappa-\mu)\varepsilon\partial_ya^p+(\rho^\varepsilon\kappa-\mu)\varepsilon a^pb^p\right\}\\
&+\left\{\rho^\varepsilon[\partial_t+(u^\varepsilon+a^ph^\varepsilon)\partial_x+(v^\varepsilon+a^pg^\varepsilon)\partial_y]
-\mu\varepsilon\Delta\right\}\partial_ya^p-(\mu-\rho^\varepsilon\kappa)\varepsilon a^p\Delta b^p,\\
V_2=&-\left\{\rho^\varepsilon[\partial_t+(u^\varepsilon+a^ph^\varepsilon)\partial_x+(v^\varepsilon+a^pg^\varepsilon)\partial_y]
-\mu\varepsilon\Delta\right\}\partial_xa^p+(\mu-\rho^\varepsilon\kappa)\varepsilon\partial_xa^p\partial_y b^p,\\
V_3=&-(h^\varepsilon\partial_x+g^\varepsilon\partial_y)\partial_ya^p+\left\{\partial_t+(u^\varepsilon+a^ph^\varepsilon)\partial_x+(v^\varepsilon+a^pg^\varepsilon)\partial_y\right\}g^p,\\
V_4=&(h^\varepsilon\partial_x+g^\varepsilon\partial_y)\partial_xa^p.
\end{align*}
Now it is position to show the proof of Proposition \ref{vector}.
\begin{proof}[Proof of Proposition \ref{vector}]
We only show the estimates of the first line of ${\bf{M}}$. First, from the boundary condition $g^a|_{y=0}=0$ and the divergence free condition, we get by H\"ardy's trick that
\begin{align*}
\|g^a\partial_ya^p\|_{L^\infty}\le \|y\partial_ya^p\|_{L^\infty}\left\|\frac{v^a}y\right\|_{L^\infty}\lesssim\|y\partial_ya^p\|_{L^\infty}\|\partial_yv^a\|_{L^\infty}=O(1),
\end{align*}
which implies
\begin{align}\label{C.2}
\|\rho^\varepsilon\partial_y(v^a-a^pg^a)\|_{L^\infty}=O(1).
\tag{C.2}
\end{align}
Recall \eqref{2.14} and the definitions of $u^p$ and $h^p$, we have
\begin{align*}
u^a-a^ph^a=u^0-\overline{u^0}+u^p-a^p(h^0-\overline{h^0}+h^p)+O(\sqrt{\varepsilon})=u^0-\overline{u^0}+a^p(h^0-\overline{h^0})+O(\sqrt{\varepsilon}).
\end{align*}
From \eqref{B.12}, we find that $u^a-a^ph^a$ is of order $\sqrt{\varepsilon}$ and it is direct to verify
\begin{align}\label{C.3}
\|\rho^\varepsilon\partial_y(u^a-a^ph^a)\|_{L^\infty}=O(1).
\tag{C.3}
\end{align}
Similarly, one also has
\begin{align}\label{C.4}
\partial_yh^a-b^ph^a=O(1).\tag{C.4}
\end{align}
For the third term $b^p[\rho^\varepsilon(a^p)^2-1]g^\varepsilon$, we separate it into two parts by \eqref{3.1}.
\begin{align*}
b^p[\rho^\varepsilon(a^p)^2-1]g^\varepsilon=b^p[\rho^\varepsilon(a^p)^2-1]g^a+\varepsilon^{\frac32}b^p[\rho^\varepsilon(a^p)^2-1]g.
\end{align*}
The first one is of order $O(1)$ by \eqref{C.2}, and the second one is also of order $O(1)$ by the assumption \eqref{3.24}. Thus we obtain
\begin{align}\label{C.5}
b^p[\rho^\varepsilon(a^p)^2-1]g^\varepsilon=O(1).
\tag{C.5}
\end{align}
Similar estimates hold  for $\rho^\varepsilon(v^\varepsilon+2a^pg^\varepsilon)\partial_ya^p$, and the remaining terms of ${\bf{M}}_{13}$ are of order $O(1)$. Therefore, we conclude
\begin{align}\label{C.6}
{\bf{M}}_{13}=O(1).
\tag{C.6}
\end{align}
Finally, for ${\bf{M}}_{14}$, there may be four terms which behaves like $O(\varepsilon^{-\frac12})$.
\begin{align*}
&-\partial_yh^a+\rho^\varepsilon a^p\partial_yu^a-\rho^\varepsilon a^p\partial_ya^ph^\varepsilon-b^p[\rho^\varepsilon(a^p)^2-1]h^\varepsilon\\
=&-(\partial_yh^a-b^ph^a)+\rho^\varepsilon a^p[\partial_yu^a-(\partial_ya^p+a^pb^p)h^a]\\
=&-(\partial_yh^a-b^ph^a)+\rho^\varepsilon a^p[\partial_y(u^a-a^ph^a)+a^p(\partial_yh^a-b^ph^a)],
\end{align*}
which is of order $O(1)$ due to \eqref{C.3} and \eqref{C.4}. Combining the estimates above together, we deduce that
\begin{align}\label{C.7}
{\bf{M}}=O(1).
\tag{C.7}
\end{align}
Again by using H\"ardy's trick, we achieve that
\begin{align}\label{C.8}
\|{\bf{C}}^\varepsilon\|_{L^2}\le C\|{\bf{U}}\|_{L^2}.
\tag{C.8}
\end{align}
By the same procedure, we also get
\begin{align}\label{C.9}
\|\partial_\tau{\bf{C}}^\varepsilon\|_{L^2}\le C\|\partial_\tau{\bf{U}}\|_{L^2}+(1+Q(t))\|{\bf{U}}\|_{L^2},\quad \|y\partial_y{\bf{C}}^\varepsilon\|_{L^2}\le C\|y\partial_y{\bf{U}}\|_{L^2}+(1+Q(t))\|{\bf{U}}\|_{L^2}.
\tag{C.9}
\end{align}
Moreover, the similar estimates hold for ${\bf{C}}^0$.
\begin{align}\label{C.10}
\|{\bf{C}}^0\|_{L^2}\le C\|{\bf{U}}\|_{L^2},\quad \|\mathcal{Z}{\bf{C}}^0\|_{L^2}\le C\|\mathcal{Z}{\bf{U}}\|_{L^2}.
\tag{C.10}
\end{align}
The proof of Proposition \ref{vector} is done.
\end{proof}

\end{document}